\documentclass[11pt]{amsart}
\allowdisplaybreaks[4]
\usepackage[square, comma, sort&compress, numbers]{natbib}
\usepackage{epsfig}
\usepackage{rotating}
\usepackage{mathtools}
\usepackage{bm}
\usepackage{extarrows}
\usepackage[colorlinks,linkcolor=red,anchorcolor=blue,citecolor=blue]{hyperref}
\usepackage{booktabs}
\usepackage[justification=centering]{caption}
\usepackage{threeparttable}
\usepackage{chngpage}
\usepackage{xr}
\usepackage[toc,page,title,titletoc,header]{appendix}
\newcommand{\red }{\color{red}}

\usepackage{amssymb}
\usepackage{anysize}
\usepackage{hyperref}
\usepackage{extarrows}

\numberwithin{equation}{section}
\def\beginn{\begin{eqnarray*}}
\def\endn{\end{eqnarray*}}
\def\beginy{\begin{eqnarray}}
\def\endy{\end{eqnarray}}
\def\begine{\begin{enumerate}}
\def\ende{\end{enumerate}}

\def\be{\begin{equation}}
\def\ee{\end{equation}}
\def\bea{\begin{eqnarray}}
\def\eea{\end{eqnarray}}

\numberwithin{equation}{section}
\theoremstyle{plain}
\newtheorem{thm}{Theorem}[section]
\newtheorem{prop}{Proposition}
\newtheorem{lem}{Lemma}
\newtheorem{coro}{Corollary}
\newtheorem{rmk}{Remark}
{
\theoremstyle{definition}

\newtheorem{assu}{Assumption}
}

\newtheorem{deff}{Definition}

\newcommand{\non}{\nonumber \\}
\newcommand{\bbA}{{\bf A}}

\newcommand{\bbB}{{\bf B}}
\newcommand{\bbC}{{\bf C}}

\newcommand{\bbD}{{\bf D}}
\newcommand{\bbd}{{\bf d}}
\newcommand{\bbe}{{\bf e}}

\newcommand{\bbf}{{\bf f}}
\newcommand{\bbF}{{\bf F}}

\newcommand{\bbG}{{\bf G}}

\newcommand{\bbH}{{\bf H}}

\newcommand{\bbw}{{\bf w}}
\newcommand{\bbI}{{\bf I}}

\newcommand{\bbM}{{\bf M}}

\newcommand{\bbQ}{{\bf Q}}

\newcommand{\bbr}{{\bf r}}

\newcommand{\bbS}{{\bf S}}

\newcommand{\bbT}{{\bf T}}

\newcommand{\bbU}{{\bf U}}
\newcommand{\bbu}{{\bf u}}
\newcommand{\bbV}{{\bf V}}
\newcommand{\bbv}{{\bf v}}

\newcommand{\bbW}{{\bf W}}

\newcommand{\bbX}{{\bf X}}

\newcommand{\bbx}{{\bf x}}

\newcommand{\bbY}{{\bf Y}}
\newcommand{\bby}{{\bf y}}
\newcommand{\bbZ}{{\bf Z}}

\newcommand{\ep}{\ensuremath{\epsilon}}

\newcommand{\wh}{\widehat}
\newcommand{\bfv}{{\mathbf{v}}}

\newcommand{\Cov}{\text{Cov}}

\newcommand{\diag}{\mathop{\sf diag}}

\newcommand{\beq}{\begin{equation}}
\newcommand{\eeq}{\end{equation}}
\newcommand{\beas}{\begin{eqnarray*}}
\newcommand{\eeas}{\end{eqnarray*}}
\newcommand{\bei}{\begin{itemize}}
\newcommand{\eei}{\end{itemize}}
\newcommand{\ben}{\begin{enumerate}}
\newcommand{\een}{\end{enumerate}}

\makeatletter

\newcommand{\Rmnum}[1]{\expandafter\@slowromancap\romannumeral #1@}
\makeatother
\makeatletter 
\@addtoreset{equation}{section}
\makeatother  

\marginsize{35mm}{35mm}{38mm}{40mm}

\begin{document}
\title[]{Limiting Laws for Divergent Spiked Eigenvalues and Largest Non-spiked Eigenvalue of Sample Covariance Matrices}
\author{T. Tony Cai}
\author{Xiao Han}
\author{Guangming Pan}

\thanks{ Tony Cai was supported in part by NSF Grant DMS-1712735 and  NIH Grant R01 GM-123056. Guangming Pan was supported in part by by a MOE Tier 2 grant 2014-T2-2-060 and
by a MOE Tier 1 Grant RG25/14 at the Nanyang Technological University, Singapore.}

\address{The Wharton School, University of Pennsylvania, USA}
\email{tcai@wharton.upenn.edu}

\address{Division of Mathematical Sciences, School of Physical and Mathematical Sciences, Nanyang Technological University, Singapore}
\email{xhan011@e.ntu.edu.sg}

\address{Division of Mathematical Sciences, School of Physical and Mathematical Sciences, Nanyang Technological University, Singapore}
\email{gmpan@ntu.edu.sg}


\date{\today}

\maketitle

\begin{abstract}
We study the asymptotic distributions of the spiked eigenvalues and the largest nonspiked eigenvalue of the sample covariance matrix under a general covariance matrix model with divergent spiked eigenvalues, while the other eigenvalues are bounded but otherwise arbitrary. The limiting normal distribution for the spiked sample eigenvalues is established. It has distinct features that the asymptotic mean relies on not only the population spikes but also the nonspikes and that the asymptotic variance in general depends on the population eigenvectors. In addition, the limiting  Tracy-Widom law for the largest nonspiked sample eigenvalue is obtained.

Estimation of the number of spikes and the convergence of the leading eigenvectors are also considered.
The results hold even when the number of the spikes diverges. As a key technical tool, we develop a Central Limit Theorem for a type of random quadratic forms where the random vectors and random matrices involved are dependent. This result can be of independent interest.


\end{abstract}

\section{Introduction}
\label{sec:intro}

Covariance matrix plays a fundamental role in multivariate analysis and high-dimensional statistics. There has been significant recent interest in studying the properties of the leading eigenvalues and eigenvectors of the sample covariance matrix, especially in the high-dimensional setting. See, for example, \cite{J01, Hoyle04, BS006, P07, Nadler08, J09, M09, Birnbaum12, CMW13, CMW15}.  These problems are not only of interest in their own right they also have close connections to important statistical problems such as  principal component analysis and testing for the covariance structure of high-dimensional data.

Principal component analysis (PCA) is a widely used technique in multivariate analysis for a range of purposes, including dimension reduction, data visualization, clustering, and feature extraction \cite{ander84,hastie2009elements}.
PCA is particularly well suited for the settings where the signal of interest lies in a much lower dimensional subspace and it has been applied in a broad range of fields such as genomics,  image recognition, data compression, and financial econometrics. For example,  widely used factor models in financial econometrics typically assume that a small number of unknown common factors  drive the asset returns \cite{Chamberlain83}. In PCA, the leading eigenvalues and eigenvectors of the population covariance matrix need to be estimated from data and are conventionally estimated by their empirical counterparts. It is thus important to understand the spectral properties of the sample covariance matrix.

\subsection{The Problem}

To be concrete, consider the data matrix $\bbY=\bold\Gamma\bbX$ where $\bbX=(\bbx_1,\cdots,\bbx_n)$ is a $(p+l)\times n$ random matrix whose entries are independent with zero mean and unit variance and $\bold\Gamma$ is a $p\times (p+l)$ deterministic matrix with $l/p\rightarrow 0$.
 Let $\bold\Sigma=\bold\Gamma\bold\Gamma^{\intercal}$ be the population covariance matrix.
The sample covariance matrix is defined as
  \begin{equation}\label{a2}
  \bbS_n=\frac{1}{n} \bbY \bbY^{\intercal} =\frac{1}{n}\bold\Gamma\bbX\bbX^{\intercal}\bold\Gamma^{\intercal}.
   \end{equation}
Denote the singular value decomposition (SVD) of $\bold\Gamma$ by
 \begin{equation}\label{a3}
\bold\Gamma =  \bbV\Lambda^{1\over 2}\bbU,
 \end{equation}
where $\bbV$ and $\bbU$ are $p\times p$ and $p\times (p+l)$ orthogonal matrices respectively ($\bbV\bbV^{\intercal}=\bbU\bbU^{\intercal}=\bbI$), and $\bold\Lambda$ is a diagonal matrix consisting in descending order of the eigenvalues  $\mu_1\geq \cdots \geq \mu_p$ of $\bold\Sigma$.

In statistical applications such as PCA, one is most interested in the setting where there is a clear separation between a few leading eigenvalues and the rest. In this case, the leading principal components account for a large proportion of the total variability of the data. We consider in the present paper the setting where there are $K$ spiked eigenvalues that are separated from the rest. More specifically, we assume that $\mu_1\ge \cdots \ge \mu_K$ tend to infinity, while the other eigenvalues  $\mu_{K+1}\geq \cdots \geq\mu_{p}$ are bounded but otherwise arbitrary.  Write
\begin{equation}\label{a7}\bold\Lambda=\left(
 \begin{matrix}
   \bold\Lambda_S & 0\\
  0 &\bold\Lambda_P
  \end{matrix}
  \right),
   \end{equation}
where $\bold\Lambda_S=\diag(\mu_1,...,\mu_K)$ and $\bold\Lambda_P=\diag(\mu_{K+1},...,\mu_{p})$.

 A typical example of (\ref{a7}) is the factor model
\begin{equation}
\bbY=\bold\Lambda\bbF+\bbT\bbZ= (\begin{matrix}\bold\Lambda &\bbT \end{matrix}) \left(\begin{matrix}\bbF\\ \bbZ \end{matrix}\right)
\end{equation}
where $\bold\Lambda$ is  $p\times K$-dimensional factor loading, $\bbF$ is the corresponding $K\times n$  factor, $\bbT$ is $p\times p$ matrix and $\bbZ$ is the idiosyncratic noise matrix. A common assumption is that the singular values of the factor part $\bold\Lambda\bbF$ are significantly larger than those of the noise part (otherwise the signals are overwhelmed by noise). Indeed, \cite{O09} considered the weak factor model to test the number of factors, where the leading eigenvalues contributed by the factor part are of order $p^{\theta}$ for some $\theta\in(0,1)$. \cite{BS002} and \cite{F13} assume that the leading eigenvalues of the pervasive factor model are of order $p$. Here
$\bold\Gamma= (\begin{matrix}\bold\Lambda &\bbT \end{matrix})$ is not a square matrix, and thus it is necessary to consider the setting where $\bold\Gamma$ is rectangular.

A second example is the covariance matrix $\bold\Sigma$ used in the intraclass correlation model, where the covariance matrix is of the form
  $$\bold\Sigma=(1-\rho)\bbI+\rho\bbe\bbe^{\intercal}.$$
Here $\bbI$ is the identity matrix, $\bbe=(1,1,...,1)^{\intercal}$ and $0<\rho<1$. It is easy to see that the leading eigenvalue of $\bold\Sigma$ is $p\rho+(1-\rho)$, while the other eigenvalues are equal to $(1-\rho)$. 

We study in the present paper the asymptotic distributions of the leading eigenvalues and the largest nonspiked eigenvalue of the sample covariance matrix $\bbS_n$, under the general spiked covariance matrix model given in (\ref{a3}) and (\ref{a7}) with divergent spiked eigenvalues $\mu_1\geq \cdots \geq \mu_K$.  In many statistical applications, determining the number of principal components is an important problem. We also consider estimation of the number of spikes as well as the convergence of the leading eigenvectors.

The model defined through (\ref{a3}) and (\ref{a7}) belongs to the class of spiked covariance matrix models. Johnstone \cite{J01} was the first to introduce a special spiked covariance matrix model, where the population covariance matrix is diagonal and is of the form
\begin{eqnarray}\label{0227.2}
 \bold\Sigma=\diag(\mu_1^2,...,\mu_K^2,1,...,1)
 \end{eqnarray}
with
$\mu_1 > \mu_2 \cdots \ge \mu_K> 1$.
\cite{J01} established the limiting Tracy-Widom distribution for the maximum eigenvalue of the real Wishart matrices when $p$ and $n$ are comparable. The spiked covariance matrix model (\ref{0227.2}) in \cite{J01}  has been extended in various directions. So far the focus has mostly been on the settings of bounded spiked eigenvalues with all the nonspiked eigenvalues being equal to 1. See more discussion in Section \ref{background.sec}.

%


\subsection{Our contributions}

In this paper, we first establish the limiting normal distribution for the spiked eigenvalues of the sample covariance matrix $\bbS_n$. The limiting distribution has a distinct feature. Unlike in the more conventional settings,  the asymptotic variance in general depends on the population eigenvectors. More precisely, the variance of a spiked sample eigenvalue depends on the right singular vector matrix  $\bbU$ defined in the SVD (\ref{a3}) (but not the left singular vector matrix $\bbV$).
The limiting distribution of the spiked sample eigenvalues also precisely characterizes the dependence on the corresponding population spiked eigenvalues as well as the nonspiked ones.  New technical tools are needed to establish the result. In particular, we develop a Central Limit Theorem (CLT) for a type of random quadratic forms where the random vectors and random matrices involved are dependent. This result can be of independent interest.
In addition, we establish the limiting  Tracy-Widom law for the largest nonspiked eigenvalue of $\bbS_n$.

The limiting distributions for the spiked eigenvalues and the largest nonspiked eigenvalue have important applications. In particular,  based on our theoretical results, we propose an algorithm for estimating the number of the spikes, which is of interest in many statistical applications.
 We also consider the properties of  the sample eigenvectors corresponding to the spiked eigenvalues and show that they are consistent estimators of the population eigenvectors in terms of the $L_2$ norm.   An important improvement of our paper over many known results in the literature is  that our results hold even when the number of the spikes diverges as $n, p\rightarrow \infty$, and we allow the nonspiked eigenvalues  to be unequal.

\subsection{Background and related work}
\label{background.sec}


Since the seminal work of Johnstone  \cite{J01},  the special spiked covariance matrix model (\ref{0227.2}) has been studied much further and the model has been extended in various directions. See, for example \cite{BS006,BBK,BY08,B14,J09,M09,P07,S13,WF16,CMW13, CMW15}. We discuss briefly here some of these results. This review is by no means exhaustive.

Paul \cite{P07} showed that if $p/n\to \gamma\in (0,1)$ as $n\to\infty$, and the largest eigenvalue $\mu_1$ of $\bold\Sigma$ satisfies $\mu_1 \leq (1+\sqrt{\gamma})$, then the leading sample principal eigenvector $\wh\bfv_1$ is asymptotically almost surely orthogonal to the leading population eigenvector $\bfv_1$, i.e.,  $|\bfv_1'\wh\bfv_1|\to 0$ almost surely.
Thus, in this case, $\wh\bfv_1$ is not useful at all as an estimate of $\bfv_1$.
Even when $\mu_1 > (1+\sqrt{\gamma})$, the angle between $\bfv_1$ and $\wh\bfv_1$ still does not converge to zero unless $\mu_1\to\infty$.

Baik and Silverstein \cite{BS006} considered a case where the covariance matrix
 \begin{eqnarray}\label{1005.1h}
 \bold\Sigma=\bbV\left(
 \begin{matrix}
   \Lambda_S & 0\\
  0 &\bbI
  \end{matrix}
  \right)\bbV^{\intercal}
  \label{a1}
\end{eqnarray}
with $\Lambda_S$ being  a diagonal matrix of fixed rank  and $\bbV$  a unitary matrix. It is shown that the spiked eigenvalues tend to some limits in probability, assuming that the spectral norm of $\Lambda_S$ is bounded and $\lim_{n\rightarrow \infty}\frac{p}{n}=\gamma \in (0,\infty)$. Bai and Yao \cite{BY08} further showed that the spiked eigenvalues converge in distribution to Gaussian distribution or the eigenvalues of a finite dimensional matrix with i.i.d. Gaussian entries.  Baik, et al. \cite{BBK} investigated the asymptotic behavior of the largest  eigenvalue when the entries of $\bbX$ follow the standard complex Gaussian distribution and observed a phase transition phenomenon that the asymptotic distribution depends on the scale of the spiked population eigenvalues.  Recently, Bloemendal et al. \cite{B14} obtained the precise large deviation of the spiked eigenvalues and non-spiked eigenvalues
under a more general model than (\ref{a1}).
We should note that the above results only consider the the case of bounded spiked eigenvalues with the nonspiked eigenvalues all being equal to 1.

Jung and Marron \cite{M09}  and Shen et al. \cite{S13}  considered the model
\begin{eqnarray}\label{shen}
   \bbY=\bbV\Lambda^{1\over 2}\bbX,
\end{eqnarray}
where the entries of $\bbX$ are i.i.d. standard normal random variables, and $\Lambda=\diag(\mu_1,...,\mu_K,\mu_{K+1},\cdots,\mu_p)$ is the diagonal matrix consisting of the population eigenvalues, and $\bbV$ is an orthogonal matrix. \cite{M09}  and \cite{S13} showed the almost sure convergence of the spiked eigenvalues when the spiked population eigenvalues satisfy that $p/(\mu_jn),j=1,\cdots,K$ tend to positive constants or zero and $\mu_{K+1},\cdots,\mu_p$ are approximately equal to one. The almost sure convergence of the eigenvectors associated with the spikes is also investigated.

Wang and Fan  \cite{WF16} further developed the asymptotic distribution of the largest sample eigenvalues of the model (\ref{shen}) under a more general setting, which allows $\mu_{K+1},...,\mu_p$ to be any bounded number and the entries of $\bbX$ to be i.i.d. subGaussian random variables. The asymptotic behaviors of the corresponding eigenvectors are also discussed in \cite{WF16}. Here we would like to point out that \cite{WF16} did not provide the limits in probability of spikes for general $\mu_{K+1},...,\mu_p$ when $p/(\mu_jn),j=1,\cdots,K$, tend to positive constants. To the best of our knowledge, the asymptotic behavior of the spiked eigenvalues for general $\mu_{K+1},...,\mu_p$ when $p/(\mu_jn), \ j=1,\cdots,K$, converge to positive constants is still open.

Note that \cite{M09}, \cite{S13} and \cite{WF16} swapped the roles of the sample size $n$ and the dimension $p$ so that they essentially studied the matrix $\bbX^{\intercal}\Lambda\bbX$. This is equivalent to assuming that the population covariance matrix is  diagonal. Indeed, as will be seen later, in general the asymptotic variance of the spiked eigenvalues depends on the population eigenvectors.
This phenomenon does not occur under the previously studied model.

 \subsection{Organization of the paper}

The rest of the paper is organized as follows. Section \ref{sec2h}  establishes the limiting normal distribution for the spiked eigenvalues and the limiting Tracy-Widom distribution for the largest nonspiked eigenvalue of the sample covariance matrix $\bbS_n$.
An algorithm for identifying the number of spikes is developed in Section \ref{sec5h}. Section \ref{Eigenvectors.sec} considers the properties of the principal components and shows that
the sample eigenvectors corresponding to the spiked eigenvalues are consistent estimators of the population eigenvectors in terms of the $L_2$ norm. Most of the results developed for $\bbS_n$ also hold for the centralized sample covariance matrices and this is discussed in Section \ref{Centralized.sec}.
 Section \ref{sec6h} investigates the numerical performance through  simulations and an application of a factor model.  The proof of one of the main results  
is given in Section \ref{proofs.sec} and the proof of the other results is provided in the supplementary material \cite{CHP17}.

\section{Asymptotics for Spiked Eigenvalues and Largest Nonspiked Eigenvalue of $\bbS_n$}\label{sec2h}
\label{CLT-Eigenvalues.sec}

We investigate in this section the limiting laws for the leading eigenvalues and the largest nonspiked eigenvalue of the sample covariance matrix $\bbS_n$ under the general spiked covariance matrix model (\ref{a3}) and (\ref{a7}) with divergent spiked eigenvalues $\mu_1\geq \cdots \geq \mu_K$, while the other eigenvalues are bounded but otherwise arbitrary.  We begin with the notation that will be used throughout the rest of the paper.

For two sequences of positive numbers $a_n$ and $b_n$, we write $a_n\gtrsim b_n$ when $a_n\geq cb_n$ for some absolute constant $c>0$, and $a_n\lesssim b_n$ when $b_n\gtrsim a_n$. We write $a_n \sim b_n$ when both $a_n\gtrsim b_n$ and $a_n\lesssim b_n$ hold. Moreover, we write $a_n \ll b_n$ when $a_n/b_n\rightarrow0$. For a sequence of random variables $A_n$, if $A_n$ converges to $b$ in probability, then we write $A_n \stackrel{i.p.}\rightarrow b$.
We say an event $\mathcal{A}_n$ holds with high probability if $\mathbb{P}(\mathcal{A}_n)\ge 1-O(n^{-l})$ for some constant $l>0$. Denote the $j$-th largest eigenvalue of a matrix $\bbM$ by $\lambda_j(\bbM)$ and the largest singular value by $\|\bbM\|$. Set $\|\bbM\|_F=\sqrt{tr(\bbM\bbM^{\intercal})}$.    For simplicity, denote by $\lambda_1\geq\lambda_2\geq...\geq\lambda_K\geq\cdots\geq \lambda_p$ the ordered eigenvalues of the sample covariance matrix $\bbS_n$, and denote by $\mu_1\geq\mu_2\geq...\geq\mu_K\geq\cdots\geq \mu_p$ the ordered eigenvalues of the population covariance matrix $ \bold\Sigma$. Throughout this paper $c$ and $C$ are constants that may vary from place to place.

\medskip
To investigate the sample covariance matrix $\bbS_n$ in (\ref{a2}) with the population covariance matrix $\bold\Sigma$ specified in (\ref{a3}) and (\ref{a7}) we make the following assumptions.

\begin{assu}\label{0829-1a}
$\{\bbx_{j}=(\bbx_{1j},\cdots,\bbx_{p+l,j})^{\intercal}$, $j=1,...,n\}$ are i.i.d.  random vectors.  $\{\bbx_{ij}$: $i=1,...,p+l$, $j=1,...,n\}$ are independent random variables such that $\mathbb{E}\bbx_{ij}=0$, $\mathbb{E}|\bbx_{ij}|^2=1$, $\mathbb{E}|\bbx_{ij}|^4=\gamma_{4i}$ and $\sup_i\gamma_{4i}\le C$.
\end{assu}

\begin{assu}\label{0829-1}
$p\gtrsim n$ and
the $K$ largest population eigenvalues $\mu_i$ are such that $d_i\equiv\frac{p}{n\mu_i}\rightarrow 0$, $i=1,2,...,K$. And for $i=K+1,...,p$, $\mu_i$ are bounded by $C$. Moreover, $\frac{K}{n^{1/6}}\rightarrow 0$ and $K^2d_K\rightarrow 0$.
 \end{assu}

\noindent{ Assumption $2'$}.
 $\frac{p}{n}\rightarrow 0$, $\mu_i\gg 1$, $i=1,...,K$ and  $K\ll \min\{p, n^{1/6}\}$.

\medskip
Note that we do not assume that $p$ and $n$ are of the same order. The following theorems hold either under Assumption \ref{0829-1} or Assumption  {\red $2'$}  except Theorem \ref{0606-2}. We only give the proofs under Assumption \ref{0829-1}. The proofs under Assumption  {\red $2'$} are similar and thus we omit them.
\begin{assu}\label{0829-1c}
 There exists a positive constant $c$ not depending on $n$ such that $\frac{\mu_{i-1}}{\mu_i}\ge c>1$, $i=1,2,...,K$.
 \end{assu}
Assumption \ref{0829-1c} implies that the spiked eigenvalues are well-separated. It also implies that $\lambda_1>\lambda_2>...>\lambda_K$ with probability tending to 1 by Theorem \ref{0318-1} below.



\subsection{Asymptotic behavior of the spiked sample eigenvalues}

 Our first result gives the limits in probability for the spiked eigenvalues of $\bbS_n$,  $\lambda_1\geq...\geq\lambda_K$.

\begin{thm}\label{0318-1}
Suppose that Assumption \ref{0829-1a} holds. Moreover, either Assumption \ref{0829-1} or Assumption {\red $2'$}  holds.  Then
\begin{eqnarray}\label{0318.1}
\frac{\lambda_i}{\mu_i}-1=O_p(d_i+\frac{K^4}{n}+\frac{1}{\mu_i}),
\end{eqnarray}
uniformly for all $i=1,...,K$.
\end{thm}

\begin{rmk}\label{rk1}
As mentioned in the introduction, PCA is an important statistical tool for analyzing high-dimensional data. Several recent results on high-dimensional PCA are quite relevant to Theorem \ref{0318-1}. Recently \cite{B17} considered AIC and BIC criteria for selecting the number of significant components in high dimensional PCA when $p$ and $n$ are comparable.
Comparing to the paper \cite{B17}, Theorem \ref{0318-1} here covers Lemma 2.2(i) of \cite{B17} and we allow $K$ to tend to infinity.
Their assumption $\mu_{K+1}=\cdots=\mu_p=1$ is also relaxed to bounded eigenvalues here. In addition, checking the proof of Theorems 3.3 and 3.4 of \cite{B17}, we find that for general population covariance matrices, their criteria $\tilde A_j$ and $\tilde B_j$ for estimating the number of spikes may not work since it highly depends on the assumption $\mu_{K+1} =\cdots=\mu_p=1$, as demonstrated in Table \ref{tab3} given in Section \ref{sec6h}.
In addition,  Theorem \ref{0318-1} also covers part of Theorem 3.1 in \cite{S13} where it assumes normality for the data.
\end{rmk}

Note that $\frac{\lambda_i}{\mu_i}\stackrel{i.p.}\rightarrow 1$ does not imply that $\lambda_i$ is a good estimator of $\mu_i$  due to the fact that $\mu_i$ tends to infinity. Moreover, Theorem \ref{0318-1} does not precisely characterize how the nonspiked population eigenvalues affect the spiked sample eigenvalues. To see this, it is helpful to make a comparison with the conventional setting studied in  \cite{BS006}.  Consider the model (\ref{1005.1h}) and recall the assumptions of \cite{BS006} that $1+\sqrt{\gamma}<\mu_i=O(1)$ and $\gamma=\lim_{n\rightarrow \infty}\frac{p}{n}\in (0,\infty)$. It was shown in \cite{BS006} that
\begin{equation}\label{ab1}
\lambda_i\stackrel{a.s.}\rightarrow \mu_i+\frac{\gamma\mu_i}{\mu_i-1}.
\end{equation}
So  the effect of the population eigenvalues on the corresponding sample eigenvalues can be precisely characterized in the setting considered in \cite{BS006}.  On the other hand,  one cannot see the 
effect of the nonspiked population eigenvalues on the spiked sample eigenvalues from \eqref{ab1}. Note that if there are no spikes, then all the sample eigenvalues are not bigger than $(1+\sqrt{\gamma})^2 $ with probability one. When there are sufficiently large spikes, the sample spikes are pulled outside of the boundary $(1+\sqrt{\gamma})^2 $ due to the population spikes with probability one. Moreover, (\ref{ab1}) precisely quantifies the effect of the population spike. In view of this, one would ask whether there is a similar phenomenon for unbounded spikes. Indeed, it is natural to   imagine that for the case $\mu_i\rightarrow \infty$, the term $\frac{\gamma\mu_i}{\mu_i-1}$ will not disappear and thus one needs to subtract it from $\lambda_i$ in order to obtain the CLT.
Surprisingly, a more precise limit of $\lambda_i$ turns out to be determined not only by $\mu_i$ but also the nonspiked eigenvalues. This is very different from (\ref{ab1}) and can be seen clearly from  (\ref{h1023.2}) below.

We now characterize how the population eigenvalues including spiked eigenvalues and non-spiked eigenvalues affect the sample spiked eigenvalues.
 To this end,
corresponding to (\ref{a7}), partition $\bbU$ as  $\bbU= \left(\begin{matrix} \bbU_1 \\ \bbU_2\end{matrix}\right)$,  where $\bbU_1$ is the $K\times (p+l)$ submatrix of $\bbU$, and define
  \begin{equation}\label{a8}
  \bold\Sigma_1=\bbU_2^{\intercal}\bold\Lambda_P\bbU_2.
    \end{equation}
  For any distribution function $H$, its Stieltjes transform is defined by
$$m_H(z)=\int \frac{1}{\lambda-z}dH(\lambda), \quad \text{for all} \ z\in \mathbb{C}^+.$$ For any $\theta\neq 0$, let $\tilde m_{\theta}(z)$  be the unique solution to the following equation
\begin{eqnarray}\label{0331.2h}
\tilde m_{\theta}(z)=- \left(z-\frac{1}{n}tr(\bbI+\tilde m_{\theta}(z)\frac{\bold\Sigma_1}{\theta})^{-1}\frac{\bold\Sigma_1}{\theta}\right)^{-1}, \ \ z\in \mathbb{C}^+,
\end{eqnarray}
where $\mathbb{C}^+$ denotes the complex upper half plane and $\bold\Sigma_1$ is defined in (\ref{a8}). Here $\tilde m_{\theta}(z)$ is the limit of the Stieltjes transform of the empirical distribution function of the random matrix $\frac{1}{n\theta}\bbX^{\intercal}\bold{\Sigma}_1\bbX$, associated with the nonspiked population eigenvalues.  Indeed, as will be seen, for $\theta\gg\frac{p}{n}$,
$$\tilde m_{\theta}(z)-\frac{1}{n}\mathbb{E}tr(z\bbI-\frac{1}{n\theta}\bbX^{\intercal}\bold{\Sigma}_1\bbX)^{-1}\rightarrow 0$$
for $z\in \mathbb{C}^+$ by a slight modification of the proof of Appendix \ref{calculatemean}. One can also refer to (1.6) of  \cite{BPWZ2014b} or (6.12)-(6.15) of \cite{BS06} for (\ref{0331.2h}). One may see below that $\tilde m_{\theta}(z)$ describes the collective contribution of the nonspiked eigenvalues of $\bold\Sigma$ to the spiked sample eigenvalues.

 By (\ref{0331.2h}), we set $\theta_i$ to be the solution to
\begin{eqnarray}\label{1101.6}
\tilde m_{\theta_i}(1)+\frac{\theta_i}{\mu_i}=0,
\end{eqnarray}
  where $\tilde m_{\theta_i}(1)=\lim\limits_{z\in \mathbb{C}^+\rightarrow 1}\tilde m_{\theta_i}(z)$.
 It turns out that $\theta_i$ instead of $\mu_i$ is the more precise limit of the spiked sample eigenvalues $\lambda_i$. 
From (\ref{1101.6}) one can see that $\theta_i$ depends on $\mu_i$ as well as the nonspiked part $\bold\Sigma_1$.
Indeed, this point can be seen more clearly from (\ref{h1023.2}) below. To the best of our knowledge, such a dependence of $\theta_i$ on $\mu_i$ as well as the nonspiked part $\bold\Sigma_1$ has never been appeared in the literature before.

\begin{assu}\label{100817a}
 Assume that the following limits exist:
 $$\sigma_i=\lim_{p\rightarrow \infty}\sqrt{\sum\limits_{j=1}^{p+l}(\gamma_{4j}-3)u_{ij}^4+2},\quad \sigma_{ij}=\lim_{p\rightarrow \infty}\sum\limits_{s=1}^{p+l}(\gamma_{4s}-3)u_{is}^2u_{js}^2.
 $$
 \end{assu}

We are ready to state the asymptotic distribution of the spiked eigenvalues of $\bbS_n$.  Let $\bbu_i^{\intercal}$ be the $i$-th row of $\bbU$ with $u_{ij}$ being the $(i,j)$-th entry of $\bbU$.

\begin{thm}\label{0408-2}
Suppose that Assumptions \ref{0829-1a}, \ref{0829-1c}, and \ref{100817a} hold. Moreover, either Assumption \ref{0829-1} or Assumption {\red $2'$}  hold. 
 Then  for all $i=1,2,..., K$,
\begin{eqnarray}\label{h0408.4}
\sqrt n\frac{\lambda_i-\theta_i}{\theta_i}\stackrel{D}{\longrightarrow} N\left(0, \sigma_i^2\right).
\end{eqnarray}
Moreover, for any fixed $r\ge 2$
\begin{eqnarray}\label{0831.3}
\left(\sqrt{n}\frac{\lambda_1-\theta_1}{\theta_1},..., \sqrt{n}\frac{\lambda_r-\theta_r}{\theta_r}\right)\stackrel{D}{\longrightarrow} N\left(0,\bold\Sigma^{(r)}\right),
\end{eqnarray}
 where $\bold\Sigma^{(r)}=(\bold\Sigma^{(r)}_{ij})$ with
\begin{eqnarray*}\bold\Sigma^{(r)}_{ij}=
\begin{cases}
\sigma_i^2, & i=j \cr
\sigma_{ij}, & i\neq j,
\end{cases}
\end{eqnarray*}
  \end{thm}

 It follows from (\ref{0331.2h}) and (\ref{1101.6}) that $\tilde m_{\theta_i}(1)\rightarrow -1$. Therefore $\frac{\theta_i}{\mu_i}\rightarrow 1$. However, we can not replace $\theta_i$ by $\mu_i$ in (\ref{0831.3}) directly because the convergence rate of $\frac{\theta_i}{\mu_i}$ to 1 is unknown. Indeed, by (\ref{0331.2h}), we have
 \begin{eqnarray}\label{h1023.1}
\theta=-\frac{\theta}{\tilde m_{\theta}(1)}+\frac{p-K}{n}\int\frac{tdF_{\bold\Lambda_P}(t)}{1+t\tilde m_{\theta}(1)\theta^{-1}},
\end{eqnarray}
where $F_{\bold\Lambda_P}$ is the empirical spectral distribution of $\bold\Lambda_P$. Here for any $n\times n$ symmetric matrix $\bbA$ with real eigenvalues, the empirical spectral distribution (ESD) of $\bbA$ is defined as
$$F_{\bbA}(x)=\frac{1}{n}\sum_{i=1}^n I_{\{\lambda_i(\bbA)\le x\}}.$$
Together with (\ref{1101.6}), we conclude that
 \begin{eqnarray}\label{h1023.2}
\theta_i=\mu_i(1+\frac{p-K}{n}\int\frac{tdF_{\bold\Lambda_P}(t)}{\mu_i-t}).
\end{eqnarray}
By the Taylor's expansion we have
\begin{eqnarray}\label{1129.2}
\frac{\theta_i}{\mu_i}=1+ ff_i+O(\frac{p}{n\mu_i^{2}}),
\end{eqnarray}
where
$$f=\frac{1}{p-K}\sum_{j=K+1}^{p}\mu_j  \quad\mbox{and} \quad  f_{i}=\frac{p-K}{n\mu_i}.$$
In particular, for the special case $\mu_{K+1}=...=\mu_{p}=1$,  (\ref{h1023.2}) yields that
 \begin{equation}
 \label{a271017}
 \theta_i=\mu_i(1+\frac{p-K}{n(\mu_i-1)}).
 \end{equation}
It is interesting to note that, although here the spiked eigenvalues $\mu_1,\cdots,\mu_K$ are divergent, this is consistent with the right hand side of (\ref{ab1}), which is for the conventional setting of bounded spiked eigenvalues.
It then follows from (\ref{1129.2}) that
 \begin{equation}
 \label{a4}
  \sqrt n\Big(\frac{\lambda_i}{\mu_i}-1-ff_i+O(\frac{p}{n\mu_i^2})\Big)\stackrel{D}{\rightarrow} N\left(0, \sigma_i^2\right).
 \end{equation}

\begin{rmk}
We note that Assumption \ref{100817a} is not needed if we consider the individual asymptotic distribution of the spiked sample eigenvalues. To see this, it suffices to normalize $(\lambda_i-\theta_i)/\theta_i$ by  $\sigma_i=\sqrt{\sum\limits_{j=1}^{p+l}(\gamma_{4j}-3)u_{ij}^4+2}$. Moreover,  the joint distribution of $\frac{\lambda_i-\theta_i}{\sigma_i\theta_i}$, $i=1,...,r$ tends to the normal distribution with the covariance matrix being the correlation matrix corresponding to $\bold\Sigma^{(r)}$.
 \end{rmk}

\begin{rmk}\label{1127-1}
It is helpful to compare the above theorem with Theorem 3.1 of \cite{WF16}. Besides the difference between the models in (\ref{a3}) and (\ref{shen}), one of the key differences is that $\sigma_i^2$ in (\ref{a4})  depends on the entries of the eigenvector matrix $\bbU$ while the variance in Theorem 3.1 of \cite{WF16} does not depend on it. This is due to the fact that \cite{WF16} assumes that $\bbU=\bbI$. Secondly, Theorem 3.1 of \cite{WF16} involves $O_p(\frac{\sqrt{p}}{\sqrt{n}\mu_i})$ which reduces to $O(\frac{p}{n\mu_i^2})$ (essentially $O(\frac{1}{\mu_i})$) in (\ref{a4}) by dropping the additional $\frac{\sqrt{p}}{\sqrt{n}}$. Thirdly we also allow $K$ to diverge. Fourthly \cite{WF16} assumes $x_{ij}$ to be subGaussian random variables while Theorem \ref{0408-2} holds under the bounded fourth moment assumption.
\end{rmk}


In view of (\ref{1129.2}) we need to estimate $f$ and $f_i$ in practice. A natural estimator of $f_i$ is $\frac{p-K}{n\lambda_i}$ by Theorem \ref{0318-1}. For $f$, one can use
\begin{equation}\label{a5}
\hat f=\frac{\frac{1}{n}tr(\Gamma\bbX\bbX^{\intercal}\Gamma)-\sum_{i=1}^K\lambda_i}{p-K-pK/n}
\end{equation}
which was proposed in \cite{WF16}. When $p\sim n$, by Proposition \ref{0103-1} in the next section, $K$ can be estimated accurately.

Moreover, Theorem \ref{0408-2} can be extended to the case when the population eigenvalues $\mu_i$ have multiplicity more than one.\begin{assu}\label{0829-1h}
Suppose that $K\ll n^{1/6}$,
 $\alpha_{\mathcal L}=\mu_{K}=...=\mu_{K-n_{\mathcal L}}<\alpha_{\mathcal{L}-1}=\mu_{K-n_{\mathcal L}+1}...<\alpha_{1}=\mu_{n_1}=...=\mu_1$, and there exists a constant $c$ such that $\frac{\alpha_{i-1}}{\alpha_i}\ge c>1$, $i=1,2,...,\mathcal L$. Moreover, $n_1$,..., $n_{\mathcal L}$ are finite.
 \end{assu}

 \begin{assu}\label{100817b}
 Suppose that the following limits exist
 $$G(r_i,k_1,k_2,l_1,l_2)=\lim\limits_{n\rightarrow \infty}n^2\times \Cov(\bbu_{r_i+k_1}^{\intercal}\bbx_1\bbu_{r_i+l_1}^{\intercal}\bbx_1,\bbu_{r_i+k_2}^{\intercal}\bbx_1\bbu_{r_i+l_2}^{\intercal}\bbx_1).$$
  \end{assu}

If either the fourth moments $\gamma_{4s}=3$, $s=1,...,p+l$ or the entries of the population eigenvectors satisfy $\min\limits_{r\in\{k_1,k_2,l_1,l_2\}}\max\limits_{j}|u_{r_i+r,j}|=o(1)$, then
$$g(r_i,k_1,k_2,l_1,l_2)=\begin{cases}1 &\  \text{if}\ k_1=k_2 \ \text{and}\ l_1=l_2 \ \text{or}\ k_1=l_2 \ \text{and}\ l_1=k_2 \\
0& \text{otherwise}.
\end{cases}$$

 Then we have the following result.
 \begin{thm}\label{0831-1}
Suppose that Assumptions \ref{0829-1a},   \ref{0829-1h}  and  \ref{100817b} hold. Moreover, either Assumption \ref{0829-1} or Assumption {\red $2'$}  holds. Let
$$\theta_i=\alpha_i(1+\frac{p-K}{n}\int\frac{tdF_{\bold\Lambda_P}(t)}{\alpha_i-t}).$$
Let $r_i=\sum_{j=0}^{i-1}n_j$, for $i=1,2,..., \mathcal L$. Then
\begin{eqnarray}\label{0408.4}
 \frac{\sqrt n}{\theta_i}(\lambda_{r_i+1}-\theta_i, \lambda_{r_i+2}-\theta_i,..., \lambda_{r_i+n_i}-\theta_i)\stackrel{D}{\rightarrow}  \mathcal{R}_i,
\end{eqnarray}
  where $\mathcal{R}_i$ are the eigenvalues of $n_i\times n_i$ Gaussian matrix $\mathfrak{S}_i$ with $\mathbb{E}\mathfrak{S}_i=0$ and the covariance of the $(\mathfrak{S}_i)_{k_1, l_1}$ and $(\mathfrak{S}_i)_{k_2, l_2}$ being $G(r_i,k_1,k_2,l_1,l_2)$.
\end{thm}

The proof of Theorem \ref{0408-2} requires new technical tools.  The following CLT for a type of random quadratic forms, where the random vectors and random matrices involved are dependent, plays a key role in the proof. This result can be of independent interest.

\begin{thm}\label{1101-1}
Suppose that Assumption \ref{0829-1a} holds and the spectral norm of $\Sigma_1$ is bounded. 
In addition, suppose that there exist orthogonal unit vectors  $\bbw_1$ and $\bbw_2$ such that $\bbw_1^{\intercal}\bbU_2^{\intercal}=\bbw_2^{\intercal}\bbU_2^{\intercal}=0$ and $\bbw_1^{\intercal}\bbw_2=0$.   If $\frac{\theta}{\frac{p+l}{n}}\rightarrow \infty$ and $\theta\rightarrow \infty$, then
\begin{eqnarray}\label{1101.1}
\frac{\sqrt n}{\tilde \sigma_1}\Big(\bbw_1^{\intercal}\bbX(n\bbI-\bbX^{\intercal}\frac{\bold\Sigma_1}{\theta}\bbX)^{-1}\bbX^{\intercal}\bbw_1+\tilde m_{\theta}(1)\Big)\stackrel{D}{\rightarrow} N\left(0, 1\right)
\end{eqnarray}
and
\begin{eqnarray}\label{1101.2}
\frac{\sqrt n}{\tilde\sigma_{12}}\bbw_1^{\intercal}\bbX(n\bbI-\bbX^{\intercal}\frac{\bold\Sigma_1}{\theta}\bbX)^{-1}\bbX^{\intercal}\bbw_2\stackrel{D}{\rightarrow} N\left(0, 1\right)
\end{eqnarray}
where  $\tilde \sigma_1^2=\sum_{j=1}^{p+l}[(\gamma_{4j}-3)w_{1j}^4]+2$, $\tilde\sigma_{12}^2= \sum\limits_{s=1}^{p+l}[(\gamma_{4s}-3)w_{1s}^2w_{2s}^2]+1$ and $w_{ij}$ is the $j$-th element of $\bbw_i$, $i=1,2$.
\end{thm}



\subsection{Tracy-Widow law for the largest nonspiked eigenvalue of $\bbS_n$}

We now turn to the limiting distribution of the largest nonspiked eigenvalue of the sample covariance matrix $\bbS_n$. The limiting law is of interest in its own right and it is also important for the estimation of the number of the spikes.  To this end we introduce additional assumptions.

\begin{assu}\label{0829-1b}
There exist constants $c_k$ such that $\mathbb{E}|\bbx_{ij}|^k\le c_k$ for all $k\in \mathbb{N}^+$.
\end{assu}

 \begin{assu}\label{1128-2} Recall (\ref{a7}) and (\ref{a8}).
Let $m_{\Sigma_1}(z)$ be the Steiltjes transform of the limit of the spectral distribution (LSD) of $\frac{1}{n}\bbX^{\intercal}\bold\Sigma_1\bbX$  and let $\gamma_+$ be the right most end point of the LSD of $\bbX^{\intercal}\bold\Sigma_1\bbX$. Suppose that
\begin{eqnarray}\label{1010.1h}
\lim\sup_{n}\mu_{K+1} d<1,
\end{eqnarray}
where $d=-\lim\limits_{z\in \mathbb{C}^+\rightarrow \gamma_+}m_{\Sigma_1}(z)$.
 \end{assu}
Intuitively, (\ref{1010.1h}) restricts the upper bound of $\mu_{K+1}$ to ensure $\lambda_{K+1}$ to be a nonspiked eigenvalue.
 Denote the $i$-th largest eigenvalue of $\frac{1}{n}\bbX^{\intercal}\bold\Sigma_1\bbX$ by $\nu_i$. Note that  the limiting law of $\nu_{1}$ is the Type-1 Tracy-Widom distribution.


\begin{thm}\label{0606-2}
Suppose Assumptions  \ref{0829-1c},  \ref{0829-1b}, and \ref{1128-2} hold. In addition, either Assumption \ref{0829-1} or \ref{0829-1h} holds. 
Recalling $l$ above \eqref{a2}, $l\ll n^{1/6}$ and $p\sim n$. 
 For any $i$ satisfying $1\le i-K\leq \log n$,
we have, with high probability,
$$|\lambda_i-\nu_{i-K}|\le n^{-2/3-\ep},$$
In particular,  $\lambda_{K+1}$ has limiting Type-1 Tracy-Widom distribution.
\end{thm}

\begin{rmk}
Theorem \ref{0606-2} shows that the non-spiked sample eigenvalues $\lambda_{K+1}$, $\lambda_{K+2}$,..., $\lambda_{K+r}$ share the same asymptotic distribution as $\nu_{1}$, $\nu_{2}$,..., $\nu_{r}$ since the fluctuation of $\nu_{1}$, $\nu_{2}$,..., $\nu_{r}$ are $n^{-2/3}\gg n^{-2/3-\ep}$
. Here r is a fixed integer. See \cite{BPZ2014a} and \cite{KY11} for more details.
\end{rmk}


\section{Estimating The Number of Spiked Eigenvalues}\label{sec5h}
\label{Number.sec}

Identifying the number of spikes is an important problem for a range of statistical applications. For example, a critical step in PCA is the determination of the number of the significant principal components. This issue arises in virtually all practical applications where PCA is used. Choosing the number of principal components is often subjective and based on heuristic methods. As an application of the main theorems discussed in the last section, we propose in this section a procedure to identify the number of the spiked eigenvalues.

Suppose that the conditions of Theorem \ref{0606-2} hold. Define the asymptotic variance of $\nu_1$ by (see also (3) of \cite{K2007} )
\begin{eqnarray}\label{0604.3}
\sigma_n^3=\frac{1}{d^3}(1+\frac{p-K}{n}\int (\frac{\lambda d}{1-\lambda d})^3dF_{\bold\Lambda_P}(\lambda)).
\end{eqnarray}
By Theorem \ref{0606-2}, $\lambda_{K+1}$ has the same asymptotic distribution as $\nu_1$. Together with Theorem 1 of \cite{K2007}, we have
 \begin{eqnarray}\label{0216.5}
 n^{2/3}\frac{\lambda_{K+1}-\gamma_+}{\sigma_n}\stackrel{D}{\longrightarrow} TW1,
\end{eqnarray}
where $TW1$ is the Type-1 Tracy-Widom distribution. Onatski \cite{O09} also established such a result for the complex case, but Theorem 1 of \cite{O09} requires that the spiked eigenvalues are much bigger than $n^{2/3}$ and $p/n=o(1)$. Moreover, the statistics used in \cite{O09} does not estimate $\gamma_+$ and $\sigma_n$, while our approach estimates them.


Recall that $\gamma_+$ is the asymptotic mean of $\lambda_{K+1}$. From (\ref{0216.5}) one can see that the confidence interval of $\gamma_+$ is $[\lambda_{K+1}-w^*\sigma_n n^{-2/3},\lambda_{K+1}+w^*\sigma_n n^{-2/3}]$, where $w^*$ is a suitable critical value from the Type-1 Tracy-Widom distribution. This, together with Theorem \ref{0408-2}, implies that it suffices to count the number of the eigenvalues of $\bbS_n$ that lie beyond $(\gamma_++w^*\sigma_n n^{-2/3}\log n)$ to estimate the number of spikes $K$ where $\log n$ can be replaced by any number tending to infinity. However, in practice $\gamma_+$ and $\sigma_n$ are unknown and need to be estimated. 

We first consider estimation of $\sigma_n$. It turns out that
\begin{eqnarray}\label{0604.6}
\sigma_n= \left(-\lim_{z\rightarrow \gamma_+^+}\frac{\int \frac{dF_{0}(x)}{(x-z)^3}}{(\int \frac{dF_{0}(x)}{(x-z)^2})^3}\right)^{1/3},
\end{eqnarray}
where $F_0(x)$ is the limit of the spectral distribution function of $\frac{1}{n}\bbX^*\bold\Sigma_1\bbX$
(see Section 7 in the supplementary material).  Moreover, one can verify that with high probability
\begin{eqnarray}\label{h0604.7}
\lambda_{K+1}\le \lambda_{n^{1/6}}+ \log n \times n^{-5/9}
\end{eqnarray}
(see Section 7 in the supplementary material). In view of (\ref{h0604.7}) we estimate $F_{0}(x)$ by its empirical version
$\lambda_{n^{1/6}},\lambda_{n^{1/6}+1},..., \lambda_{n}$ in (\ref{0604.6}), i.e. we exclude the first $n^{1/6}$ eigenvalues of $\bbS_n$. Moreover, for $\gamma_+$ in (\ref{0604.6}), we use $\lambda_{n^{1/6}}+n^{-4/9}$ to replace it. The reason for using  $\lambda_{n^{1/6}}+n^{-4/9}$ to estimate $\gamma_+$ instead of $\lambda_{n^{1/6}}$  is to avoid singularity in  $\int \frac{dF_0(x)}{(x-\gamma_+)^3}$. The estimator of $\sigma_n$ is then given by
$$\hat \sigma_n = \left(-\frac{\frac{1}{n-n^{1/6}}\sum_{i=n^{1/6}}^n \frac{1}{(\lambda_i-z_0)^3}}{(\frac{1}{n-n^{1/6}}\sum_{i=n^{1/6}}^n\frac{1}{(\lambda_i-z_0)^2})^3}\right)^{1/3}, \quad\mbox{\rm where}\quad  z_0=\lambda_{n^{1/6}}+n^{-4/9}.$$

We next consider estimation of $\gamma_+$, the asymptotic mean of $\lambda_{K+1}$. By the assumption that $K\ll n^{1/6}$, it follows from Theorems \ref{0408-2} and  \ref{0606-2} that $\lambda_{n^{1/6}}$ is not a spiked eigenvalue. Based on this, an upper bound of $\lambda_{K+1}$ is given in (\ref{h0604.7}). 
Hence we use the following $\hat p_0$ as an initial upper bound of $\lambda_{K+1}$
\begin{eqnarray}\label{0604.7}
\hat p_0= \lambda_{n^{1/6}}+\log n \times n^{-5/9}.
\end{eqnarray}

Although  $\hat p_0$ is a good upper bound for $\lambda_{K+1}$ theoretically, it does not depend on $\sigma_n$ and hence in practice $\hat p_0$ may not work well. Based on (\ref{0216.5}), we propose the following iteration approach to update $\hat p_0$. The idea behind the iteration is that even if $\hat p_0$ is not larger than $\lambda_{K+1}$ in practice, $\hat p_0$ is still close to $\lambda_{K+1}$. Thus by (\ref{0216.5}), there is at least one eigenvalue in the interval $[\hat p_{0},\hat p_{0}+w^*m_n\sigma_n n^{-2/3}]$, where $m_n\rightarrow \infty$.


\begin{enumerate}
\item Define the initial value $\hat p_0$ in (\ref{0604.7}).

\item Suppose that we have $\hat p_{m-1}$. If there is at least one eigenvalue of $\bbS_n$ belonging to $[\hat p_{m-1},\hat p_{m-1}+2.02(\log n)\sigma_n n^{-2/3}]$, where $2.02$ is the $99\%$ quantile of Type-1 Tracy-Widom distribution, we renew $\hat p_{n}=\hat p_{n-1}+2.02\log n\sigma_n n^{-2/3}.$ 
    Here $\log n$ can be also replaced by the other number tending to infinity too. Otherwise the iteration stops.

\item After getting $\hat p_{n}$, we return to Step 2 until the iteration stops.

\item Denote the final value of the above iteration by $\hat p_{end}$. We define $\hat K$ to be the number of eigenvalues larger than $\hat p_{end}$. 
\end{enumerate}


Theorem \ref{0606-2} implies that $\hat K$ is a good estimator of  the number of significant components $K$.
\begin{prop}\label{0103-1}
Under the conditions of Theorem \ref{0606-2}, we have $\hat K=K$ with high probability.
\end{prop}

\subsection*{Identifying The Number of Factors}

A closely related problem is the estimation of the number of factors under a factor model, which is widely used in financial econometrics.  Consider the factor model
\begin{eqnarray}\label{3(1)}
\bby_{t}=\Lambda\bbf_t+\bbT\varepsilon_{t}, \ \  t=1, 2, \ldots, n,
\end{eqnarray}
where $\Lambda$ is  $p\times K$-dimensional factor loading, $\bbf_t$ is the corresponding $K$-dimensional  factor, $\{\varepsilon_{it}: i=1,2,\ldots,p; t=1,2,\ldots,n\}$ are the independent idiosyncratic components.

In many applications, the number of factors $K$ is unknown. An important step in factor analysis is to determine the value of $K$. Let  $\bbF=(\bbf_1,...,\bbf_n)$, $\bbZ=(\varepsilon_1,...,\varepsilon_n)$ and $\bbY=(\bby_1,...,\bby_n)$. Then (\ref{3(1)}) can be rewritten as
\begin{equation}\label{0102.1}
\bbY=\bold\Lambda\bbF+\bbT\bbZ= (\begin{matrix}\bold\Lambda &\bbT \end{matrix}) \left(\begin{matrix}\bbF\\ \bbZ \end{matrix}\right).
\end{equation}

Suppose that $\left(\begin{matrix}\bbF\\ \bbZ \end{matrix}\right)$ satisfies Assumptions \ref{0829-1a} and \ref{0829-1b} and $(\begin{matrix}\bold\Lambda &\bbT \end{matrix})$ satisfies Assumptions \ref{0829-1} and \ref{1128-2}. It is easy to conclude that the $(K+1)$-st largest eigenvalue of ${1\over n} \bbY\bbY^{\intercal}$ follows the Type-1 Tracy-Widom distribution asymptotically.
The following result is a direct consequence of Proposition \ref{0103-1}.
\begin{coro}\label{0103-1x}
For the model (\ref{3(1)}), if $\left(\begin{matrix}\bbF\\ \bbZ \end{matrix}\right)$ satisfies Assumptions \ref{0829-1a} and \ref{0829-1b} and $(\begin{matrix}\bold\Lambda &\bbT \end{matrix})$ satisfies Assumptions \ref{0829-1} and \ref{1128-2}, $K\ll n^{1/6}$ and $p\sim n$, then we have $\hat K=K$ with high probability.
\end{coro}
Comparing to the approaches in \cite{BS002} and \cite{O09}, here we allow the number of factors $K$ to diverge with $n$. Moreover, we only assume that the spiked population eigenvalues diverge to infinity, while \cite{BS002} and \cite{O09} assume that they are much larger than $n^{2/3}$ or grow linearly with $n$.


\section{Estimating the Eigenvectors}\label{sec3h}
\label{Eigenvectors.sec}

As mentioned in the introduction, the leading eigenvectors of the population covariance matrix are of significant interest in PCA and many other statistical applications. They are conventionally estimated by their empirical counterparts.

We consider in this section estimation of the population eigenvectors associated with the spiked population eigenvalues $\mu_1$,...,$\mu_K$, involved in $\sigma_i^2$ in (\ref{0831.3}). To this end, we first characterize the relationship between the sample eigenvectors and the corresponding population eigenvectors. Write the population eigenvectors matrix $\bbV$ as $\bbV=(\bbv_1,\cdots,\bbv_{p})$.

\begin{thm}\label{0503-1}
Suppose that the  conditions of Theorem \ref{0408-2} hold. Let $\xi_i$ be the eigenvector of  $\bbS_n$ corresponding to the eigenvalue $\lambda_i$. 
  Then for $1\le i \le K$, we have
\begin{eqnarray}\label{1130.1}
\bbv_i^{\intercal}\xi_i\xi_i^{\intercal}\bbv_i\stackrel{i.p.}{\longrightarrow} 1.
\end{eqnarray}
\end{thm}

Theorem \ref{0503-1} also implies that for $i=1,...,K$, $j=1,...,p$, $i\neq j$, we have
$$\bbv_j^{\intercal}\xi_i\xi_i^{\intercal}\bbv_j\stackrel{i.p.}{\longrightarrow} 0.$$
One should notice that the convergence is uniformly for $j=1,...,p$ since $1=\xi_i^{\intercal}\xi_i=\sum_{j=1}^p\bbv_j^{\intercal}\xi_i\xi_i^{\intercal}\bbv_j$.

Theorem \ref{0503-1} shows that the sample eigenvector $\xi_i$ is a good estimator of $\bbv_i$ up to a sign difference.
An immediate application of Theorem \ref{0503-1} is to estimate $\sigma_i^2$ for the case when $\bbV=\bbU^{\intercal}$ and $\gamma_{41}=...=\gamma_{4p}=\gamma_{4}$ by Corollary \ref{1130-1}. This corollary shows that the empirical eigenvector plays an important role in statistical inference of the spiked eigenvalue.

\begin{coro}\label{1130-1}
 Under the conditions of Theorem \ref{0503-1}, we have
$$\sum_{j=1}^pv_{ij}^4-\sum_{j=1}^p\xi_{ij}^4\stackrel{i.p.}{\longrightarrow} 0.$$
\end{coro}

We now consider the extension to the case when the multiplicity of the population eigenvalues $\mu_i$ is more than one.
Correspondingly the following corollary holds and its proof is the same as that of Theorem \ref{0503-1}. 
\begin{coro}\label{0831-2}
Recall the definition of $r_i$ above (\ref{0408.4}). Under the conditions of Theorem \ref{0831-1}, 
The angle between $\bbv_k$, $k\in \{r_{i-1}+1,..., r_i\}$ and the subspace spanned by $\{\xi_j, j=r_{i-1}+1,..., r_i\}$  tends to 0 in probability. In other words,
we have
$$\bbv_k^{\intercal}(\sum_{j=r_{i-1}+1}^{r_i}\xi_j\xi_j^{\intercal})\bbv_k\stackrel{i.p.}{\longrightarrow} 1, \ k\in \{r_{i-1}+1,..., r_i\}.$$
\end{coro}

Corollary \ref{0831-2} shows that the sample eigenvectors $\{\xi_j, j=r_{i-1}+1,..., r_j\}$ are close to the space spaned by $\{\bbv_j, j=r_{i-1}+1,..., r_j\}$. 


\section{Centralized sample covariance matrices}
\label{Centralized.sec}

So far we have focused on  the non-centralized sample covariance matrix $S_n$. We now turn to its centralized version
\[
\tilde \bbS_n = \frac{1}{n}\sum_{i=1}^n\bold\Gamma(\bbx_i-\bar \bbx)(\bbx_i-\bar \bbx)^{\intercal}\bold\Gamma^{\intercal}=\Gamma\bbX(\bbI-\frac{1}{n}\mathbf{1}\mathbf{1}^{\intercal})\bbX^{\intercal}\bold\Gamma^{\intercal},
\]
 where $\mathbf{1}$ is the $n\times 1$ vector with all elements being 1. Denote $(\bbI-\frac{1}{n}\mathbf{1}\mathbf{1}^{\intercal})$ by $\Upsilon$. First we have the following Lemma.
\begin{lem}\label{1101-1x}
Under the conditions of Theorem \ref{1101-1x}, we have
\begin{eqnarray}\label{1101.1x}
\frac{\sqrt n}{\tilde \sigma_1}\Big(\bbw_1^{\intercal}\bbX\Upsilon(n\bbI-\Upsilon\bbX^{\intercal}\frac{\bold\Sigma_1}{\theta}\bbX\Upsilon)^{-1}\bbX^{\intercal}\bbw_1+\tilde m_{\theta}(1)\Big)\stackrel{D}{\rightarrow} N\left(0, 1\right)
\end{eqnarray}
and
\begin{eqnarray}\label{1101.2x}
\frac{\sqrt n}{\tilde\sigma_{12}}\bbw_1^{\intercal}\bbX\Upsilon(n\bbI-\Upsilon\bbX^{\intercal}\frac{\bold\Sigma_1}{\theta}\bbX\Upsilon)^{-1}\Upsilon\bbX^{\intercal}\bbw_2\stackrel{D}{\rightarrow} N\left(0, 1\right)
\end{eqnarray}
where $\tilde \sigma_1^2=\sum_{j=1}^{p+l}[(\gamma_{4j}-3)\bbw_{1j}^4]+2$, $\tilde\sigma_{12}^2=\sum\limits_{s=1}^{p+l}[(\gamma_{4s}-3)\bbw_{1s}^2\bbw_{2s}^2]+1$ and $\bbw_{ij}$ is the $j$-th element of $\bbw_i$, $i=1,2$.
\end{lem}

By Lemma \ref{1101-1x} and checking carefully the proofs of the  main results, it can be seen that all arguments remain valid if  $\bbX$ is  replaced by $\bbX\Upsilon$ (note that $\Upsilon^2=\Upsilon$).   So Theorem \ref{0318-1}--Corollary \ref{0831-2} hold for $\frac{1}{n}\Gamma\bbX\Upsilon\bbX^{\intercal}\Gamma^{\intercal}$ as well.



\section{Numerical Results }
\label{sec6h}

In this section we illustrate some of the theoretical results obtained earlier through numerical experiments. We first use simulation to confirm that the asymptotic behavior of the spiked eigenvalues is indeed affected by the population eigenvectors.

Let $K=2$ and $\Lambda_{P}=\diag(\mu_3,...,\mu_p)$. Suppose that $\{\mu_i,\  i=3,...,p\}$ are i.i.d. copies of the uniform random variable $U(1, 2)$. Define $\bbv_1=(\frac{1}{\sqrt 2},\frac{1}{\sqrt 2})^{\intercal}$, $\bbv_2=(\frac{1}{\sqrt 2},-\frac{1}{\sqrt 2})^{\intercal}$, $\breve\bbV=(\bbv_1,\bbv_2)$ and $\Lambda_{S}=\diag(800,200)$. We now define two different population matrices
$$\bold\Sigma_2=\left(
 \begin{matrix}
   \bold\Lambda_S & 0\\
  0 &\bold\Lambda_P
  \end{matrix}
  \right),\qquad \bold\Sigma_3=\left(
 \begin{matrix}
   \breve\bbV\bold\Lambda_S\breve\bbV^{\intercal} & 0\\
  0 &\bold\Lambda_P
  \end{matrix}
  \right)\ .$$
  Then the eigenvalues of $\bold \Sigma_2$ and $\bold \Sigma_3$ are the same but the eigenvectors corresponding to the first two largest eigenvalues are different. Consider the case $p=n$ and $\bbX=(x_{ij})$ are i.i.d. $U(-\sqrt 3,\sqrt 3)$. Denote by  $\check \lambda_1$ and $\breve \lambda_1$ respectively the largest eigenvalues of the sample covariance matrices $\frac{1}{n}\bold \Sigma_2^{1/2}\bbX\bbX^{\intercal}\bold \Sigma_2^{1/2}$ and $\frac{1}{n}\bold \Sigma_3^{1/2}\bbX\bbX^{\intercal}\bold \Sigma_3^{1/2}$. Table \ref{tab1.1} reports the sample variance of the rescaled eigenvalues $\frac{\sqrt n\check \lambda_1}{800}$ and $\frac{\sqrt n\breve \lambda_1}{800}$ .
It can be seen that the behavior of the spiked sample eigenvalues is indeed affected by the population eigenvectors.

{\small  \begin{table}[htbp]
 \begin{center}
  \caption{ The variances of the rescaled largest eigenvalues }
 \label{tab1.1}
    \begin{tabular}{rrrrrrrrrr}
        \toprule
      p    & 200 & 400 & 600 & 800 & 1000 \\
              \midrule
    $\bold \Sigma_2$ & 0.8111
  & 0.7965 &	0.8287&	0.7574	& 0.7874
  \\
    $\bold \Sigma_3$ & 1.2507
  & 1.4051	&1.2800&	1.5012&	1.3911  \\
        \bottomrule
    \end{tabular}%
    \end{center}
    \end{table}
    }

We now consider estimating the number of factors under the factor model (\ref{0102.1}):
$$\bbY=\Lambda\bbF+\bbT\bbZ.$$
In the simulation, the entries of $\bbF$ and $\bbZ$ follow the standard Gaussian distribution. Consider two choices: $\bbT=\bbT_1$ or $\bbT_2$, where $\bbT_1=\bbI$, $\bbT_2=\diag(\underbrace{1,1,...,1}_{p/2},\underbrace{\frac{1}{\sqrt 2},...,\frac{1}{\sqrt 2}}_{p/2})$. Let $\Lambda$ be a $p\times K$  matrix with nonzero entries being $(\Lambda_{11},...,\Lambda_{KK})=(\sqrt{b^2_1-1},...,\sqrt{b^2_K-1})$ where $K=5\lceil n^{1/7}\rceil+1$, and $(b_1,..,b_K)=\sqrt{(6,...,6+K-1)*r}+1$, $0\le r\le 1$.


Since the estimator in \cite{O09} performs better than that in \cite{BS002}, we shall only consider the estimator given in \cite{O09} for our comparisons. We compare the accuracy of estimating the number of factors $K$ for three methods: our procedure proposed in Section \ref{sec5h}, the method introduced in \cite{O09}, and  the approach given in \cite{B17}, which are denoted by CHP, Ons, and BYK, respectively.
Here we omit the simulation results of BIC used in \cite{B17} for reasons of space.   The initial value of $\hat p_0$ in (\ref{0604.7}) is replaced by $\lambda_{15\lceil n^{1/6}\rceil }+\log n \times n^{-5/9}$ according to our extensive simulations in order to reduce the number of updating iteration. 
 Here we replace $\lambda_{\lceil n^{1/6}\rceil }$ by $\lambda_{15\lceil n^{1/6}\rceil }$ and one should note that all of the conclusions in Section \ref{sec5h} still hold since $15$ is a constant.
 The approach in Section 5.3 of \cite{O09} uses the ratio of the differences of the adjacent sample eigenvalues to conduct the sequential test of
$$H_0: K=k_0\ vs \ H_1: k_0<K<k_1,$$
from $k_0=0$ to $k_0=k_1-1$.  
\cite{B17} uses AIC based on sample eigenvalues to estimate $K$.

Different combinations of $n$ and $p$ are considered. The following tables report the proportion of times  the number of factors is correctly identified, i.e. $\hat K=K$, where for each $(n,p)$ we repeat 500 times.  Different choices of $r$ (ranging from 0.3 to 1) are also considered.
From Tables \ref{tab1} and \ref{tab2}, one can see that the accuracy of our approach increases as $(n,p)$ become larger. Comparing to \cite{O09}, one can find that our approach works much better when the number of factors increases with $n$. This is reasonable since the estimator given in \cite{O09} is very sensitive to the predetermined  non-spiked eigenvalue (i.e. $k_1$ in \cite{O09}). If $k_1$ is too large, the power may be poor. Tables \ref{tab1} and \ref{tab2} show that the method based on the AIC criterion and our procedure have similar performance. But as mentioned earlier in Remark \ref{rk1}, the model in \cite{B17} only allows that $\mu_{K+1}=...=\mu_{p}=1$, which is a special case of what we consider in the present paper. Indeed, Table \ref{tab3} also confirms that for the non-identity matrix $\bbT_2$, the method based on the AIC criterion performs much worse than our approach.  Therefore, our procedure is preferred for the case where $\mu_{K+1},...,\mu_{p}$ are unknown.

\begin{table}[!htbp]
 \begin{center}
  \caption{  \label{tab1} Ratio of Identifying The Correct Number of Factors with $\bbT_1$ }

    \begin{tabular}{lccc|ccc|ccc}
        \toprule
       $r\backslash (n,p)$  & & \multicolumn{1}{l}{(50,50)}&  &  & \multicolumn{1}{l}{(50,100)} & & & \multicolumn{1}{l}{(50,150)}&  \\
              \midrule
  & CHP & Ons& BYK &  CHP & Ons& BYK &CHP & Ons& BYK   \\
  \midrule
    0.3   & 0.608 & 0.052& 0.610 & 0.192 & 0.072& 0.330 & 0.068 &0.060 & 0.122  \\
    0.4   & 0.816 & 0.064 & 0.706& 0.442  &0.046 & 0.618& 0.184 &0.056 & 0.368 \\
    0.5   & 0.904 & 0.044 & 0.662 & 0.676 & 0.040 &0.788 & 0.450  &0.062 & 0.606 \\
    0.6   & 0.892 & 0.038 & 0.612& 0.832 & 0.044 & 0.880& 0.638   & 0.064&0.800 \\
    0.7   & 0.906 & 0.050 &0.636 &0.880  & 0.044 & 0.870&0.756  & 0.064&0.866 \\
    0.8   & 0.914 & 0.060 & 0.638&0.918 & 0.048 & 0.886 &0.868   &0.070 & 0.880\\
    0.9   & 0.908 & 0.054 &0.648 &0.948 & 0.058 & 0.866 & 0.916  &0.056& 0.910\\
    1.0     & 0.914 & 0.050 & 0.616&0.946 & 0.052 & 0.872 &0.912 & 0.082& 0.896\\

        \bottomrule
    \end{tabular}%
      \end{center}

 \begin{center}
  \caption{  \label{tab2} Ratio of Identifying The Correct Number of Factors with $\bbT_1$}

    \begin{tabular}{lccc|ccc|ccc}
        \toprule
       $r\backslash (n,p)$  & & \multicolumn{1}{l}{(100,100)}&  &  & \multicolumn{1}{l}{(100,200)} & & & \multicolumn{1}{l}{(100,300)}&  \\
              \midrule
  & CHP & Ons& BYK &  CHP & Ons& BYK &CHP & Ons& BYK   \\
  \midrule
    0.3   & 0.954 & 0.052& 0.974 & 0.772 & 0.034& 0.854 & 0.392 &0.076 & 0.482  \\
    0.4   & 0.980 & 0.038 & 0.982& 0.942  &0.034& 0.984& 0.782 &0.058 & 0.908 \\
    0.5   & 0.956 & 0.056 & 0.974 & 0.964 & 0.052 &0.990 & 0.938  &0.056 & 0.976 \\
    0.6   & 0.972 & 0.050 & 0.976& 0.980 & 0.048 & 0.994& 0.966   & 0.066&0.990 \\
    0.7   & 0.970 & 0.058 &0.974 &0.978  & 0.050 & 0.986&0.972  & 0.074&0.996 \\
    0.8   & 0.954 & 0.040 & 0.974&0.972 & 0.042 & 0.998 &0.984  &0.064 & 0.980\\
    0.9   & 0.954 & 0.050 &0.980&0.970 & 0.042 & 0.986 & 0.980  &0.044& 0.984\\
    1.0     & 0.950 & 0.052 & 0.972&0.958 & 0.052 & 0.984 &0.982 & 0.074& 0.988\\

        \bottomrule
    \end{tabular}
      \end{center}

 \begin{center}
  \caption{  \label{tab3} Ratio of Identifying The Correct Number of Factors with $\bbT_2$}

    \begin{tabular}{lccc|ccc|ccc}
        \toprule
       $r\backslash (n,p)$  & & \multicolumn{1}{l}{(100,100)}&  &  & \multicolumn{1}{l}{(100,200)} & & & \multicolumn{1}{l}{(100,300)}&  \\
              \midrule
  & CHP & Ons& BYK &  CHP & Ons& BYK & CHP & Ons& BYK   \\
  \midrule
    0.3   & 0.946  & 0.062  & 0.490  & 0.938  & 0.062  & 0.658  & 0.792  & 0.040  & 0.716  \\
    0.4   & 0.928  & 0.042  & 0.454  & 0.974  & 0.042  & 0.624  & 0.968  & 0.044  & 0.710  \\
    0.5   & 0.944  & 0.044  & 0.424  & 0.968  & 0.058  & 0.682  & 0.986  & 0.038  & 0.704  \\
    0.6   & 0.926  & 0.052  & 0.440  & 0.966  & 0.046  & 0.672  & 0.978  & 0.066  & 0.654  \\
    0.7   & 0.926  & 0.034  & 0.434  & 0.970  & 0.066  & 0.662  & 0.972  & 0.040  & 0.670  \\
    0.8   & 0.918  & 0.060  & 0.450  & 0.978  & 0.060  & 0.650  & 0.986  & 0.042  & 0.660  \\
    0.9   & 0.928  & 0.052  & 0.434  & 0.978  & 0.052  & 0.608  & 0.980  & 0.058  & 0.670  \\
    1.0   & 0.930  & 0.048  & 0.410  & 0.980  & 0.036  & 0.614  & 0.976  & 0.048  & 0.658  \\

        \bottomrule
    \end{tabular}%
      \end{center}
\end{table}%

\section{Proofs}
\label{proofs.sec}

In this section, we prove one of the main results, Theorem \ref{1101-1}. The proof of Theorem \ref{0408-2} is involved. For reasons of space, we prove Theorem \ref{0408-2}  in detail in the supplement \cite{CHP17}.
The proofs of the other results and additional technical lemmas are also provided in the supplement \cite{CHP17}.

\subsection{Proof of Theorem \ref{1101-1}}\label{Appb}

The main idea of this proof is to first express $\bbw_1^{\intercal}\bbX(n\bbI-\bbX^{\intercal}\frac{\bold\Sigma_1}{\theta}\bbX)^{-1}\bbX^{\intercal}\bbw_1$ as a sum of martingale differences and then apply the central limit theorem for the martingale difference.

We below consider the case $p\gtrsim n$ and  prove (\ref{1101.1}) only because the case $\frac{p}{n}\rightarrow 0$ and (\ref{1101.2}) can be proved similarly. First of all, we need to do truncation and centralization on $\bbx_{ij}$ as in the first paragraph of Section \ref{0914-1}  in the supplement \cite{CHP17}. In fact, by (\ref{0321.3})-(\ref{0912.1}),  we conclude that the truncation and centralization do not affect the CLT. i.e. we can get the following inequality similar to (\ref{1014.1h})
$$\bbw_1^{\intercal}\bbX(n\bbI-\bbX^{\intercal}\frac{\bold\Sigma_1}{\theta}\bbX)^{-1}\bbX^{\intercal}\bbw_1=\bbw_1^{\intercal}\tilde \bbX(n\bbI-\tilde \bbX^{\intercal}\frac{\bold\Sigma_1}{\theta}\tilde \bbX)^{-1}\tilde \bbX^{\intercal}\bbw_1+o_p(\frac{1}{\sqrt n}),$$
where $\tilde \bbX$ is the truncated and centralized version of $\bbX$.
The argument is standard and we omit the details here. Therefore, for simplicity we below assume that
$$E\bbx_{ij}=0,\quad |\bbx_{ij}|\le \delta_n \sqrt[4]{np}. $$

\subsection*{Calculation of The Variance}
Define the following events
$$F_d=\{\|\frac{1}{n}\bbX^{\intercal}\bold\Sigma_1\bbX\|\le 4\|\bold\Sigma_1\|(1+\frac{p}{n})\}, \ F^{(k)}_d=\{\|\frac{1}{n}\bbX_k^{\intercal}\bold\Sigma_1\bbX_k\|\le 4\|\bold\Sigma_1\|(1+\frac{p}{n})\}, k=1,...,n,$$
where $\bbX_k=\bbX-\bbx_k\bbe_k^{\intercal}$, $\bbx_k$ is the $k$-th column of $\bbX$ and $\bbe_k=(0,..,0,1,0,...,0)^{\intercal}$ is a $M$-dimensional vector with only $k$-th element being 1.
 By Theorem 2 of \cite{BG2012}, we have
 \begin{equation}\label{1101.5}
 I(F_d)=1 \quad {\rm and} \quad  I(F^{(k)}_d)=1, \ k=1,...,n
 \end{equation}
with high probability. 

 We define $\frac{\bold\Sigma_1}{\theta}=\tilde{\bold\Sigma}_1$, $\bbA=\bbI-\frac{1}{n}\bbX^{\intercal}\tilde{\bold\Sigma}_1\bbX$, $\bbA_k=\bbI-\frac{1}{n}\bbX_k^{\intercal}\tilde{\bold\Sigma}_1\bbX_k$ and $\bbA_{(k)}=\bbA_k-\frac{1}{n}\bbX_k^{\intercal}\tilde{\bold\Sigma}_1\bbx_k\bbe_k^{\intercal}$. Then
 $\bbA=\bbA_k-\frac{1}{n}(\bbe_k\bbx_k^{\intercal}\tilde{\bold\Sigma}_1\bbX_k+\bbX_k^{\intercal}\tilde{\bold\Sigma}_1\bbx_k\bbe_k^{\intercal}+\bbe_k\bbx_k^{\intercal}\tilde{\bold\Sigma}_1\bbx_k\bbe_k^{\intercal})$. Therefore,
\begin{eqnarray}\label{1127.1}
\bbw_1^{\intercal}\bbX(n\bbI-\bbX^{\intercal}\frac{\bold\Sigma_1}{\theta}\bbX)^{-1}\bbX^{\intercal}\bbw_1=\frac{1}{n}\bbw_1^{\intercal}\bbX\bbA^{-1}\bbX^{\intercal}\bbw_1.
\end{eqnarray}

By the definition of $\bbX_k$ and $\bbA_k$, we observe that the $k$-th row and $k$-th colomn of $\bbA_k$ are 0 except for the diagonal entry. Hence it is not hard to conclude the following important facts
\begin{eqnarray}\label{1101.3}
\bbe_k^{\intercal}\bbA_k^{-1}\bbe_k=1,
\end{eqnarray}

\begin{eqnarray}\label{1101.4}
\bbe_i^{\intercal}\bbA_k^{-1}\bbe_k=0, \ \ i\neq k
\end{eqnarray}
and
\begin{eqnarray}\label{1101.4h}
\bbX_k\bbA_k^{-1}\bbe_k=\bbX_k\bbe_k=0.
\end{eqnarray}
   In the sequel, we prove the central limit theorem for $\frac{1}{n}\bbw_1^{\intercal}\bbX\bbA^{-1}\bbX^{\intercal}\bbw_1I(F_d)$ instead of $\frac{1}{n}\bbw_1^{\intercal}\bbX\bbA^{-1}\bbX^{\intercal}\bbw_1$. In fact, it follows from (\ref{1101.5}) that $I(F_d)=1$ with high probability. Therefore $\frac{1}{n}\bbw_1^{\intercal}\bbX\bbA^{-1}\bbX^{\intercal}\bbw_1$ and $\frac{1}{n}\bbw_1^{\intercal}\bbX\bbA^{-1}\bbX^{\intercal}\bbw_1I(F_d)$ have the same central limit theorem. Let $\mathbb{E}_k=\mathbb{E}(.|\bbx_1,...,\bbx_k)$, $\mathbb{E}=\mathbb{E}(.)$  and write
\begin{eqnarray}\label{0314.1}
&&\bbw_1^{\intercal}\bbX\bbA^{-1}\bbX^{\intercal}\bbw_1I(F_d)-\mathbb{E}\bbw_1^{\intercal}\bbX\bbA^{-1}\bbX^{\intercal}\bbw_1I(F_d)\\
&&=\sum_{k=1}^n(\mathbb{E}_k-\mathbb{E}_{k-1})\bbw_1^{\intercal}\bbX\bbA^{-1}\bbX^{\intercal}\bbw_1I(F_d)\non
&&=\sum_{k=1}^n(\mathbb{E}_k-\mathbb{E}_{k-1})\big(\bbw_1^{\intercal}\bbX\bbA^{-1}\bbX^{\intercal}\bbw_1I(F_d)-\bbw_1^{\intercal}\bbX_k\bbA_k^{-1}\bbX_k^{\intercal}\bbw_1I(F^{(k)}_d)\big)\non
&&=\sum_{k=1}^n(\mathbb{E}_k-\mathbb{E}_{k-1})\big(\bbw_1^{\intercal}\bbX\bbA^{-1}\bbX^{\intercal}\bbw_1-\bbw_1^{\intercal}\bbX_k\bbA_k^{-1}\bbX_k^{\intercal}\bbw_1\big)I(F_d)+o_p(n^{-2})\non
&&=\sum_{k=1}^n(\mathbb{E}_k-\mathbb{E}_{k-1})(I_1+2I_2+I_3-\bbw_1^{\intercal}\bbX_k\bbA_k^{-1}\bbX_k^{\intercal}\bbw_1)I(F_d)+o_p(n^{-2}),\nonumber
\end{eqnarray}
where the third equality follows from (\ref{1101.5}),
$I_1=(\bbw_1^{\intercal}\bbx_k)^2\bbe_k^{\intercal}\bbA^{-1}\bbe_k$, $I_2=\sum_{i\neq k}\bbw_1^{\intercal}\bbx_k\bbw_1^{\intercal}\bbx_i\bbe_i^{\intercal}\bbA^{-1}\bbe_k,$
 and
 $I_3=\sum_{i,j\neq k}\bbw_1^{\intercal}\bbx_i\bbw_1^{\intercal}\bbx_j\bbe^{\intercal}_i\bbA^{-1}\bbe_j.$
We define
 \begin{eqnarray}\label{0324.4}
 a_k=1-\frac{1}{n}(\bbx_k^{\intercal}\tilde{\Sigma}_1\bbX_k\bbA_{(k)}^{-1}\bbe_k+\bbx_k^{\intercal}\tilde{\Sigma}_1\bbx_k\bbe_k^{\intercal}\bbA_{(k)}^{-1}\bbe_k)
 \end{eqnarray}and
\begin{eqnarray}\label{0324.5}
b_k=1-\frac{1}{n}\bbe_k^{\intercal}\bbA_k^{-1}\bbX_k^{\intercal}\tilde{\Sigma}_1\bbx_k.
 \end{eqnarray}

 We next simplify the formula. Noting that $\bbw_1^{\intercal}\bbX=\bbw_1^{\intercal}\bbX_k+\bbw_1^{\intercal}\bbx_k\bbe_k^{\intercal}$, by the formulas
  \begin{eqnarray}\label{1101.11}
  \bbW^{-1}=\bbQ^{-1}-\frac{\bbQ^{-1}(\bbW-\bbQ)\bbQ^{-1}}{1+tr(\bbQ^{-1}(\bbW-\bbQ))}
  \end{eqnarray}
  and
    \begin{eqnarray}\label{1101.11h}
  (\bbQ+\sum_{j=1}^mab_j^{\intercal})^{-1}a=\frac{\bbQ^{-1}a}{1+\sum_{j=1}^mb_j^{\intercal}\bbQ^{-1}a},
  \end{eqnarray}
  we have
 \begin{eqnarray}\label{0314.5}
 &&\\
\bbA^{-1}&=&\bbA_{(k)}^{-1}+\frac{\bbA_{(k)}^{-1}(\bbe_k\bbx_k^{\intercal}\tilde{\Sigma}_1\bbX_k+\bbe_k\bbx_k^{\intercal}\tilde{\Sigma}_1\bbx_k\bbe_k^{\intercal})\bbA_{(k)}^{-1}}{na_k}\non
&=&\bbA_k^{-1}+\frac{\bbA_k^{-1}\bbX_k^{\intercal}\tilde{\Sigma}_1\bbx_k\bbe_k^{\intercal}\bbA_k^{-1}}{nb_k}+\frac{\bbA_{(k)}^{-1}(\bbe_k\bbx_k^{\intercal}\tilde{\Sigma}_1\bbX_k+\bbe_k\bbx_k^{\intercal}\tilde{\Sigma}_1\bbx_k\bbe_k^{\intercal})\bbA_{(k)}^{-1}}{na_k}\nonumber
\end{eqnarray}
and
\begin{eqnarray}\label{0314.2}
\quad I_1&=&(\bbw_1^{\intercal}\bbx_k)^2\bbe_k^{\intercal}\bbA^{-1}\bbe_k=\frac{(\bbw_1\bbx_k)^2\bbe_k^{\intercal}\bbA_{(k)}^{-1}\bbe_k}{a_k}\\
&=&\frac{(\bbw_1\bbx_k)^2\bbe_k^{\intercal}\bbA_k^{-1}\bbe_k}{a_k(1-\frac{1}{n}\bbe_k^{\intercal}\bbA^{-1}_k\bbX_k^{\intercal}\tilde{\Sigma}_1\bbx_k)}=\frac{(\bbw_1\bbx_k)^2\bbe_k^{\intercal}\bbA_k^{-1}\bbe_k}{a_kb_k}=\frac{(\bbw_1\bbx_k)^2}{a_kb_k}.\nonumber
\end{eqnarray}
Moreover, it follows from (\ref{1101.3}), (\ref{1101.4}) and (\ref{1101.11}) that
 \begin{eqnarray}\label{1101.13}
b_k=1-\frac{1}{n}\bbe_k^{\intercal}\bbA_k^{-1}\bbX_k^{\intercal}\tilde{\Sigma}_1\bbx_k=1
 \end{eqnarray} and
 \begin{eqnarray}\label{1101.12}
\quad\quad a_k&=&1-\frac{1}{n}\bbx_k^{\intercal}\tilde{\Sigma}_1\bbX_k\bbA_{(k)}^{-1}\bbe_k=1-\frac{1}{n^2}\bbe_k^{\intercal}\bbA_k^{-1}\bbe_k\bbx_k^{\intercal}\tilde{\Sigma}_1\bbX_k\bbA_k^{-1}\bbX_k^{\intercal}\tilde{\Sigma}_1\bbx_k\\
 &=&1-\frac{1}{n^2}\bbx_k^{\intercal}\tilde{\Sigma}_1\bbX_k\bbA_k^{-1}\bbX_k^{\intercal}\tilde{\Sigma}_1\bbx_k.\nonumber
 \end{eqnarray}
By the Cauchy interlacing property we know
 \begin{eqnarray}
&&\quad  \frac{1}{n^2}\bbx_k^{\intercal}\tilde{\Sigma}_1\bbX_k\bbA_k^{-1}\bbX_k^{\intercal}\tilde{\Sigma}_1\bbx_kI(F_d)\le\frac{1}{n^2}\bbx_k^{\intercal}\tilde{\Sigma}_1\bbx_k\|\tilde{\Sigma}_1^{1/2}\bbX_k\bbA_k^{-1}\bbX_k^{\intercal}\tilde{\Sigma}_1^{1/2}\|I(F_d)\\
 &&=\frac{1}{n^2}\bbx_k^{\intercal}\tilde{\Sigma}_1\bbx_k\|\bbA_k^{-1}\bbX_k^{\intercal}\tilde{\Sigma}_1\bbX_k\|I(F_d)\le \frac{1}{n^2} \bbx_k^{\intercal}\tilde{\Sigma}_1\bbx_k\|\bbA_k^{-1}\|\|\bbX_k^{\intercal}\tilde{\Sigma}_1\bbX_k\|I(F_d)\non
 &&\le 2(\frac{p}{n\theta})^2.\nonumber
 \end{eqnarray}
This implies that
\begin{eqnarray}\label{0321.1}
a_kI(F_d)=1+O((\frac{p}{n\theta})^2).
\end{eqnarray}

As for the term $i\neq k$, by (\ref{1101.4}), (\ref{1101.4h}), (\ref{1101.11}) and (\ref{1101.11h}) we have
\begin{equation}\label{0314.12}
\bbA^{-1}\bbe_k=\frac{\bbA_{(k)}^{-1}\bbe_k}{a_k}=\frac{\bbA_k^{-1}\bbe_k}{a_k}+\frac{\bbA_k^{-1}\bbX_k^{\intercal}\tilde{\Sigma}_1\bbx_k}{a_kb_k} = \frac{\bbA_k^{-1}\bbe_k}{a_k}+\frac{\bbA_k^{-1}\bbX_k^{\intercal}\tilde{\Sigma}_1\bbx_k}{a_k}.
\end{equation}
We then conclude that
\begin{eqnarray}\label{0330.1}
I_2=\sum_{i\neq k}\bbw_1^{\intercal}\bbx_k\bbw_1^{\intercal}\bbx_i\bbe_i^{\intercal}\bbA^{-1}\bbe_k=\frac{\bbw_1^{\intercal}\bbX_k\bbA_k^{-1}\bbX_k^{\intercal}\tilde{\Sigma}_1\bbx_k\bbx^{\intercal}_k\bbw_1}{na_k}.
\end{eqnarray}
It follows from (\ref{1101.4}), (\ref{1101.4h}) and  (\ref{0314.5}) that for $i,j\neq k$
\begin{eqnarray}\label{0314.3}
&&\\
&&I_3=\sum_{i,j\neq k}\bbw_1^{\intercal}\bbx_i\bbw_1^{\intercal}\bbx_j\bbe^{\intercal}_i\bbA^{-1}\bbe_j\non
&&=\sum_{i,j\neq k}\bbw_1^{\intercal}\bbx_i\bbw_1^{\intercal}\bbx_j\bbe^{\intercal}_i\bbA_k^{-1}\bbe_j+\sum_{i,j\neq k}\bbw_1^{\intercal}\bbx_i\bbw_1^{\intercal}\bbx_j\bbe^{\intercal}_i\frac{\bbA_{(k)}^{-1}(\bbe_k\bbx_k^{\intercal}\tilde{\Sigma}_1\bbx_k\bbe_k^{\intercal}+\bbe_k\bbx_k^{\intercal}\tilde{\Sigma}_1\bbX_k)\bbA_{(k)}^{-1}}{na_k}\bbe_j\non
&&=\bbw_1^{\intercal}\bbX_k\bbA_k^{-1}\bbX_k^{\intercal}\bbw_1+\frac{\bbw_1^{\intercal}\bbX_k\bbA_{(k)}^{-1}(\bbe_k\bbx_k^{\intercal}\tilde{\Sigma}_1\bbx_k\bbe_k^{\intercal}+\bbe_k\bbx_k^{\intercal}\tilde{\Sigma}_1\bbX_k)\bbA_{(k)}^{-1}\bbX_k^{\intercal}\bbw_1}{na_k}\nonumber
.\end{eqnarray}
Consider $(\mathbb{E}_k-\mathbb{E}_{k-1})(I_3-\bbw_1^{\intercal}\bbX_k\bbA_k^{-1}\bbX_k^{\intercal}\bbw_1)I(F_d)$ next.

 We claim that
\begin{eqnarray}\label{1102.1}
\frac{\bbw_1^{\intercal}\bbX_k\bbA_{(k)}^{-1}(\bbe_k\bbx_k^{\intercal}\tilde{\Sigma}_1\bbx_k\bbe_k^{\intercal}+\bbe_k\bbx_k^{\intercal}\tilde{\Sigma}_1\bbX_k)\bbA_{(k)}^{-1}\bbX_k^{\intercal}\bbw_1}{na_k}
\end{eqnarray}
is negligible. Let $\bbB_k=\tilde{\Sigma}_1\bbX_k\bbA_k^{-1}\bbX_k^{\intercal}\bbw_1\bbw_1^{\intercal}\bbX_k\bbA_k^{-1}\bbX_k^{\intercal}\tilde{\Sigma}_1$.
Indeed, by (\ref{1101.11}) and (\ref{1101.3})-(\ref{1101.4h}) we have
$\bbA_{(k)}^{-1}=\bbA_k^{-1}+\frac{1}{n}\bbA_k^{-1}\bbX_k^{\intercal}\tilde{\Sigma}_1\bbx_k\bbe_k^{\intercal}\bbA_k^{-1}.$
This, together with (\ref{1101.3}), (\ref{1101.4}) and (\ref{1101.4h}) implies that
\begin{eqnarray*}
(\ref{1102.1})&=&\frac{\bbw_1^{\intercal}\bbX_k\bbA_k^{-1}\bbX_k^{\intercal}\tilde{\Sigma}_1\bbx_k\bbe_k^{\intercal}\bbA_k^{-1}\bbe_k\bbx_k^{\intercal}\tilde{\Sigma}_1\bbX_k\bbA_k^{-1}\bbX_k^{\intercal}\bbw_1}{n^2a_k}
=\frac{\bbx_k^{\intercal}\bbB_k\bbx_k}{n^2a_k}.
\end{eqnarray*}
It follows from (\ref{0314.3}) and (\ref{1101.3})-(\ref{1101.4h}) that
$$\sum_{k=1}^n(\mathbb{E}_k-\mathbb{E}_{k-1})(I_3-\bbw_1^{\intercal}\bbX_k\bbA_k^{-1}\bbX_k^{\intercal}\bbw_1)I(F_d)=\sum_{k=1}^n(\mathbb{E}_k-\mathbb{E}_{k-1})\frac{\bbx_k^{\intercal}\bbB_k\bbx_k}{n^2a_k}I(F_d^{(k)})+o_p(n^{-2}).$$
 Considering the second moment of the above equation, by Lemma 8.10 of \cite{BS06} we have
\begin{eqnarray}\label{1102.2}
&&\sum_{k=1}^n\mathbb{E}|(\mathbb{E}_k-\mathbb{E}_{k-1})\frac{\bbx_k^{\intercal}\bbB_k\bbx_k}{n^2a_k}|^2I(F^{(k)}_d)\le \frac{4}{n^4}\sum_{k=1}^n\mathbb{E}|\bbx_k^{\intercal}\bbB_k\bbx_k|^2I(F_d^{(k)})\\
&&\le \frac{8}{n^4}\sum_{k=1}^n\mathbb{E}|\bbx_k^{\intercal}\bbB_k\bbx_k-tr\bbB_k|^2I(F_d^{(k)})+\frac{8}{n^4}\sum_{k=1}^n\mathbb{E}|tr\bbB_k|^2I(F_d^{(k)})\non
&&\le \frac{Cp^2}{n\theta^2}\ll N,\nonumber
\end{eqnarray}
where we used the inequality
$$tr\bbB_k\le \bbX_k\bbA_k^{-1}\bbX_k^{\intercal}\tilde{\Sigma}_1^2\bbX_k\bbA_k^{-1}\bbX_k^{\intercal}I(F^{(k)}_d)=O(\frac{p^2}{\theta^2}).$$
We conclude that
$$\frac{1}{n}\sum_{k=1}^n(\mathbb{E}_k-\mathbb{E}_{k-1})(I_3-\bbw_1^{\intercal}\bbX_k\bbA_k^{-1}\bbX_k^{\intercal}\bbw_1)I(F_d)=o_p(\frac{1}{\sqrt n}),$$
which is negligible.

Next we consider $I_1$ and $I_2$. It follows from (\ref{0314.2}) and (\ref{0330.1}) that
\begin{eqnarray}\label{0315.1}
&&\frac{1}{\sqrt{n}}\sum_{k=1}^n(\mathbb{E}_{k}-\mathbb{E}_{k-1})(I_1+2I_2)I(F_d)\\
&&=\frac{2}{\sqrt{n}}\sum_{k=1}^n(\mathbb{E}_{k}-\mathbb{E}_{k-1})\big(\frac{(\bbw_1\bbx_k)^2}{2a_k}+\frac{\bbw_1^{\intercal}\bbX_k\bbA_k^{-1}\bbX_k^{\intercal}\tilde{\Sigma}_1\bbx_k\bbx_k^{\intercal}\bbw_1}{na_k}\big)I(F_d).\nonumber
\end{eqnarray}
We claim that the second term of (\ref{0315.1}) is negligible. Actually, similar to (\ref{1102.2}), it is easy to show that
$$\sum_{k=1}^n(\mathbb{E}_{k}-\mathbb{E}_{k-1})\frac{\bbw_1^{\intercal}\bbX_k\bbA_k^{-1}\bbX_k^{\intercal}\tilde{\Sigma}_1\bbx_k\bbx_k^{\intercal}\bbw_1}{na_k}I(F_d)=o_p(\sqrt n)$$
Therefore, the leading term of (\ref{0315.1}) is
\begin{eqnarray*}
&&\frac{1}{\sqrt{n}}\sum_{k=1}^n(\mathbb{E}_{k}-\mathbb{E}_{k-1})\frac{(\bbw_1^{\intercal}\bbx_k)^2}{a_k}I(F_d)\non
&=&\frac{1}{\sqrt{n}}\sum_{k=1}^n(\mathbb{E}_{k}-\mathbb{E}_{k-1})\frac{(1-a_k)(\bbw_1^{\intercal}\bbx_k)^2}{a_k}I(F_d)+\frac{1}{\sqrt{n}}\sum_{k=1}^n(\mathbb{E}_{k}-\mathbb{E}_{k-1})(\bbw_1^{\intercal}\bbx_k)^2I(F_d).
\end{eqnarray*}

Similar to (\ref{1102.2}), by (\ref{0321.1}) we can show that
$$\frac{1}{\sqrt{n}}\sum_{k=1}^n(\mathbb{E}_{k}-\mathbb{E}_{k-1})\frac{(1-a_k)(\bbw_1^{\intercal}\bbx_k)^2}{a_k}I(F_d)=o_p(1).$$
It suffices to show CLT for
\begin{eqnarray}\label{1127.3}
\frac{1}{\sqrt{n}}\sum_{k=1}^n(\mathbb{E}_{k}-\mathbb{E}_{k-1})(\bbw_1^{\intercal}\bbx_k)^2=\frac{1}{\sqrt{n}}\sum_{k=1}^n\left[(\bbw_1^{\intercal}\bbx_k)^2-1\right].
\end{eqnarray}
By the CLT for the sum of i.i.d. variables, we conclude that
$$\frac{1}{\sqrt{n}\sigma}\sum_{k=1}^n(\mathbb{E}_{k}-\mathbb{E}_{k-1})(\bbw_1^{\intercal}\bbx_k)^2\stackrel{D}{\rightarrow} N(0,\sigma^2),$$
where
\begin{eqnarray}\label{0411.1}
\sigma^2&=&\frac{1}{n}\mathbb{E}\left[(\bbw_1^{\intercal}\bbx_k)^2-1\right]^2=\frac{\sum_{i=1}^{p+l}\gamma_{4i}\bbw_{1i}^4+3\sum_{i\neq j}^{p+l}\bbw_{1i}^2\bbw_{1j}^2-1}{n}\\
&=&\sum_{i=1}^{p+l}(\gamma_{4i}-3)\bbw_{1i}^4+2.\nonumber
\end{eqnarray}

\subsection{ Calculation of the Mean}\label{calculatemean}

This section is to calculate the expectation of $\frac{1}{n} \bbw_1^{\intercal}\bbX\bbA^{-1}\bbX^{\intercal}\bbw_1I(F_d)$. The strategy is to prove that
\begin{eqnarray}\label{1129.12}
\sqrt n\mathbb{E}\left[\frac{1}{n} \bbw_1^{\intercal}\bbX^0\bbA^{-1}(\bbX^0)^{\intercal}\bbw_1I(F_d)+\tilde m_{\theta}(1)\right]\rightarrow 0,
\end{eqnarray}
 and
\begin{eqnarray}\label{1129.13}
\frac{1}{\sqrt n} \mathbb{E}\left[\bbw_1^{\intercal}\bbX\bbA^{-1}\bbX^{\intercal}\bbw_1I(F_d)-\bbw_1^{\intercal}\bbX^0\bbA^{-1}(\bbX^0)^{\intercal}\bbw_1I(F_d)\right]\rightarrow 0,
\end{eqnarray}
 where $\bbX^0=(\bbx_1^{0},...,\bbx_n^{0})$ is $(p+l)\times n$ matrix with i.i.d. standard Gaussian random variables. As before, we omit $I(F_d)$ in the following proof.

  We prove (\ref{1129.13}) first by the Lindeberg's strategy. Define
  $$\bbZ^1_{k}=\sum_{i=1}^k\bbx_i\bbe_i^{\intercal}+\sum_{i=k+1}^n\bbx^{0}_i\bbe_i^{\intercal}, \ \bbZ^0_{k}=\sum_{i=1}^{k-1}\bbx_i\bbe_i^{\intercal}+\sum_{i=k}^n\bbx^{0}_i\bbe_i^{\intercal}, $$  $$\bbZ_{k}=\sum_{i=1}^{k-1}\bbx_i\bbe_i^{\intercal}+\sum_{i=k+1}^N\bbx^{0}_i\bbe_i^{\intercal},\ \hat \bbA^1_k=\bbI-\frac{1}{n}(\bbZ^1_k)^{\intercal}\tilde{\Sigma}_1\bbZ^1_k,$$ $$\hat \bbA^0_k=\bbI-\frac{1}{n}(\bbZ^0_k)^{\intercal}\tilde{\Sigma}_1\bbZ^0_k \ \ \   \text{and} \ \ \  \hat \bbA_k=\bbI-\frac{1}{n}\bbZ_k^{\intercal}\tilde{\Sigma}_1\bbZ_k.$$
   Then we have
 $\bbX=\bbZ^1_N$, $\bbX^0=\bbZ^0_1$, $\bbZ^0_{k+1}=\bbZ^1_k$. It follows that
\begin{eqnarray}\label{1129.4}
&&\frac{1}{\sqrt n} \mathbb{E}\left[\bbw_1^{\intercal}\bbX\bbA^{-1}\bbX^{\intercal}\bbw_1-\bbw_1^{\intercal}\bbX^0\bbA^{-1}(\bbX^0)^{\intercal}\bbw_1\right]\\
&&=\frac{1}{\sqrt n}\sum_{k=1}^n\mathbb{E}\left[\bbw_1^{\intercal}\bbZ^1_k(\hat \bbA^1_k)^{-1}(\bbZ^1_k)^{\intercal}\bbw_1-\bbw_1^{\intercal}\bbZ^0_k(\hat \bbA^0_k)^{-1}(\bbZ^0_k)^{\intercal}\bbw_1\right]\non
&&=\frac{1}{\sqrt n}\sum_{k=1}^n\mathbb{E}\left[\bbw_1^{\intercal}\bbZ^1_k(\hat \bbA^1_k)^{-1}(\bbZ^1_k)^{\intercal}\bbw_1-\bbw_1^{\intercal}\bbZ_k\hat \bbA_k^{-1}\bbZ_k^{\intercal}\bbw_1\right]\non
&&+\frac{1}{\sqrt n}\sum_{k=1}^n\mathbb{E}\left[\bbw_1^{\intercal}\bbZ_k\hat \bbA_k^{-1}\bbZ_k^{\intercal}\bbw_1-\bbw_1^{\intercal}\bbZ^0_k(\hat \bbA^0_k)^{-1}(\bbZ^0_k)^{\intercal}\bbw_1\right].\nonumber
\end{eqnarray}
For any $k$, similar to the expansions from (\ref{0314.5})-(\ref{1102.1}), we can get
\begin{eqnarray}\label{1129.5}
&&\mathbb{E}\left[\bbw_1^{\intercal}\bbZ^1_k(\hat \bbA^1_k)^{-1}(\bbZ^1_k)^{\intercal}\bbw_1-\bbw_1^{\intercal}\bbZ_k\hat \bbA_k^{-1}\bbZ_k^{\intercal}\bbw_1\right]\\
&&=\mathbb{E}\left[\frac{(\bbw_1\bbx_k)^2}{\hat a_k}+\frac{2\bbw_1^{\intercal}\bbZ_k\hat\bbA_k^{-1}\bbZ_k^{\intercal}\tilde{\Sigma}_1\bbx_k\bbx^{\intercal}_k\bbw_1}{n\hat a_k}+\frac{\bbx_k^{\intercal}\hat\bbB_k\bbx_k}{n^2\hat a_k}\right],\nonumber
\end{eqnarray}
where $\hat \bbB_k=\tilde{\Sigma}_1\bbZ_k\hat \bbA_k^{-1}\bbZ_k^{\intercal}\bbw_1\bbw_1^{\intercal}\bbZ_k\hat \bbA_k^{-1}\bbZ_k^{\intercal}\tilde{\Sigma}_1$ and $\hat a_k=1-\frac{1}{n^2}\bbx_k^{\intercal}\tilde{\Sigma}_1\bbZ_k\hat \bbA_k^{-1}\bbZ_k^{\intercal}\tilde{\Sigma}_1\bbx_k$. Let $\bar a_k=1-\frac{1}{n^2}tr\tilde{\Sigma}_1\bbZ_k\hat \bbA_k^{-1}\bbZ_k^{\intercal}\tilde{\Sigma}_1$, $\tau_k=\hat a_k-\bar a_k$. Then we have
\begin{eqnarray}\label{1129.6}
\frac{1}{\hat a_k}=\frac{1}{\bar a_k}-\frac{\tau_k}{\hat a_k\bar a_k}.
\end{eqnarray}
By Lemma 8.10 of \cite{BS06}, we conclude that
\begin{eqnarray}\label{1129.7}
&&\mathbb{E}|\tau_k|^2=\mathbb{E}|\frac{1}{n^2}\bbx_k^{\intercal}\tilde{\Sigma}_1\bbZ_k\hat \bbA_k^{-1}\bbZ_k^{\intercal}\tilde{\Sigma}_1\bbx_k-\frac{1}{n^2}tr\tilde{\Sigma}_1\bbZ_k\hat \bbA_k^{-1}\bbZ_k^{\intercal}\tilde{\Sigma}_1|^2\\
&&\le \frac{C}{n^4}tr(\tilde{\Sigma}_1\bbZ_k\hat \bbA_k^{-1}\bbZ_k^{\intercal}\tilde{\Sigma}_1)^2=O(\frac{d^2}{p}).\nonumber
\end{eqnarray}
Consider the first term at the right hand side of (\ref{1129.5}). It follows from (\ref{1129.6}), (\ref{1129.7}) and Holder's inequality  that
\begin{equation}\label{1129.8}
|\mathbb{E}(\frac{(\bbw_1\bbx_k)^2}{\hat a_k}-\frac{(\bbw_1\bbx_k)^2}{\bar a_k})|=|\mathbb{E}\frac{(\bbw_1\bbx_k)^2\tau_k}{\hat a_k\bar a_k}|
\le C\sqrt{\mathbb{E}(\bbw_1\bbx_k)^4}\sqrt{\mathbb{E}\tau_k^2}=O(\frac{d}{\sqrt p}).
\end{equation}
Thus we conclude that
$$\mathbb{E}\frac{(\bbw_1\bbx_k)^2}{\hat a_k}=\mathbb{E}\frac{(\bbw_1\bbx_k)^2}{\bar a_k}+O(\frac{d}{\sqrt p})=\mathbb{E}\frac{1}{\bar a_k}+o(\frac{1}{\sqrt n}).$$
Moreover, a similar approach can be applied to the other terms at the right hand side of (\ref{1129.5})  and thus we have
\begin{eqnarray}\label{1129.9}
&&\frac{1}{\sqrt n}\sum_{k=1}^n\mathbb{E}\left[\bbw_1^{\intercal}\bbZ^1_k(\hat \bbA^1_k)^{-1}(\bbZ^1_k)^{\intercal}\bbw_1-\bbw_1^{\intercal}\bbZ_k\hat \bbA_k^{-1}\bbZ_k^{\intercal}\bbw_1\right]\\
&&=\frac{1}{\sqrt n}\sum_{k=1}^n\mathbb{E}\left[\frac{1}{\bar a_k}+\frac{2\bbw_1^{\intercal}\bbZ_k\hat\bbA_k^{-1}\bbZ_k^{\intercal}\tilde{\Sigma}_1\bbw_1}{n\bar a_k}+\frac{tr\hat\bbB_k}{n^2\bar a_k}\right]+o(1).\nonumber
\end{eqnarray}
By the same arguments above, we can also get
\begin{eqnarray}\label{1129.10}
&&\frac{1}{\sqrt n}\sum_{k=1}^n\mathbb{E}\left[\bbw_1^{\intercal}\bbZ_k\hat \bbA_k^{-1}\bbZ_k^{\intercal}\bbw_1-\bbw_1^{\intercal}\bbZ^0_k(\hat \bbA^0_k)^{-1}(\bbZ^0_k)^{\intercal}\bbw_1\right]\\
&&=-\frac{1}{\sqrt n}\sum_{k=1}^n\mathbb{E}\left[\frac{1}{\bar a_k}+\frac{2\bbw_1^{\intercal}\bbZ_k\hat\bbA_k^{-1}\bbZ_k^{\intercal}\tilde{\Sigma}_1\bbw_1}{n\bar a_k}+\frac{tr\hat\bbB_k}{n^2\bar a_k}\right]+o(1).\nonumber
\end{eqnarray}
Combining (\ref{1129.4}), (\ref{1129.9}) and (\ref{1129.10}), the equation (\ref{1129.13}) holds.

We next prove (\ref{1129.12}).
To simplify notation, we use $\bbX$ for $\bbX^0$ and hence assume that $\bbX$ follows standard normal distribution. By $\bbw_1^{\intercal}\bbU_2^{\intercal}=0$,  we conclude that $\bbw_1^{\intercal}\bbX$ is independent of $\bbA$ and hence $\frac{1}{n}\mathbb{E} \bbw_1^{\intercal}\bbX\bbA^{-1}\bbX^{\intercal}\bbw_1=\frac{1}{n}\mathbb{E}tr\bbA^{-1}$. By (6.2.4) of \cite{BS06}(or Lemma 3.1 of \cite{BPWZ2014b}), we have
$$\frac{1}{n}\mathbb{E}tr\bbA^{-1}=\mathbb{E}\frac{1}{1+\bbr_1^{\intercal}\underline{\bbA}^{-1}_1\bbr_1},$$
where we denote $\underline{\bbA}=\tilde{\Sigma}_1^{1/2}\bbX\bbX^{\intercal}\tilde{\Sigma}_1^{1/2}-\bbI$, $\bbr_i=\frac{1}{\sqrt N}\tilde{\Sigma}_1^{1/2}\bbx_i$ and $\underline{\bbA}_j=\sum_{i\neq j}\bbr_i\bbr_i^{\intercal}-\bbI$. By Lemma 8.10 of \cite{BS06}, we have
\begin{eqnarray}\label{1128.1}
\mathbb{E}|\bbr_1^{\intercal}\underline{\bbA}^{-1}_1\bbr_1-\frac{1}{\theta N}tr\underline{\bbA}^{-1}_1\Sigma_1|\le \frac{C}{n^2}tr\tilde{\Sigma}_1^2=o(M^{-1}),
\end{eqnarray}
which concludes that
$\mathbb{E}\frac{1}{1+\bbr_1^{\intercal}\underline{\bbA}^{-1}_1\bbr_1}=\mathbb{E}\frac{1}{1+\frac{1}{\theta N}tr\underline{\bbA}^{-1}_1\Sigma_1}+o(n^{-1/2}).$
Moreover,
\begin{eqnarray}\label{1128.2}
\mathbb{E}|\frac{1}{1+\frac{1}{\theta N}tr\underline{\bbA}^{-1}_1\Sigma_1}-\frac{1}{1+\frac{1}{\theta N}\mathbb{E}tr\underline{\bbA}^{-1}_1\Sigma_1}|^2
&\le& \frac{C}{n^2}\mathbb{E}|tr\underline{\bbA}^{-1}_1\Sigma_1-\mathbb{E}tr\underline{\bbA}^{-1}_1\Sigma_1|^2\non
&\le& \frac{C}{n}\mathbb{E}|\beta_{12}\bbr_2^{\intercal}\underline{\bbA}^{-2}_{12}\bbr_2|^2=o(n^{-1}).
\end{eqnarray}
Hence $\mathbb{E}\frac{1}{1+\frac{1}{\theta N}tr\underline{\bbA}^{-1}_1\Sigma_1}=\frac{1}{1+\frac{1}{\theta N}\mathbb{E}tr\underline{\bbA}^{-1}_1\Sigma_1}+o(n^{-1/2}).$
Define $\beta_i=\frac{1}{1+\bbr_i^{\intercal}\underline{\bbA}_i^{-1}\bbr_i}$, $b_i=\frac{1}{1+\frac{1}{n\theta}\mathbb{E}tr\Sigma_1\underline{\bbA}_i^{-1}}$, and $\alpha_i=\bbr_i^{\intercal}\underline{\bbA}_i^{-1}\bbr_i-\frac{1}{n\theta}tr\Sigma_1\underline{\bbA}_i^{-1}$. By the equality that
\[
\underline{\bbA}_1+\bbI-b(\theta)\tilde \Sigma_1=\sum_{i\neq 1}\bbr_i\bbr_i^{\intercal}-b(\theta)\tilde \Sigma_1,
\]
we have
\begin{eqnarray}\label{0330.2}
\underline{\bbA}^{-1}_1=-(\bbI-b_1(\theta)\tilde \Sigma_1)^{-1}+b_1(z)A(\theta)+B(\theta)+C(\theta),
\end{eqnarray}
where
$$A(\theta)=\sum_{i\neq 1}(\bbI-b_1(\theta)\tilde \Sigma_1)^{-1}(\bbr_i\bbr_i^{\intercal}-\frac{1}{n\theta}\Sigma_1)\underline{\bbA}_i^{-1},$$
$$B(\theta)=\sum_{i\neq 1}(\beta_i-b_1)(\bbI-b_1(\theta)\tilde \Sigma_1)^{-1}\bbr_i\bbr_i^{\intercal}\underline{\bbA}_i^{-1},$$
$$C(\theta)=n^{-1}b_1(\bbI-b_1(\theta)\Sigma_1)^{-1}\tilde \Sigma_1\sum_{i\neq 1}(\underline{\bbA}^{-1}_1-\underline{\bbA}_{1i}^{-1}).$$
For $A(\theta)$, similar to (\ref{1128.1}) we have
\begin{equation}\label{1128.3}\small
\frac{1}{n}\mathbb{E}|trA(\theta)\tilde \Sigma_1|\le \frac{1}{n}\sum_{i\neq 2}\mathbb{E}|\bbr_i^{\intercal}\underline{\bbA}_i^{-1}\tilde \Sigma_1(\bbI-b_1(\theta)\tilde \Sigma_1)^{-1}\bbr_i-\frac{1}{n\theta}tr(\Sigma_1\underline{\bbA}_i^{-1}\tilde \Sigma_1(\bbI-b_1(\theta)\tilde \Sigma_1)^{-1})|=o(M^{-1}).
\end{equation}
Similar to the previous inequalities (\ref{1128.1})-(\ref{1128.2}) or as in Chapter 9 of \cite{BS06},  we can also show that $B(\theta)$ and $C(\theta)$ are negligible. Hence we get
\begin{equation}\label{0331.1}
\frac{1}{n}\mathbb{E}tr\underline{\bbA}^{-1}_1\tilde \Sigma_1=-\frac{1}{n}tr(\bbI-b_1(\theta)\tilde \Sigma_1)^{-1}\tilde \Sigma_1+o(n^{-1/2}),
\end{equation}
which implies that
\begin{eqnarray}\label{0901.1}
\frac{1}{n}\mathbb{E}tr\bbA^{-1}=\frac{1}{1-\frac{1}{n}tr(\bbI-\frac{1}{n}(\mathbb{E}tr\bbA^{-1})\tilde \Sigma_1)^{-1}\tilde \Sigma_1}+o(n^{-1/2}),
\end{eqnarray}
By the Steiltjes transform of the limit of the ESD of any sample covariance matrix, there exists only one $\tilde m_{\theta}(z)$ such that (One can also refer to (1.6) of  \cite{BPWZ2014b} or (6.12)-(6.15) of \cite{BS06})
\begin{eqnarray}\label{0331.2}
\tilde m_{\theta}(z)=-\frac{1}{z-\frac{1}{n}tr(\bbI+\tilde m_{\theta}(z)\tilde \Sigma_1)^{-1}\tilde \Sigma_1}, \ \ z\in \mathbb{C}^+.
\end{eqnarray}
Consider the difference between (\ref{0901.1})-(\ref{0331.2}) and denote $\delta=\frac{1}{n}\mathbb{E}tr\bbA^{-1}+\tilde m_{\theta}(1)$. It is easy to conclude that
$$\delta(1+\frac{\frac{1}{n}tr\left[(\bbI-\frac{1}{n}(\mathbb{E}tr\bbA^{-1})\tilde \Sigma_1)^{-1}\tilde \Sigma_1(\bbI+\tilde m_{\theta}(1)\tilde \Sigma_1)^{-1}\tilde \Sigma_1\right]}{(1-\frac{1}{n}tr(\bbI-\frac{1}{n}(\mathbb{E}tr\bbA^{-1})\tilde \Sigma_1)^{-1}\tilde \Sigma_1)(1-\frac{1}{n}tr(\bbI+\tilde m_{\theta}(1)\tilde \Sigma_1)^{-1}\tilde \Sigma_1)})=o(n^{-1/2}).$$
Together with the fact that $\|\tilde \Sigma_1\|=O(\theta^{-1})$, it follows that $\delta=o(1/\sqrt n)$. Therefore, we have shown that
\begin{eqnarray}\label{0916.3}
\sqrt n(\frac{1}{n}\mathbb{E}tr\bbA^{-1}+\tilde m_{\theta}(1))\rightarrow 0. \qed
\end{eqnarray}

\newpage
\title{\textbf{Supplement to ``Limiting Laws for Divergent Spiked Eigenvalues and Largest Non-spiked Eigenvalue of Sample Covariance Matrices''}}

This note summarizes the supplementary materials to the paper ``Limiting Laws for Divergent Spiked Eigenvalues and Largest Non-spiked Eigenvalue of Sample Covariance Matrices''.
We first briefly discuss the quantities  $\gamma_+$ and $\sigma_n$ defined in Section \ref{sec5h} and then provide detailed proofs of  the main theorems and some technical results given in the paper. More specifically, we prove in detail here Theorems \ref{0318-1},  \ref{0408-2}, \ref{0831-1}, \ref{0503-1}, \ref{0606-2}, Lemma  \ref{1101-1x} and Corollary \ref{1130-1}.

\section{Discussion on $\gamma_+$ and $\sigma_n$}

 Below we discuss the unknown parameters $\gamma_+$ and $\sigma_n$. In order to find an upper bound of $\lambda_{K+1}$, by (\ref{0216.5}), a key step is to estimate $\sigma_n$ and $\gamma_+$. By (3) and (11) of \cite{K2007}, we have
$$\sigma_n=(\frac{1}{2}\frac{\partial^3 f(z)}{\partial z^3}|_{z\rightarrow \bbd})^{1/3},$$
where
$$f(z)=-\gamma_+ z+\log(z)-\frac{p-K}{n}\int\log(1-z\lambda)dF_{\bold\Lambda_P}(\lambda)+C, \ C \ \text{is a constant.}$$
It is straightforward to get
\begin{eqnarray}\label{0604.4}
\frac{\partial f(s)}{\partial s}=-\gamma_++\frac{1}{s}+\frac{p-K}{n}\int \frac{\lambda dF_{\bold\Lambda_P}(\lambda)}{1-\lambda s}.
\end{eqnarray}
Let $t=-m_{\Sigma_1}(z)$. Then by the equality that
$$z=-\frac{1}{t}+\frac{p-K}{n}\int \frac{\lambda dF_{\bold\Lambda_P}(\lambda)}{1+\lambda t},$$
we have $\frac{\partial f(t)}{\partial t}=-\gamma_+-z$. Therefore, $\frac{\partial^3 f(t)}{\partial t^3}=-\frac{\partial^2 z}{\partial t^2}$. Recall the definition of $t$,  $$t=-m_{\Sigma_1}(z)=-\int \frac{dF_0(x)}{x-z},$$
where $F_0(x)$ is the c.d.f. determined by $m_{\Sigma_1}(z)$. We have the following two equations:
\begin{equation}\label{0604.5}
1=-\frac{\partial z}{\partial t}\int \frac{dF_0(x)}{(x-z)^2}, \ \ \  0=-\frac{\partial^2 z}{\partial t^2}\int \frac{dF_0(x)}{(x-z)^2}+2(\frac{\partial z}{\partial t})^2\int \frac{dF_0(x)}{(x-z)^3}.
\end{equation}
It follows from (\ref{0604.3}), (\ref{0604.4})-(\ref{0604.5}) that
\begin{eqnarray}\label{0604.6h}
\sigma_n=(-\lim_{z\rightarrow \gamma_+^+}\frac{\int \frac{dF_0(x)}{(x-z)^3}}{(\int \frac{dF_0(x)}{(x-z)^2})^3})^{1/3}
\end{eqnarray}
 By the singular value inequality or interlacing inequality, we have
$$\lambda_{n^{1/6}}\ge \nu_{n^{1/6}+K}.$$
By Theorem 3.14 of \cite{KY14}, we have
$$|\nu_{n^{1/6}+K}-\gamma_{n^{1/6}+K}|\le n^{-2/3},$$
with high probability, where
$$\frac{i}{n}=\int_{\gamma_{i}}^{\gamma_+} dF_0(x).$$
By Lemmas 2.3 and 2.5 of \cite{BPZ2014a}, we have $\frac{dF_0(x)}{dx}\sim \sqrt{\gamma_+-x}$, then
$$\gamma_+-\nu_{1/6+K}\sim n^{-5/9}.$$
Therefore
$\gamma_+-\nu_{n^{1/6}+K}\le \frac{\log n}{2} \times n^{-5/9}$ with high probability. Therefore, together with Theorem \ref{0606-2}, with high probability
\begin{eqnarray}\label{h0604.7h}
\lambda_{K+1}\le \lambda_{n^{1/6}}+ \log n \times n^{-5/9}.
\end{eqnarray}
 \section{Proof of Theorem \ref{0318-1}}

Below, we consider $i=1,...,K$. Note that the non-zero eigenvalues of $\Gamma\bbX\bbX^T\Gamma^T$ are equal to those of $\bbU\bbX\bbX^T\bbU^T\Lambda$. By Weyl's inequality, we have
$$|\sigma_i(\Lambda^{1/2}\bbU\bbX)-\sigma_i(\left(
 \begin{matrix}
   \Lambda_S^{1/2} & 0\\
  0 &0
  \end{matrix}
  \right)\bbU\bbX)|\le \|\left(
 \begin{matrix}
   0 & 0\\
  0 &\Lambda_P^{1/2}
  \end{matrix}
  \right)\bbU\bbX\|,$$
where $\sigma_i(\bbA)$ is the $i$-th largest singular value of $\bbA$. By Theorem 1 of \cite{BG2012}, under Assumption \ref{0829-1}(ii), with probability tending to 1, we have $\|\frac{1}{n}\bbU_2\bbX\bbX^T\bbU_2^T\Lambda_P\|\le \|\frac{1}{n}\bbU_2\bbX\bbX^T\bbU_2^T\|\|\Lambda_P\|\le \|\frac{1}{n}\bbX\bbX^T\|\|\Lambda_P\|\le\frac{2Cp}{n}$. Define $\bbB=\left(
 \begin{matrix}
   \Lambda_S & 0\\
  0 &0
  \end{matrix}
  \right)$. By assumption 3, we have \begin{eqnarray}\label{0318.2}
\frac{\lambda_i(\frac{1}{n}\bbU\bbX\bbX^T\bbU^T\Lambda)-\lambda_i(\frac{1}{n}\bbU\bbX\bbX^T\bbU^T\bbB)}{\mu_i}=O_p(d_i).
\end{eqnarray} Moreover, it is easy to see that the non-zero eigenvalues of $\lambda_i(\frac{1}{n}\bbU\bbX\bbX^T\bbU^T\bbB)$ are the same as those of the  $K\times K$ block $\bbC=\frac{1}{n}\Lambda^{1/2}_S\bbU_1\bbX\bbX^T\bbU^T_1\Lambda^{1/2}_S$, where $\bbU_1$ is the first $K$ rows of $\bbU$. By Theorem 7.1 of \cite{BY08} and Chebyshev's inequality, we can show that $\|\frac{1}{n}\bbU_1\bbX\bbX^T\bbU^T_1-\bbI_K\|_{\infty}=O_p(\frac{K}{\sqrt{n}})$. Moreover, the determinant for calculating the eigenvalue $\lambda_i(\frac{1}{n}\Lambda^{1/2}_S\bbU_1\bbX\bbX^T\bbU^T_1\Lambda^{1/2}_S)$ is equivalent to
\begin{eqnarray}\label{0324.1}
\det(\frac{1}{n}\bbU_1\bbX\bbX^T\bbU^T_1-\lambda_i(\bbC)\Lambda_S^{-1})=0.
\end{eqnarray}
By the Leibniz's formula for the determinant, it is easy to conclude that $\frac{\lambda_i(\bbC)}{\mu_i}-1=O_p(\frac{K^4}{n})$ uniformly for all $i=1,...,K$. Combining with (\ref{0318.2}), we conclude that $$\frac{\lambda_i(\frac{1}{n}\bbU\bbX\bbX^T\bbU^T\Lambda)-\mu_i}{\mu_i}=O_p(\frac{K^4}{n}+d_i)$$ uniformly for all $i=1,...,K$.

\section{Proof of Theorem \ref{0408-2}}\label{Appc}

\subsection{Outline of The Proof }\label{Appa}

 If $\lambda_i$ is the  spiked  eigenvalue of $\bbS_n$, then by the determinantal equation (\ref{0912.2h}) below we conclude that $\lambda_i$ satisfies the following equation
\begin{equation}\label{out1}
\det(\Lambda^{-1}_{\bbS}-\frac{1}{n}\bbU_1\bbX(\lambda_i\bbI-\frac{1}{n}\bbX^{\intercal}\bbU^{\intercal}_2\Lambda_P\bbU_2\bbX)^{-1}\bbX^{\intercal}\bbU^{\intercal}_1)=0.
\end{equation}
We will prove that the diagonal entries of $\frac{1}{n}\bbU_1\bbX(\lambda_i\bbI-\frac{1}{n}\bbX^{\intercal}\bbU^{\intercal}_2\Lambda_P\bbU_2\bbX)^{-1}\bbX^{\intercal}\bbU^{\intercal}_1$ dominate the determinant above. Roughly speaking, by ignoring the negligible terms we can get the following equation
\begin{equation}\label{out1}
\mu_i^{-1}-\frac{1}{n}\bbu_i^{\intercal}\bbX(\lambda_i\bbI-\frac{1}{n}\bbX^{\intercal}\bbU^{\intercal}_2\Lambda_P\bbU_2\bbX)^{-1}\bbX^{\intercal}\bbu_i=0.
\end{equation}
We can further get
$$\mu_i^{-1}-\frac{1}{n}\bbu_i^{\intercal}\bbX(\theta_i\bbI-\frac{1}{n}\bbX^{\intercal}\bbU^{\intercal}_2\Lambda_P\bbU_2\bbX)^{-1}\bbX^{\intercal}\bbu_i
\approx(\lambda_i-\theta_i)\frac{1}{n}\bbu_i^{\intercal}\bbX(\theta_i\bbI-\frac{1}{n}\bbX^{\intercal}\bbU^{\intercal}_2\Lambda_P\bbU_2\bbX)^{-2}\bbX^{\intercal}\bbu_i.$$
Therefore the CLT of $(\lambda_i-\theta_i)$ is determined by the asymptotic distribution of $\frac{1}{n}\bbu_i^{\intercal}\bbX(\theta_i\bbI-\frac{1}{n}\bbX^{\intercal}\bbU^{\intercal}_2\Lambda_P\bbU_2\bbX)^{-1}\bbX^{\intercal}\bbu_i$. Therefore we need to establish CLT of the random quadratic forms in Theorem \ref{1101-1}. Similarly, the correlation of $\lambda_i$ and $\lambda_j$ are also determined by $\frac{1}{n}\bbu_i^{\intercal}\bbX(\theta_i\bbI-\frac{1}{n}\bbX^{\intercal}\bbU^{\intercal}_2\Lambda_P\bbU_2\bbX)^{-1}\bbX^{\intercal}\bbu_i$ and $\frac{1}{n}\bbu_j^{\intercal}\bbX(\theta_j\bbI-\frac{1}{n}\bbX^{\intercal}\bbU^{\intercal}_2\Lambda_P\bbU_2\bbX)^{-1}\bbX^{\intercal}\bbu_j$.


\section*{Proof of Theorem \ref{0408-2} under Assumption \ref{0829-1b}}
This section is to prove a weaker version of Theorem \ref{0408-2} first. i.e. We assume that Assumption \ref{0829-1b} holds instead of Assumption \ref{0829-1a}. Assumption \ref{0829-1b} is then removed at Section \ref{0914-1} in the supplementary. Define $\bbB(x)=x\bbI-\frac{1}{n}\bbX^{\intercal}\bbU^{\intercal}_2\Lambda_P\bbU_2\bbX$.

First of all, we prove CLT for  a fixed i, $i\in \{1,...,K\}$. By the definition of $\lambda_i$, it solves the equation
$$\det(\lambda_i\bbI-\frac{1}{n}\Lambda^{1/2}\bbU\bbX\bbX^{\intercal}\bbU^{\intercal}\Lambda^{1/2})=0.$$
By the simple fact that $\det(\bbI-\bbC\bbD)=\det(\bbI-\bbD\bbC)$, we have
\begin{eqnarray}\label{1101.7}
\det(\lambda_i\bbI-\frac{1}{n}\bbX^{\intercal}\bbU^{\intercal}\Lambda\bbU\bbX)=0.
\end{eqnarray}
Recalling the notations above Assumption \ref{0829-1a}, (\ref{1101.7}) is equivalent to
\begin{eqnarray}\label{1101.8}
\det(\lambda_i\bbI-\frac{1}{n}\bbX^{\intercal}\bbU^{\intercal}_2\Lambda_P\bbU_2\bbX-\frac{1}{n}\bbX^{\intercal}\bbU^{\intercal}_1\Lambda_S\bbU_1\bbX)=0.
\end{eqnarray}
By Theorem \ref{0318-1}, $\lambda_i\bbI-\frac{1}{n}\bbX^{\intercal}\bbU^{\intercal}_2\Lambda_P\bbU_2\bbX$ is invertible with probability tending to 1. Hence with probability tending to 1, (\ref{1101.8}) is equivalent to
\begin{eqnarray}\label{1101.9}
\det(\bbI-\frac{1}{n}\bbX^{\intercal}\bbU^{\intercal}_1\Lambda_S\bbU_1\bbX(\lambda_i\bbI-\frac{1}{n}\bbX^{\intercal}\bbU^{\intercal}_2\Lambda_P\bbU_2\bbX)^{-1})=0.
\end{eqnarray}
Therefore, $\lambda_i$ satisfies the following equation
\begin{equation}\label{0604.1}
\det(\bbI-\frac{1}{n}\Lambda^{1/2}_S\bbU_1\bbX\bbB^{-1}(\lambda_i)\bbX^{\intercal}\bbU^{\intercal}_1\Lambda^{1/2}_S)=0.
\end{equation}
i.e.
\begin{equation}\label{0912.2h}
\det(\Lambda^{-1}_{\bbS}-\frac{1}{n}\bbU_1\bbX\bbB^{-1}(\lambda_i)\bbX^{\intercal}\bbU^{\intercal}_1)=0.
\end{equation}

 Recalling (\ref{1101.6}), we have
$$\tilde m_{\theta_i}(1)+\frac{\theta_i}{\mu_i}=0.$$
Since $\tilde m_{\theta_i}(x)$ is a increasing function of $x$ for $x\ge 1/2$($x\ge 1/2$ is outside the spectrum of  $\tilde m_{\theta_i}(x)$)  and $\|\frac{\bold\Sigma_1}{\mu_i}\|= O_p(d_i)$, we conclude that $\theta_i=\mu_i(1+O(d_i))$. We denote $\frac{\lambda_i-\theta_i}{\theta_i}$ by $\delta_i$. For convenience, we only prove the central limit theorem for $\lambda_1$ and the other eigenvalues can be handled similarly.
First of all, we have
\begin{equation}\label{0318.3}
\bbU_1\bbX\bbB^{-1}(\lambda_1)\bbX^{\intercal}\bbU^{\intercal}_1=\bbU_1\bbX\bbB^{-1}(\theta_1)\bbX^{\intercal}\bbU^{\intercal}_1-\delta_1\theta_1\bbU_1\bbX\bbB^{-1}(\lambda_1)\bbB^{-1}(\theta_1)\bbX^{\intercal}\bbU^{\intercal}_1.
\end{equation}
Hence (\ref{0912.2h}) can be rewritten as
\begin{equation}\label{1101.9}
\det(\theta_1\Lambda^{-1}_{\bbS}-\frac{\theta_1}{n}\bbU_1\bbX\bbB^{-1}(\theta_1)\bbX^{\intercal}\bbU^{\intercal}_1+\frac{\delta_1\theta_1^2}{n}\bbU_1\bbX\bbB^{-1}(\lambda_1)\bbB^{-1}(\theta_1)\bbX^{\intercal}\bbU^{\intercal}_1)=0.
\end{equation}
  To illustrate the main idea of our proof, we give a simple example. Suppose $K=2$ and we have shown that
$$\theta_1\Lambda^{-1}_{\bbS}-\frac{\theta_1}{n}\bbU_1\bbX\bbB^{-1}(\theta_1)\bbX^{\intercal}\bbU^{\intercal}_1=\left(
 \begin{matrix}
   \hat S_n & O_p(\frac{1}{\sqrt n})  \\
    O_p(\frac{1}{\sqrt n}) & 1+o_p(1)
  \end{matrix}
  \right)$$
and
$$\frac{\theta_1^2}{n}\bbU_1\bbX\bbB^{-1}(\lambda_1)\bbB^{-1}(\theta_1)\bbX^{\intercal}\bbU^{\intercal}_1=-\left(
 \begin{matrix}
   1+o_p(1) & o_p(1)  \\
    o_p(1)  & 1+o_p(1)
  \end{matrix}
  \right),$$
  where $\sqrt n\hat S_n\stackrel{D}{\rightarrow} N\left(0, 1\right)$. Then (\ref{1101.9}) becomes
  $$\det\left(
 \begin{matrix}
  \hat S_n+\delta_1(1+o_p(1)) & O_p(\frac{1}{\sqrt n})+o_p(\delta_1)  \\
    O_p(\frac{1}{\sqrt n})+o_p(\delta_1)  & 1+o_p(1)+\delta_1(1+o_p(1))
  \end{matrix}
  \right)=0.$$
  By Leibniz's formula for the determinant of a matrix, we have
  $$\delta_1(1+o_p(1))+\hat S_n(1+o_p(1))+o_p(\frac{1}{\sqrt n})=0,$$
  which implies that $\sqrt n\delta_1=\sqrt n\hat S_n+o_p(1)\stackrel{D}{\rightarrow} N\left(0, 1\right)$.

By the example above, similar to the proof of Theorem 3.1 in \cite{BY08}, the key steps are to establish the central limit theorem for the entries of  $\frac{1}{\sqrt n}\theta_1\bbU_1\bbX\bbB^{-1}(\theta_1)\bbX^{\intercal}\bbU^{\intercal}_1$  and the entry wise limit of $\frac{\theta_1^2}{ n}\bbU_1\bbX\bbB^{-2}(\theta_1)\bbX^{\intercal}\bbU^{\intercal}_1$ by Leibniz's formula for the determinant of a matrix.

Let $\bbu_i^{\intercal}$ be the i-th row of $\bbU_1$. By Theorem \ref{1101-1}, we have
\begin{eqnarray}\label{1101.10}
\sqrt n(\frac{\theta_1}{n}\bbu_i^{\intercal}\bbX\bbB^{-1}(\theta_1)\bbX^{\intercal}\bbu_i+\tilde m_{\theta_1}(1))\stackrel{D}{\rightarrow} N\left(0, \sigma^2_{i}\right)
\end{eqnarray}
 and
 $$\frac{1}{\sqrt n}\theta_1\bbu_i^{\intercal}\bbX\bbB^{-1}(\theta_1)\bbX^{\intercal}\bbu_j\stackrel{D}{\rightarrow} N\left(0, \sigma_{ij}+1\right), \ i\neq j,$$
 where $\sigma_i$ and $\sigma_{ij}$ are defined above (\ref{h0408.4}). By Chebyshev's inequality and the proof of Theorem \ref{1101-1} we have
 \begin{eqnarray}\label{1128.4}
 &&\mathbb{P}(\max_{1\le i,j\le k}|\frac{\theta_1}{n}\bbu_i^{\intercal}\bbX\bbB^{-1}(\theta_1)\bbX^{\intercal}\bbu_j+\delta_{ij}\tilde m_{\theta_1}(1)|\ge \frac{\ep}{\sqrt n})\non
 &&\le \sum_{1\le i,j\le k} \mathbb{P}(|\frac{\theta_1}{n}\bbu_i^{\intercal}\bbX\bbB^{-1}(\theta_1)\bbX^{\intercal}\bbu_j+\delta_{ij}\tilde m_{\theta_1}(1)|\ge \frac{\ep}{\sqrt n})\non
 &&\le \sum_{1\le i,j\le k} \frac{N\mathbb{E}|\frac{\theta_1}{n}\bbu_i^{\intercal}\bbX\bbB^{-1}(\theta_1)\bbX^{\intercal}\bbu_j+\delta_{ij}\tilde m_{\theta_1}(1)|^2}{t^2}=O(\frac{K^2}{\ep^2}),
 \end{eqnarray}
 which implies that $\max_{1\le i,j\le k}|\frac{\theta_1}{n}\bbu_i^{\intercal}\bbX\bbB^{-1}(\theta_1)\bbX^{\intercal}\bbu_j+\delta_{ij}\tilde m_{\theta_1}(1)|=O_p(\frac{K}{\sqrt n})$.
 It follows that
\begin{eqnarray}\label{1128.8}
&&\theta_1\Lambda^{-1}_{\bbS}-\frac{\theta_1}{n}\bbU_1\bbX\bbB^{-1}(\theta_1)\bbX^{\intercal}\bbU^{\intercal}_1\\
&&=\left[
 \begin{matrix}
  \hat S_n & O_p(\frac{K}{\sqrt n}) & ..& ..& O_p(\frac{K}{\sqrt n}) \\
    O_p(\frac{K}{\sqrt n}) & O_p(1) & .. & ..&  O_p(\frac{K}{\sqrt n})\\
   . & ... & .. & & .\\
   . & ... & O_p(1)   & &O_p(\frac{K}{\sqrt n})\\
   O_p(\frac{K}{\sqrt n}) & ... & O_p(\frac{K}{\sqrt n})    & &O_p(1)
  \end{matrix}
  \right],\nonumber
  \end{eqnarray}
  where $\hat S_n=\frac{\theta_1}{n}\bbu_1^{\intercal}\bbX\bbB^{-1}(\theta_1)\bbX^{\intercal}\bbu_1+\tilde m_{\theta_1}(1)$.
Moreover, we claim that there exists $\delta_n\rightarrow 0$ such that
  \begin{eqnarray}\label{0916.1}
&&\|\frac{\theta_1^2}{n}\bbU_1\bbX\bbB^{-2}(\theta_1)\bbX^{\intercal}\bbU^{\intercal}_1+(1+\delta_n)\tilde m_{\theta_1}(1)\bbI\|_{\infty}=O_p(\frac{K}{\sqrt n}),
\end{eqnarray}
whose proof is given in section \ref{2.10}.
By Theorem \ref{0318-1} and (\ref{1101.6}) we have
\begin{eqnarray}\label{1128.5}
&&\|\frac{\theta_1^2}{n}\bbU_1\bbX\bbB^{-1}(\lambda_1)\bbB^{-1}(\theta_1)\bbX^{\intercal}\bbU^{\intercal}_1-\frac{\theta_1^2}{n}\bbU_1\bbX\bbB^{-2}(\theta_1)\bbX^{\intercal}\bbU^{\intercal}_1\|_{\infty}\non
&&=\delta_1\|\frac{\theta_1^3}{n^2}\bbU_1\bbX\bbB^{-1}(\lambda_1)\bbX^{\intercal}\bold\Sigma_1\bbX\bbB^{-2}(\theta_1)\bbX^{\intercal}\bbU^{\intercal}_1\|_{\infty}=O_p(\frac{K^4}{n}+d_1),
\end{eqnarray}
which, together with (\ref{0916.1}), implies that
$$\|\frac{\theta_1^2}{n}\bbU_1\bbX\bbB^{-1}(\lambda_1)\bbB^{-1}(\theta_1)\bbX^{\intercal}\bbU^{\intercal}_1+(1+\delta_n)\tilde m_{\theta_1}(1)\bbI\|_{\infty}=O_p(\frac{K^2}{\sqrt n}+\frac{K^4}{n}+d_1).$$
By Leibniz formula for determinant and a tedious calculation, one can show that
 $$\delta_1(1+O_p(K^2d_1+\frac{K^6}{n}))+\hat S_n(1+o_p(1))+o_p(\frac{1}{\sqrt n})=0.$$
By (\ref{1101.10}) we have shown that
$$\sqrt n\delta_1\stackrel{D}{\rightarrow} N\left(0, \sigma^2_{1}\right),$$
and the proof of this section is complete.
\subsubsection{Proof of (\ref{0916.1})}\label{2.10}
The proof of (\ref{0916.1}) is similar to Section \ref{calculatemean} and we merely give a sketch of the proof. We consider a special entry $\mathbb{E}(\frac{\theta_1^2}{n}\bbu_1^{\intercal}\bbX\bbB^{-2}(\theta_1)\bbX^{\intercal}\bbu_1+(1+\delta_n)\tilde m_{\theta_1}(1))^2$ of (\ref{0916.1}) as an example.
First of all, as in (\ref{0314.1})- (\ref{0411.1}), one can show that
$\mathbb{E}|\frac{\theta_1^2}{n}\bbu_1^{\intercal}\bbX\bbB^{-2}(\theta_1)\bbX^{\intercal}\bbu_1-\mathbb{E}\frac{\theta_1^2}{n}\bbu_1^{\intercal}\bbX\bbB^{-2}(\theta_1)\bbX^{\intercal}\bbu_1|^2=O(\frac{1}{n})$. Therefore by Chebyshev's inequality, we have
{\small \begin{equation*}
\frac{1}{n}\|\theta_1^2\bbU_1\bbX\bbB^{-2}(\theta_1)\bbX^{\intercal}\bbU^{\intercal}_1-\mathbb{E}\theta_1^2\bbU_1\bbX\bbB^{-2}(\theta_1)\bbX^{\intercal}\bbU^{\intercal}_1\|_{\infty}=O_p(\frac{K}{\sqrt n}).
\end{equation*}}
Next, by the interpolation method introduced in Section \ref{calculatemean} we can show that
 \begin{eqnarray}\label{0916.2}
 &&\frac{\theta_1^2}{n}\mathbb{E}\bbu_1^{\intercal}\bbX\bbB^{-2}(\theta_1)\bbX^{\intercal}\bbu_1+(1+\delta_n)\tilde m_{\theta_1}(1)\\
 &&=\frac{\theta_1^2}{n}\mathbb{E}\bbu_1^{\intercal}\bbX^0\bbB_0^{-2}(\theta_1)(\bbX^0)^{\intercal}\bbu_1+(1+\delta_n)\tilde m_{\theta_1}(1)+o(\frac{1}{\sqrt n}),\nonumber
 \end{eqnarray}
where $\bbB_0(\theta_1)=\theta\bbI-\frac{1}{n}(\bbX^0)^{\intercal}\bbU^{\intercal}_2\Lambda_P\bbU_2\bbX^0$ and the above equation  implies that
 \begin{eqnarray*}
 &&\|\frac{\theta_1^2}{n}\bbU_1\bbX^1\bbB^{-2}(\theta_1)\bbX^{\intercal}\bbU^{\intercal}_1+(1+\delta_n)\tilde m_{\theta_1}(1)\bbI\|_{\infty}\non
 &&=\|\frac{\theta_1^2}{n}\bbU_1\bbX^0\bbB_0^{-2}(\theta_1)(\bbX^0)^{\intercal}\bbU^{\intercal}_1+(1+\delta_n)\tilde m_{\theta_1}(1)\bbI\|_{\infty}+O_p(\frac{K}{\sqrt n}).
 \end{eqnarray*}
 Moreover, note that
$$\mathbb{E}\frac{\theta_1^2}{n}\bbu_1^{\intercal}\bbX^0\bbB_0^{-2}(\theta_1)(\bbX^0)^{\intercal}\bbu_1+(1+\delta_n)\tilde m_{\theta_1}(1)\bbI=
\frac{\theta_1^2}{n}\mathbb{E}tr[\bbB_0^{-2}(\theta_1)]+(1+\delta_n)\tilde m_{\theta_1}(1)\bbI.$$
 Let $\tilde \nu_i$ be the i-th largest eigenvalue of $\theta_1\bbB_0^{-1}(\theta_1)$. Then we have
 $$\frac{\theta_1^2}{n}\mathbb{E}tr\bbB_0^{-2}(\theta_1)=\frac{1}{n}\mathbb{E}\sum_{i=1}^n\tilde \nu_i^2.$$
 By (\ref{0916.3}) we have $\frac{1}{n}\mathbb{E}\sum_{i=1}^n\tilde \nu_i=-\tilde m_{\theta_1}(1)+o(\frac{1}{\sqrt n})$. Together with the simple fact that $\tilde \nu_i=1+O(\sqrt {d_K})$ with high probability, we conclude that there exists such $\delta_n\rightarrow 0$ such that
 $$\frac{1}{n}\mathbb{E}\sum_{i=1}^n\tilde \nu^2_i=-(1+\delta_n)\tilde m_{\theta_1}(1)+o(\frac{1}{\sqrt n}).$$
Up to now, we have shown that
 $$\mathbb{E}\|\frac{\theta_1^2}{n}\bbU_1\bbX^0\bbB_0^{-2}(\theta_1)(\bbX^0)^{\intercal}\bbU_1^{\intercal}+(1+\delta_n)\tilde m_{\theta_1}(1)\bbI\|_{\infty}=O(\frac{K}{\sqrt n})$$
 and hence
  $$\mathbb{E}\|\frac{\theta_1^2}{n}\bbU_1\bbX\bbB^{-2}(\theta_1)\bbX^{\intercal}\bbU_1^{\intercal}+(1+\delta_n)\tilde m_{\theta_1}(1)\bbI\|_{\infty}=O(\frac{K}{\sqrt n}).$$
\subsubsection{Joint Distribution (\ref{0831.3})}
This section aims at proving the asymptotic joint distribution of the spiked eigenvalues. i.e. (\ref{0831.3}). By the argument leading to (\ref{1127.3}), we conclude that it suffices to consider the asymptotic joint distribution of
\begin{eqnarray}\label{0411.2}
\left(\frac{1}{\sqrt{n}}\sum_{k=1}^n((\bbu_1\bbx_k)^2-1),..., \frac{1}{\sqrt{n}}\sum_{k=1}^n((\bbu_r\bbx_k)^2-1)\right), \ r\ge 2.
\end{eqnarray}
The covariance of the cross term is
\begin{eqnarray}\label{0411.3}
&&\frac{1}{n}\sum_{k=1}^n\mathbb{E}\left[\big((\bbu_i\bbx_k)^2-1\big)\big((\bbu_j\bbx_k)^2-1\big)\right]=\frac{\left(\sum_{s=1}^{p+l}(\gamma_{4s}-3)u_{is}^2u_{js}^2\right)}{n}\non
&&\longrightarrow\lim_{n\rightarrow \infty}\sum_{s=1}^{p+l}(\gamma_{4s}-3)u_{is}^2u_{js}^2=\sigma_{ij}.
\end{eqnarray}

 \section{Proof of Lemma \ref{1101-1x}}
\begin{proof}
Recalling $\bbA=\bbI-\frac{1}{n}\bbX^T\tilde \Sigma_1\bbX$, let $\bbA_{\Upsilon}=\bbI-\frac{1}{n}\Upsilon\bbX^T\tilde \Sigma_1\bbX\Upsilon$ and $\bbA_{(\Upsilon)}=\bbI-\frac{1}{n}\bbX^T\tilde \Sigma_1\bbX\Upsilon$. By Theorem \ref{1101-1}, it suffices to show that
\begin{eqnarray}\label{0223.2}
&&\\
\frac{1}{n}(\bbw_1^T\bbX\bbA^{-1}\bbX^T\bbw_1-\bbw_1^T\bbX\Upsilon\bbA_{\Upsilon}^{-1}\Upsilon\bbX^T\bbw_1)I(F_d)=o_{L_1}(1/\sqrt n)\nonumber,
\end{eqnarray}
\begin{eqnarray}\label{0223.3}
&&\\
\frac{1}{n}(\bbw_1^T\bbX\bbA^{-1}\bbX^T\bbw_2-\bbw_1^T\bbX\Upsilon\bbA_{\Upsilon}^{-1}\Upsilon\bbX^T\bbw_2)I(F_d)=o_{L_1}(1/\sqrt n),\nonumber
\end{eqnarray}
where $\tilde \Sigma_1=\frac{\bold\Sigma_1}{\theta}$.
We prove (\ref{0223.2}) and (\ref{0223.3}) can be shown similarly. In the following proof we  also omit $I(F_d)$ to simplify notation. First of all, we have
\begin{eqnarray}\label{0223.4}
&&\frac{1}{n}\bbw_1^T\bbX\Upsilon\bbA_{\Upsilon}^{-1}\Upsilon\bbX^T\bbw_1=\frac{1}{n}\bbw_1^T\bbX\bbA_{\Upsilon}^{-1}\bbX^T\bbw_1\\
&&-\frac{2}{n^2}\bbw_1^T\bbX\mathbf{1}\mathbf{1}^T\bbA_{\Upsilon}^{-1}\bbX^T\bbw_1+\frac{2}{n^3}\bbw_1^T\bbX\mathbf{1}\mathbf{1}^T\bbA_{\Upsilon}^{-1}\mathbf{1}\mathbf{1}^T\bbX^T\bbw_1.\nonumber
\end{eqnarray}
Let  $\Delta=\frac{1}{n^2}\Upsilon\bbX^T\tilde \Sigma_1\bbX\Upsilon\bbA_{\Upsilon}^{-1}$. It is easy to see that
$$\frac{1}{n}\bbA_{\Upsilon}^{-1}=\frac{1}{n}\bbI+\Delta,$$
and $\|\Delta\|=o(\frac{1}{n})$.
 It follows that
$$\frac{2}{n^2}\bbw_1^T\bbX\mathbf{1}\mathbf{1}^T\bbA_{\Upsilon}^{-1}\bbX^T\bbw_1=\frac{2}{n^2}\bbw_1^T\bbX\mathbf{1}\mathbf{1}^T\bbX^T\bbw_1+\frac{2}{n}\bbw_1^T\bbX\mathbf{1}\mathbf{1}^T\Delta\bbX^T\bbw_1.$$
A direct calculation indicates that
\begin{eqnarray}\label{0123.5}
\frac{2}{n^2}\mathbb{E}|\bbw_1^T\bbX\mathbf{1}\mathbf{1}^T\bbX^T\bbw_1|=\frac{2}{n}.
\end{eqnarray}
Holder's inequality ensures that
\begin{eqnarray}\label{0123.6}
&&\frac{2}{n}\mathbb{E}|\bbw_1^T\bbX\mathbf{1}\mathbf{1}^T\Delta\bbX^T\bbw_1|\le\\ &&\frac{2}{n}\sqrt{\mathbb{E}|\bbw_1^T\bbX\mathbf{1}\mathbf{1}^T\bbX^T\bbw_1|}\sqrt{\mathbb{E}|\bbw_1^T\bbX\Delta^T\mathbf{1}\mathbf{1}^T\Delta\bbX^T\bbw_1|}=o(1/\sqrt n).\nonumber
\end{eqnarray}
Therefore,
$$\frac{2}{n^2}\bbw_1^T\bbX\mathbf{1}\mathbf{1}^T\bbA_{\Upsilon}^{-1}\bbX^T\bbw_1=o_{L_1}(1/\sqrt n).$$ Similarly, we have
$$\frac{2}{n^3}\bbw_1^T\bbX\mathbf{1}\mathbf{1}^T\bbA_{\Upsilon}^{-1}\mathbf{1}\mathbf{1}^T\bbX^T\bbw_1=o_{L_1}(1/\sqrt n).$$
In view of (\ref{0223.4}), it remains to show that
\begin{eqnarray}\label{0123.7}
\frac{1}{n}\bbw_1^T\bbX\bbA_{\Upsilon}^{-1}\bbX^T\bbw_1-\frac{1}{n}\bbw_1^T\bbX\bbA^{-1}\bbX^T\bbw_1=o_{L_1}(1/\sqrt n).
\end{eqnarray}
It is not hard to see that
\begin{equation}\label{0123.8}
\frac{1}{n}\bbA_{\Upsilon}^{-1}-\frac{1}{n}\bbA^{-1}=\frac{1}{n^2}\bbA_{\Upsilon}^{-1}(\frac{1}{n^2}\mathbf{1}\mathbf{1}^T\bbX^T\tilde {\bold\Sigma}_1\bbX\mathbf{1}\mathbf{1}^T-\frac{1}{n}\mathbf{1}\mathbf{1}^T\bbX^T\tilde {\bold\Sigma}_1\bbX-\bbX^T\tilde {\bold\Sigma}_1\bbX\frac{1}{n}\mathbf{1}\mathbf{1}^T)\bbA^{-1}.
\end{equation}
By (\ref{0123.8}), consider one term in the left hand side of (\ref{0123.7}) first, i.e.
\begin{eqnarray}\label{0123.9}
\frac{1}{n^3}\bbw_1^T\bbX\bbA_{\Upsilon}^{-1}\mathbf{1}\mathbf{1}^T\bbX^T\tilde {\bold\Sigma}_1\bbX\bbA^{-1}\bbX^T\bbw_1.
\end{eqnarray}
By the property that $\Upsilon \mathbf{1}=0$ and $\Upsilon^2=\Upsilon$, we have
$$\frac{1}{n}\bbA_{\Upsilon}^{-1}\mathbf{1}=\frac{1}{n}\sum_{k=0}^{\infty}(\frac{1}{n}\Upsilon\bbX^T\tilde {\bold\Sigma}_1\bbX\Upsilon)^k\mathbf{1}=\frac{1}{n}\mathbf{1}.$$
It follows from (\ref{0123.5}) that
\begin{eqnarray}\label{0123.10}
\mathbb{E}|(\ref{0123.9})|&=&\frac{1}{n^3}\mathbb{E}|\bbw_1^T\bbX\mathbf{1}\mathbf{1}^T\bbX^T\tilde {\bold\Sigma}_1\bbX\bbA^{-1}\bbX^T\bbw_1|\non
&\le& \frac{1}{n^3}\sqrt{\mathbb{E}(\bbw_1^T\bbX\mathbf{1})^2}\sqrt{\mathbb{E}(\mathbf{1}^T\bbX^T\tilde {\bold\Sigma}_1\bbX\bbA^{-1}\bbX^T\bbw_1)^2}=o(1/\sqrt n).
\end{eqnarray}
Similar to (\ref{0123.10}), one can prove
$$\frac{1}{n^4}\bbw_1^T\bbX\bbA_{\Upsilon}^{-1}\mathbf{1}\mathbf{1}^T\bbX^T\tilde {\bold\Sigma}_1\bbX\mathbf{1}\mathbf{1}^T\bbA^{-1}\bbX^T\bbw_1=o_{L_1}(1/\sqrt n).$$
For the remaining term of (\ref{0123.8})
$$
\frac{1}{n}\bbw_1^T\bbX(n\bbI-\Upsilon\bbX^T\tilde {\bold\Sigma}_1\bbX\Upsilon)^{-1}\bbX^T\tilde {\bold\Sigma}_1\bbX\mathbf{1}\mathbf{1}^T(n\bbI-\bbX^T\tilde {\bold\Sigma}_1\bbX)^{-1}\bbX^T\bbw_1,$$
Similar to (\ref{0123.10}), it suffices to show
\begin{eqnarray}\label{0123.11}
\frac{1}{n}\bbw_1^T\bbX\bbA^{-1}\mathbf{1}=O_{L_1}(1/\sqrt n).
\end{eqnarray}
Actually, applying the  same strategy as in (\ref{0314.1})-(\ref{0411.1}), we can prove that
\begin{eqnarray}\label{0123.12}
\frac{1}{n}\bbw_1^T\bbX\bbA^{-1}\mathbf{1}I(F_d)-\frac{1}{n}\mathbb{E}\bbw_1^T\bbX\bbA^{-1}\mathbf{1}I(F_d)=O_{L_1}(1/\sqrt n).
\end{eqnarray}
Moreover, applying the strategy of Section \ref{calculatemean}, one can show that
\begin{eqnarray}\label{0123.12h}
\frac{1}{n}\mathbb{E}\bbw_1^T\bbX\bbA^{-1}\mathbf{1}I(F_d)=O(1/\sqrt n).
\end{eqnarray}
The detailed proof of (\ref{0123.12}) and (\ref{0123.12h}) is omitted since it is even simpler than that of Theorem \ref{1101-1}.
\end{proof}
\section{Relax Assumption \ref{0829-1b}: Truncation and Centralization}\label{0914-1}
This section is to truncate and centralize $\bbx_{ij}$. By assumption \ref{0829-1a},  there exists a positive sequence $\delta_n$ satisfying
\begin{equation}\label{0321.2}
\lim_{n\rightarrow \infty}\frac{1}{np\delta_n^4}\sum_{i=1}^{p+l}\sum_{j=1}^n\mathbb{E}|\bbx_{ij}|^4I(|\bbx_{ij}|>\delta_n\sqrt[4]{np})=0, \ \ \delta_n\downarrow 0,\ \  \delta_n\sqrt[4]{np}\uparrow \infty.
\end{equation}
We first truncate $\bbx_{ij}$ to $\hat \bbx_{ij}=\bbx_{ij}I(|\bbx_{ij}|<\delta_n\sqrt[4]{np})$ and then get the centralized version $\tilde \bbx_{ij}=\frac{\hat \bbx_{ij}-\mathbb{E}\hat \bbx_{ij}}{\sigma_i}$, where $\sigma_i$ is the standard deviation of $\hat \bbx_{ij}$.  It is easy to see that
\begin{eqnarray}\label{0321.3}
\mathbb{P}(\bbX\neq \hat\bbX)&\le& \sum_{i=1}^{p+l}\sum_{j=1}^n\mathbb{P}(|\bbx_{ij}|\ge \delta_n\sqrt[4]{np})\non
&\le& \frac{C}{np\delta_n^4}\sum_{i=1}^{p+l}\sum_{j=1}^n\mathbb{E}|\bbx_{ij}|^4I(|\bbx_{ij}|>\delta_n\sqrt[4]{np})\rightarrow 0.
\end{eqnarray}
It follows that
$$\mathbb{P}(\bbU_1 \bbX(\lambda_i\bbI- \frac{1}{n}\bbX^T\bold\Sigma_1\bbX)^{-1}\bbX^T\bbU^T_1\neq\bbU_1\hat\bbX(\lambda_i\bbI-\frac{1}{n}\hat\bbX^T\Sigma_1\hat\bbX)^{-1}\hat\bbX^T\bbU^T_1)\rightarrow 0.$$
For convenience, define $\bbB_{\bbX}(x)=x\bbI-\frac{1}{n}\bbX^T\bold\Sigma_1\bbX$. Hence with probability tending to 1, the proofs of the above theorems based on (\ref{0912.2h}) are equivalent to
\begin{equation}\label{0912.2}
\det(\Lambda^{-1}_{\bbS}-\frac{1}{n}\bbU_1\hat\bbX\bbB_{\hat\bbX}(\lambda_i)^{-1}\hat \bbX^T\bbU^T_1)=0.
\end{equation}

Note that
\begin{eqnarray}\label{0321.4}
|1-\sigma_i^2|&\le& 2|\mathbb{E}(\bbx_{ij}^2)I(|\bbx_{ij}|>\delta_n\sqrt[4]{np})|\non
&\le& 2 (np)^{-1/2}\delta_n^{-2}\mathbb{E}|\bbx_{ij}|^4I(|\bbx_{ij}|>\delta_n\sqrt[4]{np}),
\end{eqnarray}
\begin{eqnarray}\label{0321.5}
&&|\mathbb{E}\hat \bbx_{ij}|\le \delta_n^{-3}(np)^{-3/4}\mathbb{E}|\bbx_{ij}|^4I(|\bbx_{ij}|>\delta_n\sqrt[4]{np}),
\end{eqnarray}
and
\begin{eqnarray}\label{0912.1}
&&\frac{1}{n}\mathbb{E}tr(\hat\bbX-\tilde\bbX)(\hat\bbX-\tilde\bbX)^T\le \sum_{i=1}^{p+l}\sum_{j=1}^n\mathbb{E}|\hat\bbx_{ij}-\tilde\bbx_{ij}|^2\non
&&\le \frac{C}{n}\sum_{i=1}^{p+l}\sum_{j=1}^n(\frac{(1-\sigma_i)^2}{\sigma_i^2}\mathbb{E}|\hat\bbx_{ij}|^2+\frac{1}{\sigma_i^2}|\mathbb{E}\hat\bbx_{ij}|^2) =o(\frac{1}{n}).
\end{eqnarray}

By (\ref{0321.4}), (\ref{0321.5}) and (\ref{0912.1}), replacing
$\hat \bbX$ by $\tilde \bbX$, it is easy to show the perturbation is $o_p(Kn^{-1/2})$, which means that
\begin{equation}\label{1014.1h}
\frac{1}{n}\|\bbU_1\hat\bbX\bbB_{\hat\bbX}(\lambda_i)^{-1}\hat\bbX^T\bbU^T_1-\bbU_1\tilde \bbX\bbB_{\tilde\bbX}(\lambda_i)^{-1}\tilde \bbX^T\bbU^T_1\|_{\infty}=o_p(Kn^{-1/2}\mu_i^{-1}),
\end{equation}
and
$$\frac{1}{n}\|\bbu_i^T\hat\bbX\bbB_{\hat\bbX}(\lambda_i)^{-1}\hat\bbX^T\bbu_i-\bbu_i^T\tilde \bbX\bbB_{\tilde\bbX}(\lambda_i)^{-1}\tilde \bbX^T\bbu_i\|_{\infty}=o_p(n^{-1/2}\mu_i^{-1}).$$
Therefore, (\ref{0912.2h}) can be rewritten as
\begin{equation}\label{0912.2}
\det(\Lambda^{-1}_1-\bbU_1\tilde \bbX\bbB_{\tilde\bbX}(\lambda_i)^{-1}\tilde \bbX^T\bbU^T_1+o_p(Kn^{-1/2}\mu_i^{-1})(\mathbf{1}\mathbf{1}^T-\bbe_i\bbe_i^T)+o_p(n^{-1/2}\mu_i^{-1})\bbe_i\bbe_i^T)=0,
\end{equation}
where $o_p(.)$ is the entry wise order. One should notice that we deal with $(\frac{1}{n}\bbU_1\hat\bbX\bbB_{\hat\bbX}(\lambda_i)^{-1}\hat \bbX^T\bbU^T_1)_{ii}$ independently with the other entries and hence we have the order $o_p(n^{-1/2}\mu_i^{-1})\bbe_i\bbe_i^T$. From the proof of Theorem \ref{0408-2}, it is not hard to find out that the terms involving $o_p(Kn^{-1/2})$ are negligible and does not affect CLT(see (\ref{1128.8})), which means that we can prove Theorem \ref{0408-2} from the following equality
\begin{equation}\label{0912.3}
\det(\mu_i\Lambda^{-1}_1-\mu_i\bbU_1\tilde \bbX(\lambda_i\bbI-\tilde \bbX^T\Sigma_1\tilde \bbX)^{-1}\tilde \bbX^T\bbU^T_1)=0.
\end{equation}
Checking on the proof  of Theorem \ref{0408-2}, all arguments hold for $\tilde \bbX$ as well. Up to now,  we have relaxed Assumption \ref{0829-1b} and finish this section.
\section{Proof of Theorem \ref{0831-1}.}

The proof of Theorem \ref{0831-1} is almost the same as that of  Theorem \ref{0408-2}. We illustrate the joint distribution of the first $n_1$ eigenvalues as an example. Checking on the proof of Theorems \ref{0408-2} and \ref{1101-1}  carefully, we can get the following equality similar to (\ref{1128.8})
\begin{eqnarray}\label{1128.9}
&&\theta_1\Lambda^{-1}_{\bbS}-\frac{\theta_1}{n}\bbU_1\bbX\bbB^{-1}(\theta_1)\bbX^T\bbU^T_1\\
&&=\left[
 \begin{matrix}
   \tilde S_n & O_p(\frac{K}{\sqrt n}) & ..& ..& O_p(\frac{K}{\sqrt n}) \\
    O_p(\frac{K}{\sqrt n}) & O_p(1) & .. & ..&  O_p(\frac{K}{\sqrt n})\\
   . & ... & .. & & .\\
   . & ... & O_p(1)   & &O_p(\frac{K}{\sqrt n})\\
   O_p(\frac{K}{\sqrt n}) & ... & O_p(\frac{K}{\sqrt n})    & &O_p(1)
  \end{matrix}
  \right],\nonumber
     \end{eqnarray}
 where $\tilde S_n$ is a $n_1\times n_1$ matrix such that $\sqrt n \tilde S_n\stackrel{D}{\rightarrow} \mathfrak{R}_1$. Here $\mathfrak{R}_1$ follows normal distribution with $\mathbb{E}\mathfrak{R}_1=0$ and the covariance of the $(\mathfrak{R}_1)_{k_1, l_1}$ and $(\mathfrak{R}_1)_{k_2, l_2}$ is $\lim_{n\rightarrow \infty}N^2\times Cov(\bbu_{k_1}^T\bbx\bbu_{l_1}^T\bbx,\bbu_{k_2}^T\bbx\bbu_{l_2}^T\bbx)$. The asymptotic distribution of $\mathfrak{R}_1$ is ensured by  the fact that the upper left $n_1\times n_1$ block of $\theta_1\Lambda^{-1}_{\bbS}-\frac{\theta_1}{n}\bbU_1\bbX\bbB^{-1}(\theta_1)\bbX^T\bbU^T_1$ is constructed by the entries with the expressions similar to (\ref{1101.1}) or (\ref{1101.2}). Therefore, by the Skorokhod strong representation and the corresponding arguments similar to page 464-465 of \cite{BY08} we conclude Theorem \ref{0831-1}.

\section{Proof of Theorem \ref{0503-1}}

\begin{proof}
Without loss of generality, we only consider the first spiked eigenvalue $\lambda_1$. The other spiked eigenvalues $\lambda_2,...,\lambda_K$ can be handled similarly.  By Cauchy's integral theorem and the residue theorem,  with high probability, we have
$$\bbv_1^T\xi_1\xi_1^T\bbv_1=-\frac{1}{2\pi i}\oint_{\Pi}\bbv_1^T\tilde \bbG(z)\bbv_1dz,$$
where $\tilde \bbG(z)=(\bbS_n-z\bbI)^{-1}$ and $\Pi$ is a contour enclosing $\lambda_1$ but the other eigenvalues $\lambda_i$. The existence of the contour $\Pi$ is ensured by Theorem \ref{0318-1} and Assumption \ref{0829-1c}. In the sequel, we directly work on the integral $-\frac{1}{2\pi i}\oint_{\Pi}\tilde \bbG(z)dz$. Write
\begin{eqnarray}\label{0502.1}
&&\\
&&\bbv_1^T\tilde \bbG(z)\bbv_1=\bbv_1^T(\bbS_n-z\bbI)^{-1}\bbv_1=\bbe_1^T(\frac{1}{n}\Lambda^{1/2}\bbU\bbX\bbX^T\bbU^T\Lambda^{1/2}-z\bbI)^{-1}\bbe_1\non
&&=\left(
  \begin{array}{ccc}
    \frac{1}{n}\lambda_1\bbu_1^T\bbX\bbX^T\bbu_1-z & \frac{1}{n}\lambda_1^{1/2}\bbu_1^T\bbX\bbX^T\tilde \bbU_2^T\Lambda_2^{1/2} \\
    \frac{1}{n}\lambda_1^{1/2}\Lambda_2^{1/2}\tilde \bbU_2\bbX\bbX^T\bbu_1 & \frac{1}{n}\Lambda_2^{1/2}\tilde \bbU_2\bbX\bbX^T\tilde \bbU_2^T\Lambda_2^{1/2}-z\bbI \\
  \end{array}
\right)^{-1}_{11}\non
&&=\left(\frac{1}{n}\lambda_1\bbu_1^T\bbX\bbX^T\bbu_1-z-\frac{1}{n^2}\lambda_1\bbu_1^T\bbX\bbX^T\tilde \bbU_2^T\Lambda_2^{1/2}(\frac{1}{n}\Lambda_2^{1/2}\tilde \bbU_2\bbX\bbX^T\tilde \bbU_2^T\Lambda_2^{1/2}-z\bbI)^{-1}\Lambda_2^{1/2}\tilde \bbU_2\bbX\bbX^T\bbu_1\right)^{-1}.\nonumber
\end{eqnarray}
The aim is to prove $\frac{1}{n^2}\lambda_1\bbu_1^T\bbX\bbX^T\tilde \bbU_2^T\Lambda_2^{1/2}(\frac{1}{n}\Lambda_2^{1/2}\tilde \bbU_2\bbX\bbX^T\tilde \bbU_2^T\Lambda_2^{1/2}-z\bbI)^{-1}\Lambda_2^{1/2}\tilde \bbU_2\bbX\bbX^T\bbu_1$ converges to $0$ in probability. Not that $\Lambda_2$ contains the remaining $K-1$ spiked eigenvalues and the other non-spiked eigenvalues. Moreover, the non-spiked eigenvalues are all dominated by $z$. Hinted by this observation, we write $\Lambda_2=\left(
  \begin{array}{ccc}
    \Lambda_{21} & 0 \\
    0 & \Lambda_{P} \\
  \end{array}
\right)$ and $\tilde\bbU_2=\left(
  \begin{array}{c}
    \bbU_{21} \\
   \bbU_{2} \\
  \end{array}
\right)$, where $\Lambda_{21}$ is $(K-1)\times(K-1)$ diagonal matrix and $\bbU_{21}$ is the corresponding $(K-1)\times (p+l) $ eigenvector matrix. It follows that
\begin{eqnarray}\label{0502.2}
&&(\frac{1}{n}\Lambda_2^{1/2}\tilde \bbU_2\bbX\bbX^T\tilde \bbU_2^T\Lambda_2^{1/2}-z\bbI)^{-1}=\left(
  \begin{array}{ccc}
    \frac{1}{n}\Lambda_{21}^{1/2}\bbU_{21}\bbX\bbX^T\bbU_{21}^T\Lambda_{21}^{1/2}-z\bbI & \frac{1}{n}\Lambda_{21}^{1/2}\bbU_{21}\bbX\bbX^T\bbU_{2}^T\Lambda_{P}^{1/2} \\
    \frac{1}{n}\Lambda_{P}^{1/2}\bbU_{2}\bbX\bbX^T\bbU_{21}^T\Lambda_{21}^{1/2} & \frac{1}{n}\Lambda_{P}^{1/2}\bbU_{2}\bbX\bbX^T\bbU_{2}^T\Lambda_{P}^{1/2}-z\bbI \\
  \end{array}
\right)^{-1}\non
&&=\left(
  \begin{array}{ccc}
   \bbA & \bbB \\
    \bbB^T & \bbC \\
  \end{array}
\right)^{-1}=\left(
  \begin{array}{ccc}
    (\bbA-\bbB\bbC^{-1}\bbB^T)^{-1} & -(\bbA-\bbB\bbC^{-1}\bbB^T)^{-1}\bbB\bbC^{-1} \\
     -\bbC^{-1}\bbB^T(\bbA-\bbB\bbC^{-1}\bbB^T)^{-1} & \bbC^{-1}+\bbC^{-1}\bbB^T(\bbA-\bbB\bbC^{-1}\bbB^T)^{-1}\bbB\bbC^{-1} \\
  \end{array}
\right),\non
\end{eqnarray}
where $\bbA$, $\bbB$ and $\bbC$ are defined in an obvious way.
By the definitions of $\bbA$, $\bbB$ and $\bbC$ and the choice of $\Pi$, it is easy to see that
\begin{eqnarray}\label{0216.1}
\|\Lambda_{21}^{-1/2}\bbB\|=O_p(\sqrt{\frac{p}{n}}), \|\bbC\|=O_p( |z|)
 \end{eqnarray}
 and
\begin{eqnarray}\label{0216.2}
\|\bbC^{-1}\|=O_p(\frac{1}{|z|}).
  \end{eqnarray}Moreover, $\|\bbA-\Lambda_{21}+z\bbI\|=o_p(1)$ since the dimension of $\bbA$ is $(K-1)\times (K-1)$. By (\ref{0502.2}), a straight forward calculation for block matrices yields
 \begin{eqnarray}\label{0503.1}
&&\bbX^T\tilde \bbU_2^T\Lambda_2^{1/2}(\frac{1}{n}\Lambda_2^{1/2}\tilde \bbU_2\bbX\bbX^T\tilde \bbU_2^T\Lambda_2^{1/2}-z\bbI)^{-1}\Lambda_2^{1/2}\tilde \bbU_2\bbX\\
&&=\bbX^T\bbU_{21}^T\Lambda_{21}^{1/2}(\bbA-\bbB\bbC^{-1}\bbB^T)^{-1}\Lambda_{21}^{1/2}\bbU_{21}\bbX-\bbX^T\bbU_{21}^T\Lambda_{21}^{1/2}(\bbA-\bbB\bbC^{-1}\bbB^T)^{-1}\bbB\bbC^{-1}\Lambda_{P}^{1/2}\bbU_{2}\bbX\non
&&-\bbX^T\bbU_{2}^T\Lambda_{P}^{1/2}\bbC^{-1}\bbB^T(\bbA-\bbB\bbC^{-1}\bbB^T)^{-1}\Lambda_{21}^{1/2}\bbU_{21}\bbX\non
&&+\bbX^T\bbU_{2}^T\Lambda_{P}^{1/2}\left(\bbC^{-1}+\bbC^{-1}\bbB^T(\bbA-\bbB\bbC^{-1}\bbB^T)^{-1}\bbB\bbC^{-1} \right)\Lambda_{P}^{1/2}\bbU_{2}\bbX.\nonumber
\end{eqnarray}
Although the expression of (\ref{0503.1}) is complicated, it is not hard to conclude that all terms at the right hand side of (\ref{0503.1}) are $o_p(1)$ in terms of the spectral norm. For instance, we calculate one term $\bbX^T\bbU_{21}^T\Lambda_{21}^{1/2}(\bbA-\bbB\bbC^{-1}\bbB^T)^{-1}\Lambda_{21}^{1/2}\bbU_{21}\bbX$. Note that
\begin{eqnarray}\label{0503.6}
&&\|\Lambda_{21}^{1/2}(\bbA-\bbB\bbC^{-1}\bbB^T)^{-1}\Lambda_{21}^{1/2}\|\non
&&=\|(\frac{1}{n}\bbU_{21}\bbX\bbX^T\bbU^T_{21}-z\Lambda_{21}^{-1}-\Lambda_{21}^{-1/2}\bbB\bbC^{-1}\bbB^T\Lambda_{21}^{-1/2})^{-1}\|\le 2\end{eqnarray}
 with probability tending to 1, where we use the fact that $\Lambda_{21}^{-1/2}\bbB=\frac{1}{n}\bbU_{21}\bbX\bbX^T\bbU_{2}^T\Lambda_{P}^{1/2}$ and therefore $\|\Lambda_{21}^{-1/2}\bbB\bbC^{-1}\bbB^T\Lambda_{21}^{-1/2}\|=o_p(1)$ by (\ref{0216.1})-(\ref{0216.2}). Hence,
\begin{eqnarray}\label{0503.2}
&&\frac{1}{n^2}\bbu_1^T\bbX\bbX^T\tilde \bbU_2^T\Lambda_2^{1/2}(\frac{1}{n}\Lambda_2^{1/2}\tilde \bbU_2^T\bbX\bbX^T\tilde \bbU_2\Lambda_2^{1/2}-z\bbI)^{-1}\Lambda_2^{1/2}\tilde \bbU_2\bbX\bbX^T\bbu_1\non
&&\le O_p(\frac{1}{n^2}\bbu_1^T\bbX\bbX^T\bbU_{21}^T\bbU_{21}\bbX\bbX^T\bbu_1)+o_p(1).
\end{eqnarray}
 By the fact that the rank of $\bbU_{21}$ is $K-1$ and $\bbu_1^T\bbU_{21}^T=0$, it suffices to consider such  a term $\bbu_1^T\bbX\bbX^T\bbu_2\bbu_3^T\bbX\bbX^T\bbu_1$, where $\bbu_1^T\bbu_2=\bbu_1^T\bbu_3=0$, $\bbu_2^T\bbu_3=0$ or $1$. Since $\frac{1}{n^2}\mathbb{E}\bbu_1^T\bbX\bbX^T\bbu_2\bbu_3^T\bbX\bbX^T\bbu_1=O(n^{-1})$, we conclude that
 \begin{equation}\label{0503.2h}
 \frac{1}{n^2}\mathbb{E}|\bbu_1^T\bbX\bbX^T\bbU_{21}^T\bbU_{21}\bbX\bbX^T\bbu_1|=\frac{1}{n^2}\mathbb{E}\bbu_1^T\bbX\bbX^T\bbU_{21}^T\bbU_{21}\bbX\bbX^T\bbu_1=O(\frac{K^2}{n}).
 \end{equation}
 Combining (\ref{0502.2})-(\ref{0503.2h}), we get that $(\ref{0502.1})\sim (\lambda_1-z)^{-1}$ with probability tending to one. Noticing that (\ref{0503.2}) holds uniformly for $z\in \Gamma$, we have  $(\ref{0502.1})\sim (\lambda_1-z)^{-1}$ holds uniformly for $z\in \Gamma$.  i.e. with probability tending to one and for all $z\in \Gamma$, we have
 \begin{equation}\label{0227.1h}
 \bbv_1^T\xi_1\xi_1^T\bbv_1\rightarrow 1.
  \end{equation}

\end{proof}
\section{Proof of Corollary \ref{1130-1}}
Without loss of generality, we assume eigenvectors are real, otherwise we consider $\sum_{j=1}^p|v_{ij}|^4$. Since $\xi_i$ and $-\xi_i$ are regarded as the same eigenvectors in the eigenvector space, we always choose the direction such that $\bbv_1^T\xi_1\ge 0$. Therefore, by (\ref{0227.1h}) we have
$$\bbv_i^T\xi_i \stackrel{i.p.}{\longrightarrow} 1.$$
 By Theorem \ref{0503-1}, we have $\sum_{j=1}^p[v_{ij}-\xi_{ij}]^2=o_p(1)$, which implies that
$$\max_j|v_{ij}-\xi_{ij}|=o_p(1).$$
Therefore, we get
$$\sum_{j=1}^p|v_{ij}^4-\xi_{ij}^4|\le \sum_{j=1}^p(|v_{ij}|+|\xi_{ij}|)^3\max_j|v_{ij}-\xi_{ij}|=o_p(1).$$
This conclusion tells us that the sample eigenvector is a proper estimation of $\sum_{j=1}^p[v_{ij}^4]$.

\section{Proof of Theorem \ref{0606-2}}
Inspired by \cite{B14} and \cite{KY11} in this section we establish asymptotic distribution of the largest non-spiked eigenvalues of the sample covariance matrices $\frac{1}{n}\Gamma\bbX\bbX^{T}\Gamma^{T}$.  For simplicity and consistency with the papers such as \cite{BPZ2014a} and \cite{KY11}, we absorb $\frac{1}{\sqrt n}$ into $\bbX$ and consider the eigenvalues of the matrix $\Gamma\bbX\bbX^{T}\Gamma^{T}$ instead. That is to say, $var(\bbx_{ij})=\frac{1}{n}$ and $\mathbb{E}|\bbx_{ij}|^k\le \frac{c_k}{n^{k/2}}$. Without loss of generality, we assume that $\mu_{K+1}>1$. Correspondingly, $\nu_i$ is the i-th largest eigenvalue of $\bbX^T\bold\Sigma_1\bbX$ in this section. Let $\bbD(z)=z\bbI-\bbX^T\bbU^T_2\Lambda_P\bbU_2\bbX$. As the first step of the proof of Theorem \ref{0606-2}, by (\ref{0604.1}), we have the following Lemma
\begin{lem}\label{0609-1}
If $\lambda$ is not the eigenvalue of $\bbX^{T}\Sigma_1\bbX$, then $\lambda$ is the eigenvalue of $\Gamma\bbX\bbX^{T}\Gamma^{T}$ is equivalent to
$$\det(\bbI-\Lambda^{1/2}_S\bbU_1\bbX\bbD^{-1}(\lambda)\bbX^T\bbU^T_1\Lambda^{1/2}_S)=0.
$$
\end{lem}
In order to show the eigenvalue sticking, we need to prove the local law for
\begin{eqnarray}\label{1128.10}
\bbU_1\bbX\bbD^{-1}(z)\bbX^T\bbU_1^T,
\end{eqnarray}
where $\bbU_1\bbU^T_2=0$. First of all, we consider the special case $l=0$. To this end, we introduce the following linearization matrix
\begin{eqnarray}\label{1125.1}
&&\\
&&\bbH(z)\coloneqq \left(
  \begin{array}{ccc}
     z\bbI& \bbX^T\bbU_2^T\Lambda_P^{1/2}  &\bbX^T  \bbU_1^T\\
   \Lambda_P^{1/2}\bbU_2\bbX &\bbI& 0 \\
     \bbU_1  \bbX &0 & \bbI\\
  \end{array}
\right)^{-1}\non
&&=\left(
  \begin{array}{cc}
     \bbI& 0 \\
       0    &  (\bbU_2^T\Lambda_P^{1/2},\bbU_1^T)\\
  \end{array}
\right)^{-1}\left(
  \begin{array}{cc}
     z\bbI& \bbX^T \\
       \bbX     &  \tilde \Sigma\\
  \end{array}
\right)^{-1}\left(
  \begin{array}{cc}
     \bbI& 0 \\
       0     &  (\bbU_2^T\Lambda_P^{1/2},\bbU_1^T)^T \\
  \end{array}
\right)^{-1},\nonumber
\end{eqnarray}
where the last equality follows from the assumption that $L=0$ and $\tilde \Sigma=\left(
  \begin{array}{cc}
     \Lambda_P& 0 \\
       0     &  \bbI \\
  \end{array}
\right)$.
By simple calculation, it is easy to see that the lower right block of $\bbH(z)$ is equal to $(\bbI-\bbU_1\bbX\bbD^{-1}(z)\bbX^T\bbU_1^T)^{-1}$. We introduce a definition before giving the local law.
  \begin{deff}\label{1123.1}
Let
$$\xi=\{\xi^{(N)}(u):N\in \mathbb{N}, u\in U^{(N)}\}, \ \ \zeta=\{\zeta^{(N)}(u):N\in \mathbb{N}, u\in U^{(N)}\}$$
be two families of nonnegative random variables, where $U^{(N)}$ is a parameter set (can be either dependent on or independent of $N$). If for all small positive $\ep$ and $\sigma$, there exists a number $N(\ep,\sigma)$ only depending on $\ep$ and $\sigma$ such that
$$\sup_{u\in U^{(N)}}\mathbb{P}\left[|\xi^{(N)}(u)|>N^{\ep}|\zeta^{(N)}(u)|\right]\le N^{-\sigma}$$
for large enough $n\ge n(\ep,\sigma)$, then we say that $\zeta$ stochastically dominates $\xi$ uniformly in u. We denote this relationship by $\xi \prec \zeta$ or $\xi =O_{\prec}(\zeta)$. Moreover, if there exists a constant C such that $C^{-1}\le \frac{\xi}{\zeta}\le C$, then we say $\xi\sim \zeta$.
\end{deff}
By Theorem 3.7 of \cite{KY14}, we conclude that
\begin{equation}\label{0606.2}
\|(\bbI-\bbU_1\bbX\bbD^{-1}(z)\bbX^T\bbU_1^T)^{-1}-(\bbI+m_{\Sigma_1}(z))^{-1}\|_{\infty}\prec \sqrt{\frac{1}{n\kappa(z)}},
\end{equation}
where $m_{\Sigma_1}(z)$ is the unique solution of the following equation
\begin{eqnarray}\label{0331.2}
m_{\Sigma_1}(z)=-\frac{1}{z-\frac{1}{n}tr(\bbI+ m_{\Sigma_1}(z) \Sigma_1)^{-1}\Sigma_1}, \ \ z\in \mathbb{C}^+,
\end{eqnarray}
$\kappa(z)=|\Re z-\gamma_+|$, 
$n^{-2/3+5\ep}\le \Re z-\gamma_+\le 2\gamma_+$ and $\gamma_+$ is the rightmost end point of the density determined by $m_{\Sigma_1}(z)$.
Similarly, it follows from  Theorem 3.6 of \cite{KY14} that
\begin{eqnarray}\label{0606.2h}
\|(\bbI-\bbU_1\bbX\bbD^{-1}(z)\bbX^T\bbU_1^T)^{-1}-(\bbI+m_{\Sigma_1}(z))^{-1}\|_{\infty}\prec \Phi(z),
\end{eqnarray}
where  $\Phi(z)=\sqrt{\frac{\Im m_{\Sigma_1}(z)}{n\Im z}}+\frac{1}{n\Im z}$, $\Im z\ge n^{-2/3-\ep}$ and  $-c\le \Re z-\gamma_+\le n^{-2/3+5\ep}$ for some small constant c.
But this is not enough for the proof since $z$ is very large when we consider the spiked eigenvalues. We below prove a stronger version of (\ref{0606.2}) instead.

Before doing it, note that our objective is $\bbU_1\bbX\bbD^{-1}(z)\bbX^T\bbU_1^T$ instead of $ (\bbI-\bbU_1\bbX\bbD^{-1}(z)\bbX^T\bbU_1^T)^{-1}$ by (\ref{1128.10}). Therefore, we first need to  develop its upper bound from (\ref{0606.2}). By the formula that
$$\bbA^{-1}-\bbB^{-1}=-\bbA^{-1}(\bbA-\bbB)\bbB^{-1},$$
we have the following Neumann series
\begin{eqnarray}\label{0606.3}
&&\\
&&\bbU_1\bbX\bbD^{-1}(z)\bbX^T\bbU_1^T+m_{\Sigma_1}(z)\bbI=(\bbI+m_{\Sigma_1}(z)\bbI)-(\bbI-\bbU_1\bbX\bbD^{-1}(z)\bbX^T\bbU_1^T)\non
&&=\sum_{r=1}^{\infty}(-1)^{r+1}(1+m_{\Sigma_1}(z))^{r+1}\Delta^{r},\nonumber
\end{eqnarray}
where $\Delta=(\bbI-\bbU_1\bbX\bbD^{-1}(z)\bbX^T\bbU_1^T)^{-1}-(\bbI+m_{\Sigma_1}(z)\bbI)^{-1}$.  By (\ref{0606.2}), we know that $\|\Delta\|_{\infty}\prec \sqrt{\frac{1}{n\kappa}}$. Moreover, by the large deviation bound(see Lemma 3.4 of \cite{BPWZ2014b}, \cite{B14} or \cite{KY11}) we have
\begin{eqnarray}\label{1128.11}
\|m_{\Sigma_1}(z)\bbU_1\bbX\bbX^T\bbU_1^T-m_{\Sigma_1}(z)\bbI\|_{\infty}\prec \sqrt{\frac{1}{n}}.
\end{eqnarray}
 The expansion at the right hand side of (\ref{0606.3}) is ensured by the fact that $z$ is very close to or outside the support of $\bbX^T\bbU^T_2\Lambda_P\bbU_2\bbX$ and $\|\Delta\|\ll 1$. Together with the fact that $K\ll n^{1/6}\ll \sqrt{n\kappa}$, we conclude that
\begin{eqnarray}\label{0611.3}
&&\\
&&\|\bbU_1\bbX\bbD^{-1}(z)\bbX^T\bbU_1^T+m_{\Sigma_1}(z)\bbU_1^T\bbX\bbX^T\bbU_1\|_{\infty}\prec \sqrt{\frac{1}{n\kappa}}, n^{-2/3+5\ep}\le \Re z-\gamma_+\le 2\gamma_+.\nonumber
\end{eqnarray}

Up to now, we only show (\ref{0611.3}) holds for the case l=0. When $l \neq 0$,  we can find a $l\times (p+l)$ matrix $\bbU_3$ such that $\bbU_3\bbU_1^T=0$ and $\bbU_3\bbU_2^T=0$. Let $\tilde \bbU_1=(\bbU_1^T,\bbU_3^T)^T$. Since the dimension of $\tilde \bbU_1\bbX\bbD^{-1}(z)\bbX^T\tilde \bbU_1^T$ is $(l+K)\times (l+K)$ and $l+K\ll n^{1/6}$. Then by similar arguments from (\ref{0606.2}) to  (\ref{0611.3}) we have
\begin{eqnarray}\label{1128.12}
&&\\
&&\|\tilde \bbU_1\bbX\bbD^{-1}(z)\bbX^T\tilde \bbU_1^T+m_{\Sigma_1}(z)\tilde \bbU_1^T\bbX\bbX^T\tilde \bbU_1\|_{\infty}\prec \sqrt{\frac{1}{n\kappa}}, \  n^{-2/3+5\ep}\le \Re z-\gamma_+\le 2\gamma_+.\nonumber
\end{eqnarray}
This  implies that (\ref{0611.3}) also holds for the case $l\ll n^{1/6}$. Similarly, we also have
\begin{equation}\label{1128.12h}
\|\tilde \bbU_1\bbX\bbD^{-1}(z)\bbX^T\tilde \bbU_1^T+m_{\bold\Sigma_1}(z)\tilde \bbU_1^T\bbX\bbX^T\tilde \bbU_1\|_{\infty}\prec \Phi(z), \Im z\ge n^{-2/3-\ep} ,  -c\le \Re z-\gamma_+\le n^{-2/3+5\ep}.
\end{equation}
In the sequel
we  prove the local law when $z$ is far away from $\gamma_+$.
\begin{thm}\label{0610-1}
For all $\Im z\ge 0$, $\Re z=t\sim \varphi(n)$ and $\varphi(n)\rightarrow \infty$ when $n\rightarrow \infty$, we have
\begin{eqnarray}\label{0610.3}
\|\bbU_1\bbX\bbD^{-1}(z)\bbX^T\bbU_1^T+m_{\bold\Sigma_1}(z)\bbU_1\bbX\bbX^T\bbU_1^T\|_{\infty}\prec \frac{1}{\kappa(t)\sqrt n}.
\end{eqnarray}
\end{thm}
\begin{proof}
We prove
$$\bbu_1^T\bbX\bbD^{-1}(t)\bbX^T\bbu_1+m_{\bold\Sigma_1}(t)\bbu_1^T\bbX\bbX^T\bbu_1\prec \frac{1}{\kappa(t)\sqrt n}.$$
as an example. The other entries of (\ref{0610.3}) can be shown similarly.
Define
$$m^s(z)=-\bbu_1^T\bbX\bbD^{-1}(z)\bbX^T\bbu_1-m_{\bold\Sigma_1}(z)\bbu_1^T\bbX\bbX^T\bbu_1, \ z\in \mathbb{C}^+, \ \Re z\gg1,$$
and
\begin{eqnarray}\label{0610.4}
&&\\
F^s(x)=\sum_{i=1}^n\bbu_1^T\bbX\zeta_i\zeta_i^T\bbX^T\bbu_1I(\nu_i\le x)-F_0(x)(dx)\bbu_1^T\bbX\bbX^T\bbu_1,\nonumber
\end{eqnarray}
where $F_0(x)$ is the c.d.f. determined by $m_{\bold\Sigma_1}(z)$,  $\nu_i=\lambda_i(\bbX^T\bold\Sigma_1\bbX)$ and $\zeta_i$ is the corresponding eigenvector.
Hence, we have the steitjes transform
\begin{eqnarray}\label{0610.5}
m^s(z)=\int \frac{\rho^s(dx)}{x-z}, \ \Im z>0.
\end{eqnarray}
We next apply the Helffer-Sj\"{o}strand formula to the following function
$$f_z(x)=\frac{1}{x-z}.$$
 Let $\omega=x+yi\in \mathbb{C}$. Then define $\frac{\partial f(\omega)}{\partial \bar{\omega}}=\frac{\partial f(\omega)}{\partial x}+i\frac{\partial f(\omega)}{\partial y}$.
 In order to apply the Helffer-Sj\"{o}strand formula(referring to \cite{D95}), we need to look for a smooth version of $f_z(x)$, i.e. we define a smooth function $\chi(\omega)\in [0,1], \omega\in \mathbb{C}^+$ satisfying $\frac{\partial \chi(\omega)}{\partial \bar{\omega}}\le C$, where $C$ is a constant.  We choose a small constant $\omega'>0$  and require $\chi(\omega)=1$ for all $\omega$ belongs to $\omega'$-neighbourhood of $ [-1,\gamma_+]$ and $0$ outside the $2\omega'$-neighbourhood of $[-1,\gamma_+]$. By rigidity of the eigenvalues, i.e. $|\nu_1-\gamma_+|\prec n^{-2/3}$, we conclude that $supp \rho^s\subset (-2\omega',\gamma_++2\omega')$ with high probability. Therefore we can choose suitable z to be away from the support of $\bbX^T\bold\Sigma_1\bbX$, i.e. $z > \gamma_++3\omega'$. Then by the Helffer-Sj\"{o}strand formula, we have that for all $x\in supp \rho^s$,
 \begin{eqnarray}\label{0610.6}
 f_z(x)=\frac{1}{\pi}\int_{\mathbb{C}}\frac{\partial_{\bar{\omega}}(f_z(\omega)\chi(\omega))}{x-\omega}d\omega.
 \end{eqnarray}
 By the trivial fact that $\int \rho^s(dx)=0$, we have
  \begin{eqnarray}\label{0611.1}
 m^s(z)=\int\rho^s(dx) f_z(x)=\frac{1}{\pi}\int_{\mathbb{C}}f_z(\omega)\partial_{\bar{\omega}}(\chi(\omega)) m^s(\omega)d\omega,
 \end{eqnarray}
 where the second equality follows from the fact that $f_z(\omega)$ is analytic away from  $supp \rho^s$. By the definition of $\chi$, we have $\{\frac{\partial \chi}{\partial \bar\omega}\neq 0\}\subset \{\omega:dist[-1,\gamma_+]\in [\omega',2\omega']\}$ and on this interval we conclude that $|f_z(\omega)|\sim \kappa^{-1}(z)$. Moreover, following from (\ref{1128.12}), we have $m^s(\omega)\prec \frac{1}{\sqrt n}$ in the set $\{\frac{\partial \chi}{\partial \bar\omega}\neq 0\}$. Therefore we have
 $$ m^s(z)\prec \frac{1}{\sqrt n}\kappa^{-1}(z).$$
Up to now, we have shown that (\ref{0610.3}) holds when $\Im z>0$. To complete our proof, let $z=t+in^{-10}$. By the continuity of $m_{\bold\Sigma_1}(z)$ and $\bbX^T\bbU^T_2\Lambda_P\bbU_2\bbX-z\bbI$, it is easy to conclude (\ref{0610.3}).
\end{proof}
Immediately, we can get  Corollary \ref{0611-1} from Theorem \ref{0610-1}.
\begin{coro}\label{0611-1}
Under the conditions of Theorem \ref{0610-1} we have
  \begin{eqnarray}\label{0611.4}
  \|\bbU_1\bbX\bbD^{-1}(t)\bbX^T\bbU_1^T+m_{\bold\Sigma_1}(t)\bbI\|_{\infty}\prec \frac{1}{\kappa(t)\sqrt n}.
  \end{eqnarray}
\end{coro}
\begin{proof}
This corollary follows from Theorem \ref{0610-1} and the large deviation inequality that
$$\|m_{\bold\Sigma_1}(t)\bbU_1\bbX\bbX^T\bbU_1^T-m_{\bold\Sigma_1}(t)\bbI\|_{\infty}\prec \frac{1}{\kappa(t)\sqrt n}.$$
\end{proof}
By the singular value  inequality, we have the following Lemma.
\begin{lem}\label{0606-1}
$$\sigma_{K+i}(\Lambda^{1/2}\bbU\bbX)\le \sigma_i(\Lambda_P^{1/2}\bbU_2\bbX), \  i=1,2,..., p-K,$$
where $\sigma_j(.)$ represents  the $j$-th largest singular value.
\end{lem}
In view of Lemma \ref{0606-1}, there are at most $K$ spiked eigenvalues. Moreover, we need the eigenvalues of $\bbX^T\bold\Sigma_1\bbX$ to be distinct. To this end, we assume that the entries of $\bbX$ are all absolutely continues. Otherwise we consider the matrix $\bbX+e^{-n}\bbY$ instead, where $\bbY$ is a $(p+l)\times n$ matrix consisting of i.i.d. standard normal random variables. It is easy to see that such a perturbation doesn't change the desired spectral properties and then the eigenvalues of $(\bbX+e^{-n}\bbY)^T\bold\Sigma_1(\bbX+e^{-n}\bbY)$ are all distinct almost surely.

In the sequel, we assume that the following events hold and all Lemmas below are based on these events:

1. All eigenvalues of $\bbX^T\bold\Sigma_1\bbX$ are distinct.

2. For all $\alpha=1,2,...,n$, we have $\bbU_1\bbX\zeta_{\alpha}\neq 0$, where $\zeta_{\alpha}$ is the eigenvector of $\bbX^T\bold\Sigma_1\bbX$ corresponding to the $\alpha$-th largest eigenvalue.

3. The rigidity result associated with $\bbX^T\bold\Sigma_1\bbX$ holds for $\ep/2$ for all $\nu_i\ge \gamma_+-n^{-2/3+5\ep}$, for example $|\nu_1-\gamma_+|\le n^{-2/3+\ep/2}$ and
\begin{eqnarray}\label{0611.5}
 \|\bbU_1\bbX\bbD^{-1}(z)\bbX^T\bbU_1^T+m_{\bold\Sigma_1}(z)\bbI\|_{\infty}\le \frac{n^{\ep/2}}{\kappa(z)\sqrt n}, \ \Re z\gg 1.
\end{eqnarray}
Here Claims 1 and 2 hold by the absolutely continuous of the entries of $\bbX$. Claim 3 is guaranteed by Corollary \ref{0611-1} and \cite{BPZ2014a}, \cite{KY14}.
 In the sequel, define the intervals
$$I_i=[\mu_i-\mu_iKn^{-1/2+2\ep}, \mu_i+\mu_iKn^{-1/2+2\ep}], \ i=1,...,K.$$
$$I_0=[\gamma_+-n^{-2/3+2\ep}, \gamma_++n^{-2/3+2\ep}].$$
$$\Gamma(\bbd)=\bigcup_{i=0}^KI_{i}.$$
The following proposition is to prove that $\Gamma(\bbd)$ is the permission area for the spiked eigenvalues and the extremal bulk eigenvalues.

\begin{prop}\label{0613-1}
Under Assumptions \ref{0829-1} or \ref{0829-1h},  the following holds:
$$I_i\bigcap I_0=\emptyset, \  i=1,...,K,$$
and
\begin{eqnarray}\label{0609.2}
\sigma(\Gamma\bbX\bbX^{T}\Gamma^T)\bigcap[\gamma_+-n^{-2/3+2\ep},\infty)\subset\Gamma(\bbd),
\end{eqnarray}
where $\sigma(\Gamma\bbX\bbX^{T}\Gamma^T)$ represents the set of the eigenvalues of $\Gamma\bbX\bbX^{T}\Gamma^T$.
\end{prop}
\begin{proof}[Proof of Proposition \ref{0613-1}]
First of all, it is trivial to get $I_i\bigcap I_0=\emptyset, \  i=1,...,K$ by the definition of $I_i$. Therefore  it suffices to show (\ref{0609.2}). We define a $K\times K$ matrix $\bbM(t)$ with its entries  being
\begin{eqnarray}\label{0609.3}
\bbM_{ij}(t)=(\bbU_1\bbX\bbD^{-1}(t)\bbX^T\bbU_1^T)_{ij}-\delta_{ij}\mu_i^{-1}.
\end{eqnarray}
By Lemma \ref{0609-1}, we conclude that $t\in \sigma(\Gamma\bbX\bbX^{T}\Gamma^T)/\sigma(\Sigma_1^{1/2}\bbX\bbX^{T}\Sigma_1^{1/2})$ if and only if $\bbM(t)$ is singular. Therefore we focus on the value $t\nsubseteq \sigma(\Sigma_1^{1/2}\bbX\bbX^{T}\Sigma_1^{1/2})$. First we consider the case when  $t\ge \gamma_++n^{-2/3+2\ep}$. By Corollary \ref{0611-1} we have $\bbM(t)=-m_{\bold\Sigma_1}(t)\bbI-\Lambda_{\bbS}^{-1}+O(\frac{n^{\ep/2}}{\kappa(t)\sqrt{n}})$, where $\bbA=O(1)$ means $\|\bbA\|_{\infty}=O(1)$. On the other hand, for all $t\in [\log\mu_K,\infty]\setminus\bigcup_{i=1}^KI_{i}$, by $m_{\bold\Sigma_1}(t)=-\frac{1}{t}(1+o(1))$ we have
$$\min_{k}\{|m_{\bold\Sigma_1}(t)\bbI+\mu_k^{-1}|, k=1,...,K\}\ge \frac{Kn^{\ep}}{\kappa(t)\sqrt{n}}.$$
Therefore any $t\in[\log(\nu_K),\infty]\setminus\bigcup_{i=1}^KI_{i}$ is not an eigenvalue of $\Gamma\bbX\bbX^T\Gamma^T$ with high probability.
Moreover, by Weyl's inequality, we have
$$|\sigma_i(\Lambda^{1/2}\bbU\bbX)-\sigma_i(\Lambda_S^{1/2}\bbU_1\bbX)|\le \sigma_1(\Lambda_P^{1/2}\bbU_2\bbX)\sim 1,$$
which implies that the first K eigenvalues of $\Gamma\bbX\bbX^{T}\Gamma^T$ do not belong to $[\gamma_++n^{-2/3+2\ep},\log \mu_K]$ with high probability by the fact that $\sigma_K(\Lambda_S^{1/2}\bbU_1\bbX)\ge \sqrt \mu_K|(1-\sqrt{\frac{K}{n}})|\gg \log \mu_K$. Also, by Lemma \ref{0606-1}, we conclude that the ($K+1$)-th eigenvalue of $\Gamma\bbX\bbX^{T}\Gamma^T$ is smaller than $\gamma_++n^{-2/3+2\ep}$ with high probability. Therefore, together with Lemma \ref{0606-1}, $[\gamma_++n^{-2/3+2\ep},\log \nu_K]$ is a forbidden area of the eigenvalues of $\Gamma\bbX\bbX^{T}\Gamma^T$.
\end{proof}

\begin{prop}\label{0613-2}
Under Assumption \ref{0829-1}, for large enough n, each interval $I_i$, $i=1,...,K$ contains exactly one eigenvalue of $\Gamma\bbX\bbX^{T}\Gamma^T$.
\end{prop}
\begin{proof}
We choose a positive oriented contours $\mathcal{C}=\bigcup_{i=1}^K\mathcal{C}_i\subset \mathbb{C}\setminus [\gamma_-,\gamma_+]$ such that each contour $\mathcal{C}_i$ encloses $d_i$ but no other points of $\mu_j$, $j\neq i$. Moreover, the radius of each contour enclosing $\mu_i$ is of the same order of $\mu_i$. By Assumption \ref{0829-1c}, such contours exist. In view of Proposition \ref{0613-1}, it suffices to prove that there exists exactly one eigenvalue of $\Gamma\bbX\bbX^{T}\Gamma^T$ in each contour. Recalling that $\bbM(z)$ in (\ref{0609.3}), we define the following two functions
$$F_n(z)=\det(\bbM(z)), \  \ f_n(z)=\det(m_{\bold\Sigma_1}(z)\bbI+\Lambda_{\bbS}^{-1}).$$
By the definition of $\mathcal{C}$, the functions $F_n$ and $f_n$ are holomorphic in $\mathcal{C}$. Furthermore, the construction of $\mathcal{C}_i$ ensures that each $\mathcal{C}_i$ contains exactly one root of $f_n(z)=0$. For instance, we look at the first contour $\mathcal{C}_1$ containing $\mu_1$. For any $z\in \mathcal{C}_1, \Im z\neq 0$, it is easy to see that  $\Im f_n(z)\neq 0$. If $z\in \mathcal{C}_1, \Im z=0$, then $m_{\bold\Sigma_1}(z)$ is an increasing function of z. Combining with the fact that $m_{\bold\Sigma_1}(z)\bbI+\Lambda_{\bbS}^{-1}$ is a diagonal matrix, we conclude that there is only one root of $f_n(z)=0$ in $\mathcal{C}_1$.
By (\ref{0611.5}) and Lebniz's formula for the determinant, it is not hard to see that
$$|f_n(z)-F_n(z)|\le \frac{K^2n^{\ep/2}}{\sqrt n}\min_{z\in \partial \mathcal{C}_i}|f_n(z)|,$$
which implies that $F_n(z)$ also contains exactly one root of $F_n(z)=0$ in $\mathcal{C}_i$ by Rouch\'e's theorem. Notice that this arguments hold uniformly for $i=1,..., K$, by Proposition \ref{0613-1} and $I_i\subset \mathcal{C}_i$. We finish our proof.
\end{proof}
Similarly, we have
 \begin{prop}\label{0613-2h}
Under Assumption \ref{0829-1h}, for large enough n, each interval $\bigcup_{j=m_i+1}^{m_i+n_i}I_j=I_{m_i+1}$, $i=0,...,\mathcal{L}$ contains exactly $n_i$ eigenvalue of $\Gamma\bbX\bbX^{T}\Gamma^T$.
\end{prop}
Assume that $\Gamma\bbX\bbX^{T}\Gamma^T$ and $\Sigma_1^{1/2}\bbX\bbX^T\Sigma_1^{1/2}$ do not have the same eigenvalue. Before  considering the phase transition, we show the following delocalization result, which is used in the eigenvalue counting arguments.

\begin{lem}\label{0615-1}
Assume that $\zeta_i$ is the eigenvector of $(\bbX^T\bbU^T_2\Lambda_P\bbU_2\bbX-t\bbI)^{-1}$ corresponding to the eigenvalue $\nu_i\ge
\gamma_+-n^{-2/3+5\ep}$ for a sufficiently small constant $\ep$. We have
$$\bbe_k^T\bbU_1^T\bbX\zeta_i\prec \frac{1}{\sqrt n}.$$
\end{lem}
\begin{proof}
By (\ref{0611.3}) with $z=v_i+in^{-1+\iota}$, $0<\iota$, we have
\begin{equation}\label{0611.3h}
\bbe_k^T\bbU_1\bbX\bbD^{-1}(z)\bbX^T\bbU_1^T\bbe_k+m_{\bold\Sigma_1}(z)\bbe_k^T\bbU_1^T\bbX\bbX^T\bbU_1\bbe_k\prec \sqrt{\frac{1}{n\kappa}}\le n^{-1/8}.
\end{equation}
Therefore, with high probability $\bbe_k^T\bbU_1\bbX\bbD^{-1}(z)\bbX^T\bbU_1^T\bbe_k=O(1)$. Moreover,
\begin{eqnarray}\label{0611.4h}
&&-\Im \bbe_k^T\bbU_1\bbX\bbD^{-1}(z)\bbX^T\bbU_1^T\bbe_k=n^{-1+\iota}\sum_j\frac{\bbe_k^T\bbU_1\bbX\zeta_j\zeta_j^T\bbX^T\bbU_1^T\bbe_k}{|\nu_j-z|^2}\\
&&\ge n^{-1+\iota}\frac{\bbe_k^T\bbU_1\bbX\zeta_i\zeta_i^T\bbX^T\bbU_1^T\bbe_k}{|\nu_i-z|^2}=\frac{(\bbe_k^T\bbU_1^T\bbX\zeta_i)^2}{n^{-1+\iota}}.\nonumber
\end{eqnarray}
Since $\iota$ can be arbitrary small, the proof of this Lemma is complete.
\end{proof}
\subsection{The Non-spiked eigenvalues}
Considering the non-spiked eigenvalues, we prove the following area is forbidden for the eigenvalues of $\Gamma\bbX\bbX^T\Gamma^T$.
\begin{equation}\label{0612.1}
t\in[\gamma_+-n^{-2/3+2\ep},\gamma_++n^{-2/3+2\ep}],  \ \ dist(t,\sigma(\Sigma_1^{1/2}\bbX\bbX^T\Sigma_1^{1/2}))\ge n^{-2/3-2\ep}.
\end{equation}
Similar to the arguments of Proposition \ref{0613-1}, we aim at showing that for $t$ satisfying (\ref{0612.1}), $\bbM(t)$ is non singular. Choosing $\eta=n^{-2/3-\ep}$ and $z=t+i\eta$, we have
\begin{eqnarray}\label{0612.2}
&&|\left[\bbU_1\bbX\bbD^{-1}(t)\bbX^T\bbU_1^T-\bbU_1^T\bbX\bbD^{-1}(z)\bbX^T\bbU_1^T\right]_{ij}|\\
&&\le \sum_{\alpha}\frac{|\langle\bbX^T\bbU_1^T\bbe_i,\zeta_{\alpha}\rangle|^2+|\langle\bbX^T\bbU_1^T\bbe_j,\zeta_{\alpha}\rangle|^2}{2}\left|\frac{1}{\lambda_{\alpha}-t}-\frac{1}{\lambda_{\alpha}-z}\right|\non
&&\le \sum_{\alpha}\frac{|\langle\bbX^T\bbU_1^T\bbe_i,\zeta_{\alpha}\rangle|^2+|\langle\bbX^T\bbU_1^T\bbe_j,\zeta_{\alpha}\rangle|^2}{2}\frac{\eta}{\eta^2+(\lambda_{\alpha}-t)^2}\non
&&=-\Im (\bbU_1\bbX\bbD^{-1}(z)\bbX^T\bbU_1^T)_{ii}-\Im (\bbU_1\bbX\bbD^{-1}(z)\bbX^T\bbU_1^T)_{jj}\nonumber,
\end{eqnarray}
where $\zeta_{\alpha}$ is the eigenvector of $\bbX^T\bold\Sigma_1\bbX$ corresponding to the $\alpha$-th largest eigenvalue.
Therefore, by local law we have
\begin{eqnarray}\label{1128.13}
&&\\
&&\bbM(t)=-m_{\bold\Sigma_1}(z)\bbI-\Lambda_{\bbS}^{-1}+O(n^{\ep/2}\Im m_{\bold\Sigma_1}(z)+\frac{n^{\ep/2}}{n\eta})=-m_{\bold\Sigma_1}(z)\bbI-\Lambda_{\bbS}^{-1}+O(n^{-1/3+2\ep}).\nonumber
\end{eqnarray}
Since $|m_{\bold\Sigma_1}(z)|\sim 1$, we have $|m_{\bold\Sigma_1}(z)+\mu_i^{-1}|\sim 1$, $i=1,..., K$ uniformly. Therefore, it is easy to see that $\bbM(t)$ is non singular for
$t$ satisfying (\ref{0612.1}). Up to now, we are ready to prove Theorem \ref{0606-2}. 

 Actually, once the tools and results including Lemma \ref{0609-1}--(\ref{0612.1}) are available, the the proof of Theorem \ref{0606-2} is almost the same as the proof of Proposition 6.8 in \cite{KY11}. The only difference is that we only prove that the eigenvalues are sticking with the order $n^{-2/3-\ep}$ instead of $n^{-1+\ep}$, which is caused by allowing K to tend to infinity. Hence we ignore the details.


The detailed proof  is similar to Proposition 6.8 in \cite{KY11} and thus we omit it.
%

\end{document}